\documentclass[a4paper,11pt]{elsarticle}
\usepackage[latin1]{inputenc}
\usepackage[english]{babel}
\usepackage[T1]{fontenc}
\usepackage{amsfonts,amssymb,amsthm,amsmath,amsbsy}
\usepackage{mathrsfs}
\usepackage{bm}
\usepackage{fullpage}
\usepackage{indentfirst}
\usepackage{graphicx}
\usepackage[all]{xy}

\makeatletter
\def\ps@pprintTitle{%
   \let\@oddhead\@empty
   \let\@evenhead\@empty
   \let\@oddfoot\@empty
   \let\@evenfoot\@oddfoot
}
\makeatother

\newtheorem{teor}{Theorem}[section]
\newtheorem*{teor*}{Theorem}
\newtheorem{lemma}[teor]{Lemma}
\newtheorem{prop}[teor]{Proposition}
\newtheorem{corol}[teor]{Corollary}

\theoremstyle{definition}
\newtheorem{defin}[teor]{Definition}
\theoremstyle{remark}
\newtheorem{rmk}[teor]{Remark}
\newtheorem*{ex}{Example}

\newcommand{\Z}{\mathbb{Z}}
\newcommand{\N}{\mathbb{N}}
\newcommand{\C}{\mathbb{C}}
\newcommand{\End}{\mathrm{End}}
\newcommand{\Aut}{\mathrm{Aut}}
\newcommand{\Mod}{\mathrm{Mod}}
\newcommand{\Hom}{\mathrm{Hom}}
\newcommand{\res}{\mathrm{res}}
\newcommand{\ind}{\mathrm{ind}}
\newcommand{\inv}{\mathrm{inv}}
\newcommand{\infl}{\mathrm{infl}}
\newcommand{\ent}{\mathcal{O}}
\newcommand{\bs}{\backslash}
\newcommand{\frack}{\mathfrak{k}}

\newcommand{\M}{\mathcal{M}}
\renewcommand{\P}{\mathcal{P}}
\newcommand{\U}{\mathcal{U}}
\newcommand{\V}{\mathcal{V}}

\numberwithin{equation}{section}
 
\begin{document}

\begin{frontmatter}
\title{Blocks of the category of smooth $\ell$-modular representations of 
$GL(n,F)$ and its inner forms: reduction to level-$0$}
\author{Gianmarco Chinello}
\ead{gianmarco.chinello[at]unimib.it}
\address{Universit\`a degli Studi di Milano-Bicocca\\ Dipartimento di Matematica e Applicazioni\\ via Cozzi 55, 20125 Milano (Italy)}

\begin{abstract}
\noindent 
Let $G$ be an inner form of a general linear group over a non-archimedean locally compact field of residue characteristic $p$, let $R$ be an algebraically closed field of characteristic different from $p$ and let $\mathscr{R}_R(G)$ be the 
category of smooth representations of $G$ over $R$.
In this paper, we prove that a block (indecomposable summand) of $\mathscr{R}_R(G)$ is equivalent to a level-$0$ block (a block in which every object has non-zero invariant vectors for the pro-$p$-radical of a maximal compact open subgroup) of $\mathscr{R}_R(G')$, where $G'$ is a direct product of groups of the same type of $G$.
\end{abstract}

\begin{keyword}
Blocks \sep Type theory \sep Semisimple types \sep Equivalence of categories \sep Hecke algebras  \sep Modular representations of p-adic reductive groups \sep Level-$0$ representations.

\MSC[2010]{20C20, 22E50}
\end{keyword}
\end{frontmatter}

\section*{Introduction}
Let $F$ be a non-archimedean locally compact field of residue characteristic $p$ and let $D$ be a central division algebra of finite dimension over $F$ whose reduced degree is denoted by $d$.
Given $m\in \N^*$, we consider the group $G=GL_{m}(D)$ which is an inner form of $GL_{md}(F)$. 
Let $R$ be an algebraically closed field of characteristic $\ell\neq p$ and let $\mathscr{R}_R(G)$ be the category of smooth representations of $G$ over $R$, that are called $\ell$-modular when $\ell$ is positive.
In this paper, we are interested on the Bernstein decomposition of $\mathscr{R}_R(G)$ (see \cite{SeSt1} or \cite{Vig1} for $d=1$) that is its decomposition as a direct sum of full indecomposable subcategories, called \emph{blocks}.
Actually a full understanding of blocks of $\mathscr{R}_R(G)$ is equivalent to a full understanding of the whole category.

\smallskip
The main purpose of this paper is to find an equivalence of categories between any block of $\mathscr{R}_R(G)$ and a level-$0$ block of $\mathscr{R}_R(G')$ where $G'$ is a suitable direct product of inner forms of general linear groups over finite extensions of $F$. 
We recall that a level-$0$ block of $\mathscr{R}_R(G')$ is a block in which every object has non-zero invariant vectors for the pro-$p$-radical of a maximal compact open subgroup of $G'$.
This result is an important step in the attempt to describe blocks of $\mathscr{R}_R(G)$ because it reduces the problem to the description of level-$0$ blocks.

\smallskip
In the case of complex representations, Bernstein \cite{Bern} found a block decomposition of $\mathscr{R}_\C(G)$ indexed by pairs $(M,\sigma)$ where $M$ is a Levi subgroup of $G$ and $\sigma$ is an irreducible cuspidal representation of $M$, up to a certain equivalence relation called \emph{inertial equivalence}. 
In particular an irreducible representation $\pi$ of $G$ is in the block associated to the inertial class of $(M,\sigma)$ if its cuspidal support is in this class.
In \cite{BK2}, Bushnell and Kutzko introduce a method to descibe the blocks of $\mathscr{R}_\C(G)$: the \emph{theory of type}. This method consists in 
associating at every block of $\mathscr{R}_\C(G)$ a pair $(J,\lambda)$, called type, where $J$ is a compact open subgroup of $G$ and $\lambda$ is an irreducible representation of $J$, such that 
the simple objects of the block are the irreducible 
subquotients of the compactly induced representation $\ind_J^G(\lambda)$. 
In this case the block is equivalent to the category of modules over the $\C$-algebra $\mathscr{H}_\C(G,\lambda)$ of $G$-endomorphisms of $\ind_J^G(\lambda)$.
In \cite{SeSt3} (see \cite{BK3} for $d=1$) S\'echerre and Stevens describe explicitly this algebra as a tensor product of algebras of type A.

\smallskip
In the case of $\ell$-modular representations, in \cite{SeSt1} S\'echerre and Stevens (see \cite{Vig1} for $d=1$) found a block decomposition of $\mathscr{R}_R(G)$ indexed by inertial classes of pairs $(M,\sigma)$ where $M$ is a Levi subgroup of $G$ and $\sigma$ is an irreducible supercuspidal representation of $M$. 
In particular an irreducible representation $\pi$ of $G$ is in the block associated to the inertial class of $(M,\sigma)$ if its supercuspidal support is in this class.
We recall that the notions of cuspidal and supercuspidal representation are not equivalent as in complex case; however, in \cite{MS2} Minguez and S\'echerre prove the uniqueness of supercuspidal support, up to conjugation, for every irreducible representation of $G$.
We remark that to obtain the block decomposition of $\mathscr{R}_R(G)$, S\'echerre and Stevens do not use the same method as Bernstein, but they rely, like us in this paper, on the theory of semisimple types developed in \cite{SeSt3} (see \cite{BK3} for $d=1$).
Actually, they associate at every block of $\mathscr{R}_R(G)$ a pair $(\mathbf{J}, \bm\lambda)$, called \emph{semisimple supertype}. 
Unfortunately the construction of the equivalence, as in complex case, between the block and the category of modules over $\mathscr{H}_R(G,\bm\lambda)$ does not hold and one of the problems that occurs is that the pro-order of $\mathbf{J}$ can be divisible by $\ell$. 
Some partial results on descriptions of algebras which are Morita equivalent to blocks of $\mathscr{R}_R(GL_n(F))$ are given by Dat 
\cite{Dat3}, Helm \cite{Helm} and Guiraud \cite{Gui}.

\smallskip
The idea of this paper is the following.
We fix a block $\mathscr{R}(\mathbf{J},\bm\lambda)$ of $\mathscr{R}_R(G)$ associated to the semisimple supertype $(\mathbf{J}, \bm\lambda)$ and, as in \cite{SeSt1}, we can associate to it 
a compact open subgroup $\mathbf{J}_{max}$ of $G$, its pro-$p$-radical $\mathbf{J}^1_{max}$ and 
an irreducible representation $\bm\eta_{max}$ of $\mathbf{J}^1_{max}$. 
We remark that we can extend, not in a unique way, $\bm\eta_{max}$ to an irreducible representation of $\mathbf{J}_{max}$. 
Thus, we denote $\mathscr{R}(G,\bm\eta_{max})$ the direct sum of blocks of $\mathscr{R}_R(G)$ associated to $(\mathbf{J}^1_{max},\bm\eta_{max})$ and we consider the functor $$\mathbf{M}_{\bm\eta_{max}}=\Hom_G(\ind_{\mathbf{J}^1_{max}}^G\bm\eta_{max}, -):\mathscr{R}(G,\bm\eta_{max})\longrightarrow \Mod-\mathscr{H}_R(G,\bm\eta_{max})$$
where $\mathscr{H}_R(G,\bm\eta_{max})\cong\End_G\big(\ind_{\mathbf{J}^1_{max}}^G(\bm\eta_{max})\big)$. 
Using the fact that $\bm\eta_{max}$ is a projective representation, since $\mathbf{J}^1_{max}$ is a pro-$p$-group, we prove that $\mathbf{M}_{\bm\eta_{max}}$ is an equivalence of categories (theorem \ref{thm:equivalenceM}). 
This result generalizes corollary 3.3 of \cite{Chin1} where $\bm\eta_{max}$ is a trivial character.
We can also associate to $(\mathbf{J}, \bm\lambda)$ a Levi subgroup $L$ of $G$ and a group $B_L^{\times}$, which is a direct product of inner forms of general linear groups over finite extensions of $F$ and which we have denoted $G'$ above. 
If $K_L$ is a maximal compact open subgroup of $B_L^{\times}$ and $K_L^1$ is its pro-$p$-radical then $K_L/K_L^1\cong \mathbf{J}_{max}/\mathbf{J}^1_{max}=\mathscr{G}$ is a direct product of finite general linear groups.
Actually, in \cite{Chin1} is proved that the $K_L^1$-inviariant functor $\inv_{K_L^1}$ is an equivalence of categories between the level-$0$ subcategory $\mathscr{R}(B_L^{\times},K_L^1)$ of $\mathscr{R}_R(B_L^{\times})$, which is the direct sum of its level-$0$ blocks, and the category of modules over the algebra $\mathscr{H}_R(B_L^{\times},K_L^1)\cong\End_{B_L^{\times}}\big(\ind_{K^1_L}^{B_L^{\times}}1_{K^1_L}\big)$.
Now, thanks to the explicit presentation by generators and relations of $\mathscr{H}_R(B_L^{\times},K_L^1)$, presented in \cite{Chin1}, in this paper we  construct a homomorphism 
$\Theta_{\gamma,\bm\kappa_{max}}:\mathscr{H}_R(B_L^{\times},K_L^1)\longrightarrow\mathscr{H}_R(G,\bm\eta_{max})$, depending on the choice of the extension $\bm\kappa_{max}$ of $\bm\eta_{max}$ to $\mathbf{J}_{max}$ and on the choice of an intertwining element $\gamma$ of $\bm\eta_{max}$, finding elements in $\mathscr{H}_R(G,\bm\eta_{max})$ satisfying all relations defining $\mathscr{H}_R(B_L^{\times},K_L^1)$. 
Moreover, using some properties of $\bm\eta_{max}$, we prove that this homomorphism is actually an isomorphism. 
We remark that finding this isomorphism is one of the most difficult results obtained in this article and the proof in the case $L=G$ takes about half of the paper (section \ref{sec:isom}).
This also complete the results contained in the Phd thesis \cite{Chin} of the author because in it the construction of this isomorphism depends on a conjecture (see section 3.4 of \cite{Chin}). 
In this way we obtain an equivalence of categories 
$\mathbf{F}_{\gamma,\bm\kappa_{max}}:\mathscr{R}(G,\bm\eta_{max})\longrightarrow \mathscr{R}(B_L^{\times},K_L^1)$
such that the following diagram commutes
\begin{equation*}
\xymatrix{
\mathscr{R}(G,\bm\eta_{max})
\ar[rr]^{\mathbf{F}_{\gamma,\bm\kappa_{max}}}
\ar[d]_{\mathbf{M}_{\bm\eta_{max}}}^{\rotatebox{270}{$\simeq$}}
& &\mathscr{R}(B_L^{\times},K_L^1)
\ar[d]^{\inv_{K^1_L}}_{\rotatebox{90}{$\simeq$}}
\\
\Mod-\mathscr{H}_R(G,\bm\eta_{max})
\ar[rr]^{\Theta_{\gamma,\bm\kappa_{max}}^*}
& &\Mod-\mathscr{H}_R(B_L^{\times},K_L^1).
}
\end{equation*}
Then we obtain
\begin{equation*}
\mathbf{F}_{\gamma,\bm\kappa_{max}}(\pi,V)=\mathbf{M}_{\bm\eta_{max}}(\pi,V)\otimes_{\mathscr{H}_R(B^{\times}_L,K^1_L)}\ind_{K^1_L}^{B_L^{\times}}(1_{K^1_L})
\end{equation*}
for every $(\pi,V)$ in $\mathscr{R}(G,\bm\eta_{max})$, 
where the action of $\mathscr{H}_R(B^{\times}_L,K^1_L)$ on $\mathbf{M}_{\bm\eta_{max}}(\pi,V)$ depends on $\Theta_{\gamma,\bm\kappa_{max}}$.
Hence, $\mathbf{F}_{\gamma,\bm\kappa_{max}}$ induces an equivalence of categories between the block $\mathscr{R}(\mathbf{J},\bm\lambda)$ and  a level-$0$ block of $\mathscr{R}_R(B_L^{\times})$.
To understand this correspondence we need to use the functor $$\textbf{\textsf{K}}_{\bm\kappa_{max}}:\mathscr{R}(G,\bm\eta_{max})\longrightarrow \mathscr{R}_R(\mathbf{J}_{max}/\mathbf{J}^1_{max})=\mathscr{R}_R(\mathscr{G})$$
where $\mathbf{J}_{max}$ acts on $\textbf{\textsf{K}}_{\bm\kappa_{max}}(\pi)=\Hom_{\mathbf{J}^1_{max}}(\bm\eta_{max},\pi)$ 
by $x.\varphi=\pi(x)\circ \varphi\circ\bm\kappa_{max}(x)^{-1}$ for every representation $\pi$ of $G$, $\varphi\in \Hom_{\mathbf{J}^1_{max}}(\bm\eta_{max},\pi)$ and $x\in\mathbf{J}_{max}$. 
This functor is strongly used in \cite{SeSt1} to define $\mathscr{R}(\mathbf{J},\bm\lambda)$ and to prove the Bernstein decomposition of $\mathscr{R}_R(G)$.
We also consider the functor 
$\textbf{\textsf{K}}_{K_L}:\mathscr{R}(B^{\times}_L,K^1_L)\rightarrow \mathscr{R}_R(K_L/K_L^1)=\mathscr{R}_R(\mathscr{G})$
given by $\textbf{\textsf{K}}_{K_L}(Z)=Z^{K^1_L}$ for every representation $(\varrho,Z)$ of $B^{\times}_L$
where $x\in K_L$ acts on $z\in Z^{K^1_L}$ by $x.z=\varrho(x)z$.
Then the functors $\textbf{\textsf{K}}_{K_L}\circ \mathbf{F}_{\gamma,\bm\kappa_{max}}$ and $\textbf{\textsf{K}}_{\bm\kappa_{max}}$ are naturally isomorphic (proposition \ref{thm:naturalisom}) and so $\mathscr{R}(\mathbf{J},\bm\lambda)$ is equivalent to the level-$0$ block $\mathscr{B}$ of $\mathscr{R}_R(B_L^{\times})$ such that 
$\textbf{\textsf{K}}_{\bm\kappa_{max}}(\mathscr{R}(\mathbf{J},\bm\lambda))=\textbf{\textsf{K}}_{K_L}(\mathscr{B})$.
More precisely, if $\mathbf{J}^1$ is the pro-$p$-radical of $\mathbf{J}$, then $\mathbf{J}/\mathbf{J}^1=\mathscr{M}$ is a Levi subgroup of $\mathscr{G}$ and the choice of $\bm\kappa_{max}$ defines a decomposition $\bm\lambda=\bm\kappa\otimes\bm\sigma$ where $\bm\kappa$ is an irreducible representation of $\mathbf{J}$ and $\bm\sigma$ is a cuspidal representation of $\mathscr{M}$ viewed as an irreducible representation of $\mathbf{J}$ trivial on $\mathbf{J}^1$. 
If we can consider the pair $(\mathscr{M},\bm\sigma)$ up to the equivalence relation given in definition 1.14 of \cite{SeSt1}, then a representation $(\varrho,Z)$ of $B_L^{\times}$ is in $\mathscr{B}$ if it is generated by the maximal subspace of $Z^{K^1_L}$ such that every irreducible subquotient has supercuspidal support in the class of $(\mathscr{M},\bm\sigma)$.

\smallskip
One question we do not address in this paper is the structure of level-$0$ blocks of $\mathscr{R}_R(B_L^{\times})$ when the characteristic of $R$ is positive. 
Thanks to results of \cite{Chin1} we know that there is a correspondence between these  blocks and the set $\mathscr{E}$ of primitive central idempotents of $\mathscr{H}_R(B_L^{\times},K^1_L)$, which are descibed in sections 2.5 and 2.6 of \cite{Chin}. 
Hence, one possibility for understanding level-$0$ blocks of $\mathscr{R}_R(B_L^{\times})$ is to describe the algebras $e\mathscr{H}_R(B_L^{\times},K^1_L)$ with $e\in\mathscr{E}$.
On the other hand, we recall that in \cite{Dat4} Dat proves that every level-$0$ block of $\mathscr{R}_R(GL_n(F))$ is equivalent to the unipotent block of $\mathscr{R}_R(G'')$ where $G''$ is a suitable product of general linear groups over non-archimedean locally compact fields. 
Hence, putting together the result of Dat and results of this article, we obtain a method to reduce the description of any block of $\mathscr{R}_R(GL_n(F))$ to that of an unipotent block.
Unfortunately the description of the unipotent block of $\mathscr{R}_R(GL_n(F))$, or of $\mathscr{R}_R(G)$, is nowadays an hard question and it has no answer yet.

\smallskip
We now give a brief summary of the contents of each section of this paper. 
In section \ref{sec:preliminaries} we present general results on the convolution Hecke algebras $\mathscr{H}_R(\mathtt{G},\sigma)$ where $\mathtt{G}$ is generic locally profinite group and $\sigma$ a representation of an open subgroup $\mathtt{H}$ of $\mathtt{G}$. 
We see that if $\sigma$ is finitely generated then $\mathscr{H}_R(\mathtt{G},\sigma)$ is isomorphic to the endomorphism algebra of $\ind_{\mathtt{H}}^{\mathtt{G}}\sigma$.
We also define two subcategories of $\mathscr{R}_R(\mathtt{G})$ and we prove that, when they coincide, they are equivalent to the category of modules over $\mathscr{H}_R(\mathtt{G},\sigma)$.
In section \ref{sec:maximaltypes} we introduce the theory of maximal simple types; in particular we consider the Heisenberg representation $\eta$ associated to a simple character (see paragraph \ref{subsec:Heisenberg}) and we define the groups $B^{\times}=B_G^{\times}$ and $K^1=K^1_G$.
In section \ref{sec:isom} we prove that the algebras $\mathscr{H}_R(G,\eta)$ and $\mathscr{H}_R(B^{\times},K^1)$ are isomorphic. 
In section \ref{sec:semisimpletypes} we introduce the theory of semisimple types, we define the representation $\bm\eta_{max}$ and the group $B_L^{\times}$ and we prove that the algebras $\mathscr{H}_R(B_L^{\times},K_L^1)$ and $\mathscr{H}_R(G,\bm\eta_{max})$ are isomorphic.
Finally, in section \ref{sec:categoryequivalence} we
prove that $\mathbf{M}_{\bm\eta_{max}}$ and $\mathbf{F}_{\gamma,\bm\kappa_{max}}$ are equivalences of categories, we describe the correspondence between blocks of $\mathscr{R}(G,\bm\eta_{max})$ and of $\mathscr{R}(B_L^{\times},K_L^1)$ and we investigate on the dependence of these results on the choice of the extension of $\bm\eta_{max}$ to $\mathbf{J}_{max}$. 

\section{Preliminaries}\label{sec:preliminaries}
This section is written in much more generality than the remainder of this paper. We present general results  for a generic locally profinite group.

\smallskip
Let $\mathtt{G}$ be a locally profinite group (i.e. a locally compact and totally disconnected topological group) and let $R$ be a unitary commutative ring. We recall that a representation $(\pi,V)$ of $\mathtt{G}$ over $R$ is smooth if for every $v\in V$ the stabilizer $\{g\in\mathtt{G}\,|\, \pi(g)v=v\}$  is an open subgroup of $\mathtt{G}$. 
We denote $\mathscr{R}_R(\mathtt{G})$ the (abelian) category of smooth representations of $\mathtt{G}$ over $R$.
From now on all representations are considered smooth.

\subsection{Hecke algebras for a locally profinite group}\label{subsec:Hecke}

In this paragraph we introduce an algebra associated to a representation $\sigma$ of a subgroup of $\mathtt{G}$ and we prove that it is isomorphic to the endomorphism algebra of the compact induction of $\sigma$. 
This definition generalizes those in section 1 of \cite{Chin1} that corresponds to the case in which $\sigma$ is trivial.

\smallskip
Let $\mathtt{H}$ be an open subgroup of $\mathtt{G}$ such that every $\mathtt{H}$-double coset is a finite union of left $\mathtt{H}$-cosets (or equivalently $\mathtt{H}\cap g\mathtt{H}g^{-1}$ is of finite index in $\mathtt{H}$ for every $g\in \mathtt{G}$) and let $(\sigma, V_\sigma)$ be a smooth representation of $\mathtt{H}$ over $R$. 

\begin{defin}
Let $\mathscr{H}_R(\mathtt{G},\sigma)$ be the $R$-algebra of functions $\Phi:\mathtt{G}\rightarrow \End_R(V_\sigma)$ such that 
$\Phi(hgh')=\sigma(h)\circ\Phi(g)\circ\sigma(h')$ for every $h,h'\in \mathtt{H}$ and $g\in \mathtt{G}$ and whose supports are a finite union of $\mathtt{H}$-double cosets, 
endowed with convolution product
\begin{equation}\label{eq:convolution}
(\Phi_1*\Phi_2)(g)=\sum_{x}\Phi_1(x)\Phi_2(x^{-1}g)
\end{equation}
where $x$ describes a system of representatives of  
$\mathtt{G}/\mathtt{H}$ in $\mathtt{G}$. 
This algebra is unitary and the identity element is 
$\sigma$ seen as a function on $\mathtt{G}$ with support equal to 
$\mathtt{H}$.
To simplify the notation, from now on we denote 
$\Phi_1\Phi_1=\Phi_1*\Phi_2$ for all $\Phi_1,\Phi_2\in\mathscr{H}_R(\mathtt{G},\sigma)$.
\end{defin}

We observe that the sum in (\ref{eq:convolution}) is finite since the support of $\Phi_1$ is a finite union of $\mathtt{H}$-double cosets and 
by hypothesis, every $\mathtt{H}$-double coset is a finite union of left $\mathtt{H}$-cosets. 
Moreover, (\ref{eq:convolution}) is well-defined because for every $h\in\mathtt{H}$ and $x,g\in\mathtt{G}$ we have
$\Phi_1(xh)\Phi_2((xh)^{-1}g)=
\Phi_1(x)\circ\sigma(h)\circ\sigma(h^{-1})\circ\Phi_2(x^{-1}g)=\Phi_1(x)\circ\Phi_2(x^{-1}g).$

\smallskip
For every $g\in\mathtt{G}$ we denote by
$\mathscr{H}_R(\mathtt{G},\sigma)_{\mathtt{H}g\mathtt{H}}$
the submodule of $\mathscr{H}_R(\mathtt{G},\sigma)$ of functions with support in $\mathtt{H}g\mathtt{H}$.
If $g_1,g_2\in\mathtt{G}$, $\Phi_1\in \mathscr{H}_R(\mathtt{G},\sigma)_{\mathtt{H}g_1\mathtt{H}}$ and 
$\Phi_2\in \mathscr{H}_R(\mathtt{G},\sigma)_{\mathtt{H}g_2\mathtt{H}}$ then the support of $\Phi_1\Phi_2$ is in 
$\mathtt{H}g_1\mathtt{H}g_2\mathtt{H}$ and the support of $x\mapsto \Phi_1(x)\Phi_2(x^{-1}g)$ is in $\mathtt{H}g_1\mathtt{H}\cap g\mathtt{H}g_2^{-1}\mathtt{H}$.
\begin{rmk}\label{rmk:normal}
If $g_1$ or $g_2$ normalizes $\mathtt{H}$ then the support of $\Phi_1\Phi_2$ is in $\mathtt{H}g_1g_2\mathtt{H}$ and the support of $x\mapsto \Phi_1(x)\Phi_2(x^{-1}g_1g_2)$ is in $g_1\mathtt{H}$.
Hence, we obtain 
$(\Phi_1\Phi_2)(g_1g_2)=\Phi_1(g_1)\circ\Phi_2(g_2)$.
\end{rmk}

For every $g\in \mathtt{G}$ we denote $\mathtt{H}^g=g^{-1}\mathtt{H}g$ and 
$(\sigma^g,V_\sigma)$ the representation of $\mathtt{H}^g$ given by $\sigma^g(x)=\sigma(gxg^{-1})$ for every 
$x\in \mathtt{H}^g$. 
We denote $I_g(\sigma)$ the $R$-module 
$\Hom_{\mathtt{H}\cap \mathtt{H}^g}(\sigma,\sigma^g)$ and  $I_\mathtt{G}(\sigma)$ the set, called \emph{intertwining} of $\sigma$ in $\mathtt{G}$, of $g\in \mathtt{G}$ such that $I_g(\sigma)\neq 0$. 
For every $g\in I_\mathtt{G}(\sigma)$ the map 
$\Phi\mapsto\Phi(g)$ is an isomorphism of $R$-modules between $\mathscr{H}_R(\mathtt{G},\sigma)_{\mathtt{H}g\mathtt{H}}$ and $I_g(\sigma)$ and so $g\in \mathtt{G}$ intertwines $\sigma$ if and only if there exists an element $\Phi\in\mathscr{H}_R(\mathtt{G},\sigma)$ such that $\Phi(g)\neq 0$.

\smallskip
Let $\ind_{\mathtt{H}}^{\mathtt{G}}(\sigma)$ be the compact induced representation of $\sigma$ to $\mathtt{G}$. 
It is the $R$-module of functions 
$f:\mathtt{G}\rightarrow V_\sigma$, compactly supported modulo $\mathtt{H}$, such that $f(hg)=\sigma(h)f(g)$ for every $h\in\mathtt{H}$ and $g\in\mathtt{G}$ endowed with the action of $\mathtt{G}$ defined by $x.f:g\mapsto f(gx)$ for every $x,g\in \mathtt{G}$ and $f\in \ind_{\mathtt{H}}^{\mathtt{G}}(\sigma)$. 
We remark that, since $\mathtt{H}$ is open, by I.5.2(b) of \cite{Vig2} it is a  smooth representation of $\mathtt{G}$. 
For every $v\in V_\sigma$ let $i_v\in\ind_{\mathtt{H}}^{\mathtt{G}}(\sigma)$ with support in $\mathtt{H}$ defined by $i_v(h)=\sigma(h)v$ for every $h\in \mathtt{H}$. Then for every $x\in \mathtt{G}$ the function $x^{-1}.i_v$ has support $\mathtt{H}x$ and takes the value $v$ on $x$. Hence, for every $f\in \ind_{\mathtt{H}}^{\mathtt{G}}(\sigma)$ we have 
\begin{equation}\label{eq:induced}
f=\sum_{x\in\mathtt{H}\bs\mathtt{G}}x^{-1}.i_{f(x)}
\end{equation}
and so the image $i_{V_\sigma}$ of $v\mapsto i_v$  generates $\ind_\mathtt{H}^\mathtt{G}(\sigma)$ as representation of $\mathtt{G}$.

\smallskip
Frobenius reciprocity (I.5.7 of \cite{Vig2}) states that the map 
$\Hom_{\mathtt{H}}(\sigma,V)\rightarrow \Hom_{\mathtt{G}}(\ind_\mathtt{H}^\mathtt{G}(\sigma),V)$ given by $\phi\mapsto\psi$ where
$\phi(v)=\psi(i_v)$ for every $v\in V_{\sigma}$ is an isomorphism of $R$-modules.

\begin{lemma}\label{lemma:isomalgebras}
If $V_\sigma$ is a finitely generated $R$-module, the map $\xi:\mathscr{H}_R(\mathtt{G},\sigma) \rightarrow  \End_{\mathtt{G}}(\ind_\mathtt{H}^\mathtt{G}(\sigma))$ given by
\begin{equation*}
\xi(\Phi)(f)(g)=(\Phi*f)(g)=\sum_{x\in \mathtt{G}/\mathtt{H}}\Phi(x)f(x^{-1}g)
\end{equation*}
for every $\Phi\in\mathscr{H}_R(\mathtt{G},\sigma)$, $f\in\ind_\mathtt{H}^\mathtt{G}(\sigma)$ and $g\in\mathtt{G}$ is an $R$-algebra isomorphism whose inverse is given by $\xi^{-1}(\vartheta)(g)(v)=\vartheta(i_v)(g)$ 
for every $\vartheta\in\End_{\mathtt{G}}\left(\ind_\mathtt{H}^\mathtt{G}(\sigma)\right)$, $g\in\mathtt{G}$ and $v\in V_\sigma$.
\end{lemma}

\begin{proof}
See I.8.5-6 of \cite{Vig2}.
\end{proof}

\subsection{The categories $\mathscr{R}_{\sigma}(\mathtt{G})$ and $\mathscr{R}(\mathtt{G},\sigma)$}\label{subsec:categories}
In this paragraph we associate to an irreducible projective representation of a compact open subgroup of $\mathtt{G}$ two subcategories of $\mathscr{R}_R(\mathtt{G})$.

\smallskip
Let $\mathtt{K}$ be a compact open subgroup of $\mathtt{G}$ and $(\sigma,V_{\sigma})$ be an irreducible projective representation of $\mathtt{K}$ such that $V_\sigma$ is a finitely generated $R$-module. 
Then $\rho=\ind_{\mathtt{K}}^\mathtt{G}(\sigma)$ is a projective representation of $\mathtt{G}$ by I.5.9(d) of \cite{Vig2} and so the functor  
$$\mathbf{M}_\sigma=\Hom_\mathtt{G}(\rho,-):\mathscr{R}_R(\mathtt{G})\rightarrow \Mod-\mathscr{H}_R(\mathtt{G},\sigma)$$
is exact.
We remark that for every representation $(\pi,V)$ of $\mathtt{G}$ the right action of $\Phi\in\mathscr{H}_R(\mathtt{G},\sigma)$ on $\varphi\in\Hom_\mathtt{G}(\rho,V)$ is given by $\varphi.\Phi=\varphi\circ\xi(\Phi)$ where $\xi$ is the isomorphism of lemma \ref{lemma:isomalgebras}. 
Moreover if $V_1$ and $V_2$ are representations of $\mathtt{G}$ and $\mathfrak{\epsilon}\in\Hom_\mathtt{G}(V_1,V_2)$ then $\mathbf{M}_\sigma(\phi)$ maps $\varphi$ to $\phi\circ\varphi$ for every $\varphi\in \Hom_\mathtt{G}(\rho,V_1)$.

\begin{defin}
Let $\mathscr{R}_{\sigma}(\mathtt{G})$ be the full subcategory of $\mathscr{R}_R(\mathtt{G})$ whose objects are representations $V$ such that $\mathbf{M}_\sigma(V')\neq 0$ for every irreducible subquotient $V'$ of $V$. 
\end{defin}

For every representation $V$ of $\mathtt{G}$ we denote
$V^{\sigma}=\sum_{\phi\in\Hom_{\mathtt{K}}(\sigma,V)}\phi(\sigma)$ which is a subrepresentation of the restriction of $V$ to $\mathtt{K}$. We denote by $V[\sigma]$ the representation of $\mathtt{G}$ generated by $V^{\sigma}$.
If $\sigma$ is the trivial character of $\mathtt{K}$ then
$V^{\sigma}=V^{\mathtt{K}}=\{v\in V\,|\,\pi(k)v=v \text{ pour tout }k\in\mathtt{K}\}$ 
is the set of $\mathtt{K}$-invariant vectors of $V$. 

\begin{prop}\label{prop:functorM}
For every representation $V$ of $\mathtt{G}$ we have $V[\sigma]=\sum_{\psi\in \mathbf{M}_\sigma(V)}\psi(\rho)$ and so $\mathbf{M}_\sigma(V)=\mathbf{M}_\sigma(V[\sigma]).$
Moreover, if $W$ is a subrepresentation of $V$ then $\mathbf{M}_\sigma(W)=\mathbf{M}_\sigma(V)$ if and only if $W[\sigma]=V[\sigma]$.
\end{prop}

\begin{proof}
By Frobenius reciprocity we have 
$\Hom_{\mathtt{K}}(\sigma,V)\cong \mathbf{M}_\sigma(V)$ and so using (\ref{eq:induced}) we obtain 
\[
V[\sigma]=
\sum_{g\in\mathtt{G}}\pi(g)\sum_{\psi\in \mathbf{M}_\sigma(V)}\psi(i_{V_{\sigma}})
= \sum_{\psi\in \mathbf{M}_\sigma(V)}\psi\left( \sum_{g\in\mathtt{G}}g. i_{V_{\sigma}}\right)
=\sum_{\psi\in \mathbf{M}_\sigma(V)}\psi(\rho)
\]
that implies $\mathbf{M}_\sigma(V)=\mathbf{M}_\sigma(V[\sigma])$. 
Furthermore, if $W[\sigma]=V[\sigma]$ then $\mathbf{M}_\sigma(W)= \mathbf{M}_\sigma(V)$ and if
$\mathbf{M}_\sigma(W)=\mathbf{M}_\sigma(V)$ then 
$W[\sigma]=\sum_{\psi\in \mathbf{M}_\sigma(W)}\psi(\rho)=\sum_{\psi\in \mathbf{M}_\sigma(V)}\psi(\rho)=V[\sigma]$.
\end{proof}

\begin{defin}
Let $\mathscr{R}(\mathtt{G},\sigma)$ be the full subcategory of $\mathscr{R}_R(\mathtt{G})$ whose objects are representations $V$ such that $V=V[\sigma]$. 
If $\sigma$ is the trivial character of $\mathtt{K}$ we denote $\mathscr{R}(\mathtt{G},\mathtt{K})$ the subcategory of representations $V$ generated by $V^{\mathtt{K}}$. 
\end{defin}

\begin{prop}\label{prop:TFAE}
Let $V$ be a representation of $\mathtt{G}$. The following conditions are equivalent:
\begin{enumerate}[(i)]
\item for every irreducible subquotient $U$ of $V$ we have $\mathbf{M}_\sigma(U)\neq 0$;
\item for every non-zero subquotient $W$ of $V$ we have $\mathbf{M}_\sigma(W)\neq 0$;
\item for every subquotient $Z$ of $V$ we have $Z=Z[\sigma]$;
\item for every subrepresentation $Z$ of $V$ we have $Z=Z[\sigma]$.
\end{enumerate}
\end{prop}

\begin{proof} $(i)\Rightarrow (ii)$: let $W$ be a non-zero subquotient of $V$ and $W_1\subset W_2$ two subrepresentations of $W$ such that $U=W_2/W_1$ is irreducible. By (i) we have $\mathbf{M}_\sigma(U)\neq 0$ which implies $\mathbf{M}_\sigma(W_2)\neq 0$ and so $\mathbf{M}_\sigma(W)\neq 0$.
$(ii)\Rightarrow(iii)$: let $Z$ be a subquotient of $V$. 
By proposition \ref{prop:functorM} we have $\mathbf{M}_\sigma(Z)=\mathbf{M}_\sigma(Z[\sigma])$ and so $\mathbf{M}_\sigma(Z/Z[\sigma])=0$. Hence, by (ii) we obtain $Z=Z[\sigma]$. $(iv)\Rightarrow(i)$: let $U$ be an irreducible subquotient of $V$ and $Z_1\subsetneq Z_2$ be two subrepresentations of $V$ such that $U=Z_2/Z_1$. 
By (iv) we have $Z_1[\sigma]=Z_1\neq Z_2=Z_2[\sigma]$ and by proposition \ref{prop:functorM} we have $\mathbf{M}_\sigma(Z_1)\neq \mathbf{M}_\sigma(Z_2)$. Hence, we obtain  $\mathbf{M}_\sigma(U)\neq 0$.
\end{proof}

\begin{rmk}\label{rmk:categoryinclusion}
Proposition \ref{prop:TFAE} implies that $\mathscr{R}_{\sigma}(\mathtt{G})$ is contained in $\mathscr{R}(\mathtt{G},\sigma).$
\end{rmk}

\subsection{Equivalence of categories}
In this paragraph we suppose that there exists a compact open subgroup $\mathtt{K}_0$ of $\mathtt{G}$ whose pro-order is invertible in $R^\times$ and we consider the Haar measure $\mathtt{dg}$ of $\mathtt{G}$ with values in $R$ such that $\int_\mathtt{K_0}\mathtt{dg}=1$ (see I.2 of \cite{Vig2}). We prove that if the two categories introduced in paragraph \ref{subsec:categories} are equal then they are equivalent to the category of modules over the algebra introduced in paragraph \ref{subsec:Hecke}.

\smallskip
The \emph{global Hecke algebra} $\mathscr{H}_R(\mathtt{G})$ of $\mathtt{G}$ is the $R$-algebra of locally constant and compactly supported functions $f:\mathtt{G}\rightarrow R$ endowed with convolution product given by $(f_1*f_2)(x)=\int_{\mathtt{G}} f_1(g)f_2(g^{-1}x)\mathtt{dg}$
for every $f_1,f_2\in\mathscr{H}_R(\mathtt{G})$ and $x\in \mathtt{G}$ (see ... of \cite{Vig2}).
In general $\mathscr{H}_R(\mathtt{G})$ is not unitary but it has enough idempotents by I.3.2 of \cite{Vig2}. The categories $\mathscr{R}_R(\mathtt{G})$ and $\mathscr{H}_R(\mathtt{G})-\Mod$ are equivalent by I.4.4 of \cite{Vig2} and we have
$\ind_\mathtt{H}^\mathtt{G}(\tau)=\mathscr{H}_R(\mathtt{G})\otimes_{\mathscr{H}_R(\mathtt{H})} V_\tau$ for every representation $(\tau,V_\tau)$ of an open subgroup $\mathtt{H}$ of $\mathtt{G}$ by I.5.2 of \cite{Vig2}.

\smallskip
Let $\mathtt{K}$ be a compact open subgroup of $\mathtt{G}$, let $(\sigma,V_{\sigma})$ be an irreducible projective representation of $\mathtt{K}$ as in paragraph \ref{subsec:categories} and let $\rho=\ind_{\mathtt{K}}^{\mathtt{G}}(\sigma)$. 
Since $V_\sigma$ is a simple projective module over the unitary algebra $\mathscr{H}_R(\mathtt{K})$, it is isomorphic to a direct summand of $\mathscr{H}_R(\mathtt{K})$ itself because any non-zero map $\mathscr{H}_R(\mathtt{K})\rightarrow V_{\sigma}$ is surjective and splits. 
Then it is isomorphic to a minimal ideal of $\mathscr{H}_R(\mathtt{K})$ and so there exists an idempotent $e$ of $\mathscr{H}_R(\mathtt{K})$ such that  $V_{\sigma}=\mathscr{H}_R(\mathtt{K})e$. 
Hence, we obtain $\rho=\mathscr{H}_R(\mathtt{G})e$  because the map $\sum_i (f_i\otimes h_ie)\mapsto \big(\sum_if_ih_i\big)e$ is an isomorphism of $\mathscr{H}_R(\mathtt{G})$-modules between $\mathscr{H}_R(\mathtt{G})\otimes_{\mathscr{H}_R(\mathtt{K})}\mathscr{H}_R(\mathtt{K})e$ and $\mathscr{H}_R(\mathtt{G})e$ whose inverse is $fe\mapsto fe\otimes e$.

\smallskip
The algebra $\mathscr{H}_R(\mathtt{G},\sigma)$ is isomorphic to $\End_{\mathtt{G}}(\rho)\cong\End_{\mathscr{H}_R(\mathtt{G})}(\mathscr{H}_R(\mathtt{G})e)$ by lemma \ref{lemma:isomalgebras} and the map $e\mathscr{H}_R(\mathtt{G})e\rightarrow \big(\End_{\mathscr{H}_R(\mathtt{G})}(\mathscr{H}_R(\mathtt{G})e)\big)^{op}$ which maps $efe\in e\mathscr{H}_R(\mathtt{G})e$ to the endomorphism $f'e\mapsto f'efe$ of $\mathscr{H}_R(\mathtt{G})e$ is an algebra isomorphism whose inverse is $\varphi\mapsto \varphi(e)$. Then we have $\mathscr{H}_R(\mathtt{G},\sigma)^{op}\cong e\mathscr{H}_R(\mathtt{G})e$ and so the categories $e\mathscr{H}_R(\mathtt{G})e-\Mod$ and $\Mod-\mathscr{H}_R(\mathtt{G},\sigma)$ are equivalent.

\begin{teor}\label{thm:equivalence}
If $\mathscr{R}_{\sigma}(\mathtt{G})=\mathscr{R}(\mathtt{G},\sigma)$ then $V\mapsto \mathbf{M}_\sigma(V)$ is an equivalence of categories  between $\mathscr{R}(\mathtt{G},\sigma)$ and $\Mod-\mathscr{H}_R(\mathtt{G},\sigma)$ whose quasi-inverse is $W\mapsto W\otimes_{\mathscr{H}_R(\mathtt{G},\sigma)}\rho$.
\end{teor}

\begin{proof}
We take $A=\mathscr{H}_R(\mathtt{G})$ and $\mathscr{H}_R(\mathtt{G})e=\rho$ in I.6.6 of \cite{Vig2}. 
Since $\mathscr{H}_R(\mathtt{G},\sigma)^{op}\cong e\mathscr{H}_R(\mathtt{G})e$, left actions of $e\mathscr{H}_R(\mathtt{G})e$ become right actions of $\mathscr{H}_R(\mathtt{G},\sigma)$.
The functor $V\mapsto eV$ of \cite{Vig2}
from $\mathscr{H}_R(\mathtt{G})-\Mod$ to $e\mathscr{H}_R(\mathtt{G})e-\Mod$ becomes the functor 
$V\mapsto \Hom_{\mathscr{H}_R(\mathtt{G})}(\mathscr{H}_R(\mathtt{G})e,V)$ and so the functor $\mathbf{M}_\sigma$. 
Hypothesis of theorem \emph{"\'equivalence de cat\'egories"} in I.6.6 of \cite{Vig2} are satisfied by the condition 
$\mathscr{R}_{\sigma}(\mathtt{G})=\mathscr{R}(\mathtt{G},\sigma)$ and so we obtain the result. 
\end{proof}

\section{Maximal simple types}\label{sec:maximaltypes}
In this section we introduce the theory of simple types  of an inner form of a general linear group over a non-archimedean locally compact field in the case of modular representations.
We refer to sections 2.1-5 of \cite{MS} for more details. 

\smallskip
Let $p$ be a prime number. 
Let $F$ be a non-archimedean locally compact field of residue characteristic $p$ and let $D$ be a central division algebra of finite dimension over $F$ whose reduced degree is denoted by $d$. 
Given a positive integer $m$, we consider the ring $A=M_m(D)$ and the group $G=GL_m(D)$ which is an inner form of $GL_{md}(F)$. 
Let $R$ be an algebraically closed field of characteristic different from $p$. 

\smallskip
Let $\Lambda$ be an $\ent_D$-lattice sequence of $V=D^m$. 
It defines a hereditary $\ent_F$-order 
$\mathfrak{A}=\mathfrak{A}(\Lambda)$ of $A$ whose radical is denoted by $\mathfrak{P}$, a compact open subgroup $U(\Lambda)=U_0(\Lambda)=\mathfrak{A}(\Lambda)^{\times}$ of $G$ and a filtration $U_k(\Lambda)=1+\mathfrak{P}^k$ with $k\geq 1$ of 
$U(\Lambda)$ (see section 1 of \cite{SecI}).
Let $[\Lambda,n,0,\beta]$ be a simple stratum of $A$ (see for instance section 1.6 of \cite{SeSt2}).
Then $\beta\in A$ and the $F$-subalgebra $F[\beta]$ of $A$ generated by $\beta$ is a field denoted by $E$. The centralizer $B$ of $E$ in $A$ is a simple central $E$-algebra and $\mathfrak{B}=\mathfrak{A}\cap B$ is a hereditary $\ent_E$-order of $B$ whose radical is 
$\mathfrak{Q}=\mathfrak{P}\cap B$.

\smallskip
As in paragraphs 1.2 and 1.3 of \cite{SecIII} we can choose 
a simple right $E\otimes_F D$-module $N$
such that the functor $V\mapsto \Hom_{E\otimes_F D}(N,V)$ defines a Morita equivalence between the category of modules over $E\otimes_F D$ and the category of vector spaces over $D'=\End_{E\otimes_F D}(N)^{op}$ which is a central division
algebra over $E$.
We denote $A(E)=\End_D(N)$ which is a central simple $F$-algebra.
If $d'$ is the reduced degree of $D'$ over $E$ and 
$m'$ is the dimension of 
$V'=\Hom_{E\otimes_F D}(N,V)$ over $D'$, then we have $m'd'=md/ [E:F]$. 
Fixing a basis of $V'$ over $D'$ we obtain, via the Morita equivalence above, an isomorphism $N^{m'}\cong V$ of $E\otimes_F D$-modules. 
If for every $i\in \{1,\dots,m'\}$ we denote by $V^i$ the image of the $i$-th copy of $N$ by this isomorphism, we obtain a decomposition 
$V=V^1\oplus\cdots\oplus V^{m'}$
into simple $E\otimes_F D$-submodules. 
By section 1.5 of \cite{SecIII} we can choose a basis $\mathscr{B}$ of $V'$ over $D'$ so that $\Lambda$ decomposes into the direct sum of the $\Lambda^i=\Lambda\cap V^i$ for $i\in\{1,\dots,m'\}$. 
For every $i\in \{1,\dots,m'\}$, let 
$\mathrm{e}_i:V\rightarrow V^i$ be the projection on $V^i$ with kernel $\bigoplus_{j\neq i}V^j$.
In accordance with paragraph 2.3.1 of \cite{SecI} (see also \cite{BH}) the family of idempotents 
$\mathrm{e}=(\mathrm{e}_1,\dots,\mathrm{e}_{m'})$ is a decomposition conforms to $\Lambda$ over $E$.

\smallskip
By paragraphs 1.4.8 and 1.5.2 of \cite{SecIII} there exists a unique hereditary order $\mathfrak{A}(E)$ normalized by $E^{\times}$ in $A(E)$ whose radical is denoted by $\mathfrak{P}(E)$. 
For every $i\in\{1,\dots,m'\}$ we have an isomorphism $\End_D(V^i)\cong A(E)$ of 
$F$-algebras which induces an isomorphism of $\ent_F$-algebras between the hereditary orders $\mathfrak{A}(\Lambda^i)$ and $\mathfrak{A}(E)$.
Moreover, to the choice of the basis $\mathscr{B}$ corresponds the isomorphisms 
$M_{m'}(D')\cong B$ of $E$-algebras and 
$M_{m'}(A(E))\cong A$ of $F$-algebras.

\begin{rmk}\label{rmk:maximal}
If $U(\Lambda)\cap B^{\times}$ is a maximal compact open subgroup of $B^{\times}$, these isomorphisms induce
an isomorphism $\mathfrak{B}\cong M_{m'}(\ent_{D'})$ 
of $\ent_E$-algebras and, by lemma 1.6 of \cite{SecII},
two isomorphisms $\mathfrak{A}\cong M_{m'}(\mathfrak{A}(E))$ and $\mathfrak{P}\cong M_{m'}(\mathfrak{P}(E))$ of 
$\ent_F$-algebras. 
\end{rmk}

We can associate to $[\Lambda,n,0,\beta]$ two compact open subgroups $J=J(\beta,\Lambda)$, $H=H(\beta,\Lambda)$ of  $U(\Lambda)$ (see 2.4 of \cite{SeSt2}). 
For every integer $k\geq 1$ we denote 
$J^k=J^k(\beta,\Lambda)=J(\beta,\Lambda)\cap U_k(\Lambda)$ and
$H^k=H^k(\beta,\Lambda)=H(\beta,\Lambda)\cap U_k(\Lambda)$ which are pro-$p$-groups. In particular $J^1$ and $H^1$ are normal pro-$p$-subgroups of $J$ and the quotient 
$J^1/H^1$ is a finite abelian $p$-group.
\begin{rmk}\label{rmk:isomJB}
We have $J=(U(\Lambda)\cap B^{\times})J^1$ and this induce a canonical group isomorphism 
$J/J^1\cong (U(\Lambda)\cap B^{\times})/(U_1(\Lambda)\cap B^{\times})$ (see paragraph 2.3 of \cite{MS}). 
It allows us to associate canonically and bijectively a representation of $J$ trivial on $J^1$ to a representation of $U(\Lambda)\cap B^{\times}$ trivial on $U_1(\Lambda)\cap B^{\times}$.
\end{rmk}

\subsection{Simple characters, Heisenberg representation and $\beta$-extensions} \label{subsec:Heisenberg}
Let $[\Lambda,n,0,\beta]$ be a simple stratum of $A$. 
We denote by $\mathscr{C}_R(\Lambda,0,\beta)$ the set of 
\emph{simple $R$-characters} (see paragraph 2.2 of \cite{MS} and \cite{SecI}) that is a finite set of $R$-characters of $H^1$ which depends on the choice of an additive $R$-character of $F$.
If $\widetilde m\in \N^*$ and 
$[\widetilde\Lambda,\widetilde n,0,\widetilde\beta]$ is a simple stratum of $M_{\widetilde{m}}(D)$ such that there exists an isomorphism of $F$-algebras 
$\nu:F[\beta]\rightarrow F[\widetilde\beta]$ with $\nu(\beta)=\widetilde\beta$, then there exists a bijection $\mathscr{C}_R(\Lambda,0,\beta)\rightarrow \mathscr{C}_R(\widetilde\Lambda,0,\widetilde\beta)$
canonically associated to $\nu$, called \emph{transfer map}.
There also exists an equivalence relation, called \emph{endo-equivalence}, among  
simple characters in $\mathscr{C}_R(\Lambda,0,\beta)$ (see \cite{BSS}) whose 
equivalence classes are called \emph{endo-classes}.

\smallskip
Let $\theta\in\mathscr{C}_R(\Lambda,0,\beta)$. 
By proposition 2.1 of \cite{MS} there exists a finite dimensional irreducible representation $\eta$ of $J^1$, unique up to  isomorphism, whose restriction to $H^1$ contains $\theta$.
It is called \emph{Heisenberg representation} associated to $\theta$. 
The intertwining of $\eta$ is $I_G(\eta)=J^1B^{\times}J^1=JB^{\times}J$ and for every $y\in B^{\times}$ the $R$-vector space $I_y(\eta)=\Hom_{J^1\cap (J^1)^y}(\eta,\eta^y)$ has dimension $1$.  

\smallskip
A $\beta$\emph{-extension} of $\eta$ (or of $\theta$) is an irreducible representation $\kappa$ of $J$ extending $\eta$ such that $I_G(\kappa)=JB^{\times}J$. 
By proposition 2.4 of \cite{MS}, every simple character $\theta\in\mathscr{C}_R(\Lambda,0,\beta)$ admits a $\beta$-extension
and by formula (2.2) of \cite{MS} the set 
of $\beta$-extensions of $\theta$ is equal to 
$$\mathcal{B}(\theta)=\{\kappa\otimes(\chi\circ N_{B/E})\,|\, \chi \text{ character of }\ent_{E}^{\times} \text{ trivial on } 1+\wp_E\}$$
where $N_{B/E}$ is the reduced norm of $B$ over $E$ and 
$\chi\circ N_{B/E}$ is seen as a character of $J$ trivial on $J^1$ thanks to remark \ref{rmk:isomJB}. 
We observe that for every $\kappa\in\mathcal{B}(\theta)$
and every $y\in B^\times$, the $R$-vector space  $I_y(\kappa)$ has dimension $1$ because it is non-zero and it is contained in $I_y(\eta)$.

\subsection{Maximal simple types}
Let $[\Lambda,n,0,\beta]$ be a simple stratum of $A$ such that $U(\Lambda)\cap B^{\times}$ is a maximal compact open subgroup of $B^{\times}$. 
By remarks \ref{rmk:maximal} and \ref{rmk:isomJB}, there exists a group isomorphism 
$J/J^1\cong GL_{m'}(\frack_{D'})$, 
which depends on the choice of $\mathscr{B}$.

\smallskip
A \emph{maximal simple type} of $G$ associated to 
$[\Lambda,n,0,\beta]$ is a pair $(J,\lambda)$ where $\lambda$ is an irreducible representation of $J$ of the form $\lambda=\kappa\otimes\sigma$ where $\kappa\in\mathcal{B}(\theta)$ with 
$\theta\in\mathscr{C}_R(\Lambda,0,\beta)$
and $\sigma$ is a cuspidal representation of 
$GL_{m'}(\frack_{D'})$ identified to an irreducible representation of $J$ trivial on $J^1$.
If $\sigma$ is a supercuspidal representation of 
$GL_{m'}(\frack_{D'})$ then $(J,\lambda)$ is called maximal simple \emph{supertype}.

\begin{rmk}
The choice of a $\beta$-extension $\kappa\in\mathcal{B}(\theta)$ determines the decomposition $\lambda=\kappa\otimes\sigma$.
If we choose another $\beta$-extension 
$\kappa'=\kappa\otimes(\chi\circ N_{B/E})\in\mathcal{B}(\theta)$ we obtain the decomposition 
$\lambda=\kappa'\otimes \sigma'$ where $\sigma'=\sigma\otimes(\chi^{-1}\circ N_{B/E})$.
\end{rmk}

\subsection{Covers} 
Let $\M$ be a Levi subgroup of $G$, $\P$ be a parabolic subgroup of $G$ with Levi component $\M$ and unipotent radical $\U$ and let $\U^-$ be the unipotent subgroup opposed to $\U$. 
We say that a compact open subgroup $K$ of $G$ is \emph{decomposed with respect to} $(\M,\P)$ if 
$K=(K\cap\U^-)(K\cap \M)(K\cap\U)$
and every element $k\in K$ decomposes uniquely as $k=k_1k_2k_3$ with $k_1\in K\cap\U^-$, $k_2\in K\cap \M$ and $k_3\in K\cap\U$. 
Furthermore, if $\pi$ is a representation of $K$ we say that the pair $(K, \pi)$ is \emph{decomposed with respect to} $(\M,\P)$ if $K$ is decomposed with respect to 
$(\M,\P)$ and if $K\cap\U$ and $K\cap\U^-$ are in the kernel of $\pi$.

\smallskip
Let $\M$ be a Levi subgroup of $G$. Let $K$ and $K_\M$ be two compact open subgroups of $G$ and $\M$ respectively and let $\varrho$ and $\varrho_\M$ be two irreducible representations of $K$ and $K_\M$ respectively.
We say that the pair $(K,\varrho)$ is \emph{decomposed above} 
$(K_\M,\varrho_\M)$ if 
$(K,\varrho)$ is decomposed with respect to $(\M,\P)$ for every parabolic subgroup $\P$ with Levi component $\M$, if 
$K\cap\M=K_\M$ and if the restriction of $\varrho$ to $K_\M$ is equal to $\varrho_\M$.
A pair $(K,\varrho)$ is a \emph{cover} of $(K_M,\varrho_M)$ 
if it is decomposed above $(K_M,\varrho_M)$ and it satisfies condition (0.3) of \cite{Blo}. 
For more details see \cite{Blo,Vig1}.

\section{The isomorphisms $\mathscr{H}_R(G,\eta)\cong\mathscr{H}_R(B^{\times},U_1(\Lambda)\cap B^{\times})$}\label{sec:isom}

Using notations of section \ref{sec:maximaltypes}, let $[\Lambda,n,0,\beta]$ be a simple stratum of $A$ 
such that $U(\Lambda)\cap B^{\times}$ is a maximal compact open subgroup of $B^{\times}$. 
Let $\theta\in\mathscr{C}_R(\Lambda,0,\beta)$ and let $\eta$ be the Heisenberg representation associated to $\theta$. 
In this section we want to prove that the algebras $\mathscr{H}_R(G,\eta)$ and 
$\mathscr{H}_R(B^{\times},U_1(\Lambda)\cap B^{\times})$ 
are isomorphic (theorem \ref{thm:isomHecke2}). 

\smallskip
Thanks to section \ref{sec:maximaltypes}, from now on we identify $A$ with 
$M_{m'}(A(E))$, $G$ with $GL_{m'}(A(E))$,
$U(\Lambda)$ with 
$GL_{m'}(\mathfrak{A}(E))$, $U_1(\Lambda)$ with 
$\mathbb{I}_{m'}+M_{m'}(\mathfrak{P}(E))$,  $B^{\times}$ with $GL_{m'}(D')$, $K_B=U(\Lambda)\cap B^{\times}$ with $GL_{m'}(\ent_{D'})$ and $K_B^1=U_1(\Lambda)\cap B^{\times}$ with $\mathbb{I}_{m'}+M_{m'}(\wp_{D'})$. 
By section 2.4 of \cite{Chin1} we know a presentation by generators and relations of the algebra 
$\mathscr{H}_R(B^{\times},K^1_B)\cong\mathscr{H}_\Z(B^{\times},K^1_B)\otimes_\Z R$.
Using this presentation we want to find an isomorphism between 
$\mathscr{H}_R(B^{\times},K^1_B)$ and 
$\mathscr{H}_R(G,\eta)$.

\subsection{Root system of $GL_{m'}$}
In this paragraph we recall some notations and results on the root system of $\mathbf{GL}_{m'}$ contained in section 2.1 of \cite{Chin1}.

\smallskip
We denote by $\bm\Phi=\{\alpha_{ij}\;|\;1\leq i\neq j\leq m'\}$ the set of roots of $\mathbf{GL}_{m'}$ relative to torus of diagonal matrices.
Let $\bm\Phi^+=\{\alpha_{ij}\;|\;1\leq i<j\leq m'\}$,
$\bm\Phi^-=-\bm\Phi^+=\{\alpha_{ij}\;|\;1\leq j<i\leq m'\}$ and 
$\Sigma=\{\alpha_{i,i+1}\;|\;1\leq i\leq m'-1\}$ 
be, respectively, the sets of positive, negative and simple roots relative to Borel subgroup of upper triangular matrices. 
For every $\alpha=\alpha_{i,i+1}\in\Sigma$ we write $s_{\alpha}$ or $s_i$ for the transposition $(i,i+1)$. 
Let $W$ be the group generated by the $s_i$ which is the group of permutations of $m'$ elements and so the Weyl group of 
$\mathbf{GL}_{m'}$.
Let $\ell:W\rightarrow \N$ be the length function of $W$ relative to $s_1,\dots,s_{m'-1}$.
The group $W$ acts on $\bm\Phi$ by $w\alpha_{ij}=\alpha_{w(i)w(j)}$ and 
for every $w\in W$ and $\alpha\in\Sigma$ 
we have (see (2.2) of \cite{Chin1})
\begin{equation}
\ell(ws_\alpha)=\left\{
\begin{array}{ll}
\ell(w)+1 &\text{ if } w\alpha\in\bm\Phi^+\\
\ell(w)-1 &\text{ if } w\alpha\in\bm\Phi^-.
\end{array}
\right.
\end{equation}

\begin{rmk}\label{rmk:length}
By proposition 2.2 of \cite{Chin1} we have $\ell(w)=|\bm\Phi^+\cap w\bm\Phi^-|=|\bm\Phi^-\cap w\bm\Phi^+|$.
\end{rmk}

For every $P\subset\Sigma$ we denote by $\bm\Phi_P^+$ the set of positive roots generated by $P$, 
$\bm\Phi_P^-=-\bm\Phi_P^+$,  
$\mathbf{\Psi}_P^+=\bm\Phi^+\setminus\bm\Phi_P^+$ and
$\mathbf{\Psi}_P^-=-\mathbf{\Psi}_P^+$.
We denote by $W_P$ the subgroup of $W$ generated by the $s_{\alpha}$ with $\alpha\in P$ and by $\widehat P$ the complement of $P$ in $\Sigma$. We abbreviate $\widehat\alpha=\widehat{\{\alpha\}}$. 

\begin{ex}
If $\alpha=\alpha_{i,i+1}$ then $\widehat\alpha=
\{\alpha_{j,j+1}\in\Sigma\,|\,j\neq i\}$, $\bm\Psi_{\widehat\alpha}^+=\{\alpha_{hk}\in\bm\Phi^+\,|\,
1\leq h\leq i < k\leq m \}$ and
$\bm\Phi_{\widehat\alpha}^+=\{\alpha_{hk}\in\bm\Phi^+\,|\,
1\leq h<k\leq i \text{ or }i+1\leq h<k\leq m\}$.
\end{ex}

\begin{prop}\label{prop:minlength}
Let $P\subset\Sigma$ and $w$ be an element of minimal length in $wW_P\in W/W_P$.  Then $w\alpha\in\bm\Phi^+$ for every $\alpha\in \bm\Phi_P^+$ and for every $w'\in W_P$ we have 
$\ell(ww')=\ell(w)+\ell(w')$. 
\end{prop}

\begin{proof}
Proposition 2.4 and lemma 2.5 of \cite{Chin1}.
\end{proof}

\noindent
Proposition \ref{prop:minlength} implies that in each class of $W/W_P$ with $P\subset \Sigma$, there exists a unique element of minimal length and the same holds in each class of $W_P\bs W$.

\smallskip
If $\varpi$ is an uniformizer of $\ent_{D'}$ we identify 
$\tau_{i}=\left(\begin{smallmatrix} \mathbb{I}_i &0\\ 0&\varpi\mathbb{I}_{m'-i}\end{smallmatrix}\right)$ 
with $i\in\{0,\dots,m'\}$, defined in section 2.2 of \cite{Chin1}, to elements of $B^{\times}$ and then of $G$. 
For every $\alpha=\alpha_{i,i+1}\in \Sigma$ we write 
$\tau_\alpha=\tau_i$.
Let $\bm\Delta$ (resp. $\widehat{\bm\Delta}$) be the commutative monoid (resp. group) generated by $\tau_\alpha$ with $\alpha\in\Sigma$.
Then we can write every element $\tau$ of $\bm\Delta$ 
uniquely as $\tau=\prod_{\alpha\in\Sigma}\tau_\alpha^{i_\alpha}$ with $i_\alpha$ in $\N$ and uniquely as $\tau=\mathrm{diag}(1,\varpi^{a_1},\dots,\varpi^{a_{m-1}})$ with $0\leq a_1\leq\cdots\leq a_{m-1}$. 
In this case we denote $P(\tau)=\{\alpha\in\Sigma\,|\,i_{\alpha}=0\}$ and if $P\subset \{0,\dots,m\}$ or if $P\subset\Sigma$ we write $\tau_P$ in place of $\prod_{x\in P}\tau_x$. We remark that if $P\subset\Sigma$ then $P(\tau_P)=\widehat{P}$.

\subsection{The representation $\eta_\P$}
Let $\M=A(E)^{\times}\times\cdots\times A(E)^{\times}$ ($m'$ copies) which is a Levi subgroup of $G$ and let 
$\P$ be the parabolic subgroup of $G$ of upper triangular matrices with Levi component $\M$ and unipotent radical $\U$. Let $\P^-$ be the opposite parabolic subgroup to $\P$ and  $\U^-$ its unipotent radical.

\smallskip
We denote $U=K_B\cap \U$, $M=K_B\cap\M$ and $I_B=K^1_BMU$. Then $U$ is the group of unipotent upper triangular matrices with coefficients in $\ent_{D'}$, $M$ is the group of diagonal matrices with coefficients in $\ent_{D'}^{\times}$ and $I_B$ is the standard Iwahori subgroup of $K_B$. 

\smallskip
We denote by $\widetilde W$ the group $W\ltimes \widehat{\bm\Delta}$ of monomial matrices with coefficients in $\varpi^\Z$ which is called \emph{extended affine Weyl group of $B^{\times}$}.
We recall that $B^\times=I_B\widetilde W I_B$ and actually it is the disjoint union of $I_B\widetilde w I_B$ with $\widetilde w\in \widetilde W$.

\begin{rmk}\label{rmk:decomposed}
By proposition 2.16 of \cite{SecII}, which works for every decomposition $\mathrm{e}$ conforms to $\Lambda$ over $E$ and not necessarily subordinate to $\mathfrak{B}$, the groups $J^1$ and $H^1$ are decomposed with respect to $(\M,\P)$. 
Moreover, if $\M'=\prod_{i=1}^{r}GL_{m'_i}(A(E))$ with $\sum_{i=1}^rm'_i=m'$ is a standard Levi subgroup of $G$ containing $\M$ and $\P'$ is the upper standard parabolic subgroup of $G$ with Levi component $\M'$, then $J^1$ and $H^1$ are decomposed with respect to $(\M',\P')$.
\end{rmk}

Let $\mathfrak{J}^1=\mathfrak{J}^1(\beta,\Lambda)$ and $\mathfrak{H}^1=\mathfrak{H}^1(\beta,\Lambda)$ be the $\ent_{F}$-lattices of $A$ such that 
 $J^1=1+\mathfrak{J}^1$ and $H^1=1+\mathfrak{H}^1$ (see section 3.3 of \cite{SecI} or chapter 3 of \cite{BK1}). Then 
they are  
$(\mathfrak{B},\mathfrak{B})$-bimodules and  
 we have $\varpi\mathfrak{J}^1\subset\mathfrak{H}^1\subset\mathfrak{J}^1\subset M_{m'}(\mathfrak{P}(E))$.

\smallskip
Since $V^i\cong N$ for every 
$i\in\{1,\dots,m'\}$, we can identify every $\Lambda^i$ to a lattice sequence $\Lambda_0$ of $N$ with the same period of 
$\Lambda$, every $e^i\beta$ to an element $\beta_0\in A(E)$ and 
$\mathfrak{A}(\Lambda_0)$ to $\mathfrak{A}(E)$. By proposition 2.28 of \cite{SecI} the stratum 
$[\Lambda_0,n,0,\beta_0]$ of $A(E)$ is simple and critical exponents $k_0(\beta,\Lambda)$ and $k_0(\beta_0,\Lambda_0)$ are equal (for a definition of critical exponent see section 2.1 of \cite{SecI}). This implies that $\beta$ is minimal (i.e. $-k_0(\beta,\Lambda)=n$) if and only if $\beta_0$ is minimal. 
We denote 
$\mathfrak{J}^1_0=\mathfrak{J}^1(\beta_0,\Lambda_0)$, 
$\mathfrak{H}^1_0=\mathfrak{H}^1(\beta_0,\Lambda_0)$, 
$J^1_0=J^1(\beta_0,\Lambda_0)=1+\mathfrak{J}^1_0$ and 
$H^1_0=H^1(\beta_0,\Lambda_0)=1+\mathfrak{H}^1_0$. 

\begin{prop}\label{prop:structureJ1H1}
We have $\mathfrak{J}^1=M_{m'}(\mathfrak{J}^1_0)$ and 
$\mathfrak{H}^1=M_{m'}(\mathfrak{H}^1_0)$.
\end{prop}

\begin{proof}
We prove the result only for $\mathfrak{J}^1$ since the case of 
$\mathfrak{H}^1$ is similar.
We have to prove that for every $i,j\in\{1,\dots,m'\}$ we have 
$e^i\mathfrak{J}^1e^j=\mathfrak{J}^1_0$.
We need to recall the definition of $\mathfrak{J}(\beta,\Lambda)=\mathfrak{J}^0(\beta,\Lambda)$ and of
$\mathfrak{J}^k(\beta,\Lambda)$ with $k\geq 1$. 
By proposition 3.42 of \cite{SecI} if we denote $q=-k_0(\beta,\Lambda)$ and 
$s=[(q+1)/2]$ (where $[x]$ denotes the integer part of $x\in\mathbb{Q}$) we have
$\mathfrak{J}(\beta,\Lambda)=\mathfrak{B}+\mathfrak{P}^{s}$ if $\beta$ is minimal and 
$\mathfrak{J}(\beta,\Lambda)=\mathfrak{B}+\mathfrak{J}^{s}(\gamma,\Lambda)$ if $[\Lambda,n,q,\gamma]$ is a simple stratum equivalent to $[\Lambda,n,q,\beta]$. 
Then, if $\beta$ is minimal,
$\mathfrak{J}^k(\beta,\Lambda)=\mathfrak{J}(\beta,\Lambda)\cap \mathfrak{P}^k$ is equal to 
$\mathfrak{Q}^k+\mathfrak{P}^s$ if $0\leq k\leq s-1$ and to
$\mathfrak{P}^k$ if $k\geq s$. 
Otherwise, if $[\Lambda,n,q,\gamma]$ is a simple stratum equivalent to $[\Lambda,n,q,\beta]$, $\mathfrak{J}^k(\beta,\Lambda)$ is equal to  
$\mathfrak{Q}^k+\mathfrak{J}^s(\gamma,\Lambda)$ if $0\leq k\leq s-1$ and to
$\mathfrak{J}^k(\gamma,\Lambda)$ if $k\geq s$.
Similarly we obtain that if $\beta_0$ is minimal then
$\mathfrak{J}^k(\beta_0,\Lambda_0)$ is equal to 
$\wp_{D'}^k+\mathfrak{P}(E)^s$ if $0\leq k\leq s-1$ and to
$\mathfrak{P}(E)^k$ if $k\geq s$. 
Otherwise, if $[\Lambda_0,n,q,\gamma_0]$ is a simple stratum equivalent to $[\Lambda_0,n,q,\beta_0]$, $\mathfrak{J}^k(\beta_0,\Lambda_0)$ is equal to  
$\wp_{D'}^k+\mathfrak{J}^s(\gamma_0,\Lambda_0)$ if $k\leq s-1$ and to
$\mathfrak{J}^k(\gamma_0,\Lambda_0)$ if $k\geq s$.
We prove that $e^i\mathfrak{J}^k(\beta,\Lambda)e^j=\mathfrak{J}^k(\beta_0,\Lambda_0)$ for every $k\geq 0$ by induction on $q$.
If $q=n$ and so if $\beta$ and $\beta_0$ are minimal, since 
$\mathfrak{Q}=M_{m'}(\wp_{D'})$ and $\mathfrak{P}=M_{m'}(\mathfrak{P}(E))$ we have 
$e^i\mathfrak{Q}^k e^j=\wp_{D'}^k$ and 
$e^i\mathfrak{P}^k e^j=\mathfrak{P}(E)^k$ for every $k$ and so  $e^i\mathfrak{J}^k(\beta,\Lambda)e^j=\mathfrak{J}^k(\beta_0,\Lambda_0)$ for every $k\geq 0$.
Now if $q<n$ and so if $\beta$ and $\beta_0$ are not minimal, by proposition 1.20 of \cite{SeSt2} (see also the proof of theorem 2.2 of \cite{SecIII}) we can choose a simple stratum 
$[\Lambda_0,n,q,\gamma_0]$ equivalent to 
$[\Lambda_0,n,q,\beta_0]$ such that if $\gamma$ is the image of $\gamma_0$ by the diagonal embedding $A(E)\rightarrow A$ then 
$[\Lambda,n,q,\gamma]$ ia a simple stratum equivalent to 
$[\Lambda,n,q,\beta]$. By inductive hypothesis we have 
$e^i\mathfrak{J}^k(\gamma,\Lambda)e^j=\mathfrak{J}^k(\gamma_0,\Lambda_0)$ for every $k\geq 0$ and then we obtain 
$e^i\mathfrak{J}^k(\beta,\Lambda)e^j=\mathfrak{J}^k(\beta_0,\Lambda_0)$.
\end{proof}

Let $\theta_0$ be the transfer of $\theta$ to $\mathscr{C}_R(\Lambda_0,0,\beta)$. 
Since $H^1$ is a pro-$p$-group, proceeding as in proposition 2.16 of \cite{SecII}, the pair $(H^1,\theta)$ is decomposed with respect to $(\M,\P)$ and the restriction of $\theta$ to $H^1\cap \M=H^1_0\times\dots\times H^1_0$ is $\theta_0^{\otimes m'}$. We remark that in general $(J^1,\eta)$ is not decomposed with respect to $(\M,\P)$.
We denote by $\eta_0$ the Heisenberg representation of $\theta_0$ and we can consider the irreducible representation 
$\eta_\M=\eta_0^{\otimes m'}$ of 
$J^1_\M=J^1\cap\M=J^1_0\times \cdots\times J^1_0$.

\smallskip
We denote $J^1_\P=(J^1\cap \P)H^1$ and 
$H^1_\P=(J^1\cap \U)H^1$ which are subgroups of $J^1$.
They are normal in $J^1$ because $H^1$ contains the derived group of $J^1$.
Moreover, $J\cap \P$ normalizes $J^1_\P$ because $H^1$ is normal in $J$ and $J^1\cap \P$ is normal in $J\cap \P$. Then $J^1_\P$ is normal in $J^1(J\cap \P)$. 

\begin{rmk}\label{rmk:decomposed2}
Taking into account remark 5.7 of \cite{SeSt2}, proposition 5.3 of \cite{SeSt2} states that $J^1_\P$ and 
$H^1_\P$ are decomposed with respect to 
$(\M,\P)$ and so we have
$J^1_\P=(H^1\cap\U^-)J^1_\M(J^1\cap\U)$ and $
H^1_\P=(H^1\cap\U^-)(H^1\cap \M)(J^1\cap\U).$
Moreover, if if $\M'=\prod_{i=1}^{r}GL_{m'_i}(A(E))$ with $\sum_{i=1}^rm'_i=m'$ is a standard Levi subgroup of $G$ containing $\M$ and $\P'$ is the upper standard parabolic subgroup of $G$ with Levi component $\M'$, then $J^1_\P$ and $H^1_\P$ are decomposed with respect to $(\M',\P')$.
\end{rmk}

Let $\theta_\P$ be the character of $H^1_\P$ defined by $\theta_\P(uh)=\theta(h)$ for every $u\in J^1\cap\U$ and every $h\in H^1$. 
Since $J^1$ is a pro-$p$-group, proceeding as in Proposition 5.5 of \cite{SeSt2} we can construct an irreducible representation $\eta_\P$ of $J^1_\P$, unique up to isomorphism, whose restriction to $H^1_\P$ contains $\theta_\P$. 
Actually it is the natural representation of $J^1_\P$ on the $J^1\cap\U$-invariants of $\eta$. 
Furthermore, $\ind_{J^1_\P}^{J^1}(\eta_\P)$ is isomorphic to $\eta$, $I_G(\eta_\P)=J^1_\P B^{\times}J^1_\P$ and for every $y\in B^{\times}$ we have 
$\dim_R(I_y(\eta_\P))=1$.
We remark that $(J^1_\P,\eta_\P)$ is decomposed with respect to $(\M,\P)$ and the restriction of $\eta_P$ to $J^1_\M$ is $\eta_\M$. We denote by $V_\M$ the $R$-vector space of $\eta_\M$ and $\eta_\P$. 

\smallskip
Since $\ind_{J^1_\P}^{J^1}(\eta_\P)$ is isomorphic to $\eta$, 
we can identify the $R$-vector space $V_\eta$ of $\eta$ with the vector space of function $\varphi:J^1\rightarrow V_\M$ such that $\varphi(xj)=\eta_\P(x)\varphi(j)$ for every $x\in J^1_\P$ and $j\in J^1$. In this case $\eta(j)\varphi:x\mapsto \varphi(xj)$. 
By Mackey formula, $V_\M$ is a direct factor of $V_\eta$ and we can identify the subspace of function $\varphi\in V_\eta$ with support in $J^1_\P$ with it. This identification is given by $\varphi\mapsto \varphi(1)$ whose inverse is 
$v\mapsto \varphi_v$ where the support of 
$\varphi_v$ is $J^1_\P$ and $\varphi_v(1)=v$.
Let $\mathsf{p}:V_\eta\rightarrow V_\M$ be the canonical projection,  i.e. the restriction of a function in $V_\eta$ to $J^1_\P$, and let $\iota:V_\M\rightarrow V_\eta$ be the inclusion. 

\begin{rmk}
In general we can not define a representation $\kappa_P$ of $J_P=(J\cap P)H^1$ as in section 2.3 of \cite{SecII} 
or in section 5.5 of \cite{SeSt2}, 
because $\mathrm{e}$ is a decomposition conforms to $\Lambda$ over $E$ but it is not subordinate to $\mathfrak{B}$. In our case ($\mathfrak{B}$ maximal) the only decomposition conforms to $\Lambda$ over $E$ and subordinate to $\mathfrak{B}$ is the trivial one.
\end{rmk}

\begin{lemma} \label{lemma:etaP}
\mbox{}
\begin{enumerate}
	\item For every $j\in J^1_\P$ we have $\eta(j)\circ\iota=\iota\circ\eta_\P(j)$ and $\mathsf{p}\circ\eta(j)=\eta_\P(j)\circ\mathsf{p}$.
	\item For every $j\in J^1$ we have $\mathsf{p}\circ\eta(j)\circ\iota=\left\{
\begin{array}{ll}
\eta_{\P}(j)& \text{if } j\in J^1_\P\\
0 & \text{otherwise}
\end{array}
\right.
$
	\item $\sum_{j\in J^1/J^1_\P}\eta(j)\circ\iota\circ\mathsf{p}\circ\eta(j^{-1})$ is the identity of $\End_R(V_\M)$.
\end{enumerate}
\end{lemma}

\begin{proof}
To prove the first point, let $\varphi_v\in V_\M$ and $\varphi\in V_\eta$. 
Then $\eta(j)(\iota(\varphi_v))(1)=\varphi_v(j)=\eta_P(j)v$ and $\mathsf{p}(\eta(j)(\varphi))(1)=\varphi(j)=\eta_\P(j)\varphi(1)$. 
To prove the second point we observe that if $j\in J^1_P$ then $\mathsf{p}\circ\eta(j)\circ\iota=\mathsf{p}\circ\iota\circ\eta_\P(j)=\eta_\P(j)$ while if $j\notin J^1_\P$ the support of $\eta(j)(\iota(\varphi_v))$ is in $J^1_\P j^{-1}$ for every $\varphi_v\in V_{\M}$ and so $\mathsf{p}\circ\eta(j)\circ\iota=0$. 
Finally, to prove the third point we observe that for every 
$\varphi\in V_\eta$ the function
$\varphi_j=(\eta(j)\circ\iota\circ\mathsf{p}\circ\eta(j^{-1}))\varphi$ has support in $J^1_\P j^{-1}$ and 
$\varphi_j(j^{-1})=\varphi(j^{-1})$.
\end{proof}

We consider the surjective linear map 
$\mu:\End_R(V_\eta)\rightarrow \End_R(V_\M)$ given by $f\mapsto \mathsf{p}\circ f \circ \iota$.

\begin{lemma}\label{lemma:isomHeckeetaP}
The map $\zeta:\mathscr{H}_R(G,\eta)\rightarrow \mathscr{H}_R(G,\eta_\P)$ defined by $\Phi\mapsto \mu\circ\Phi$ for every $\Phi\in \mathscr{H}_R(G,\eta)$ is an isomorphism of $R$-algebras. 
Moreover, if the support of $\Phi\in \mathscr{H}_R(G,\eta)$ is in $J^1xJ^1$ with $x\in B^{\times}$ then the support of $\zeta(\Phi)$ is in $J^1_\P x J^1_\P$.
\end{lemma}

\begin{proof}
Let $\Phi\in \mathscr{H}_R(G,\eta)$. Then 
the support of $\mu\circ\Phi$ is included in the support of $\Phi$ which is compact. 
Moreover, for every $x_1,x_2\in J^1_\P$ and every $j\in J^1$ we have
$\mu(\Phi(x_1jx_2))=\mathsf{p}\circ\eta(x_1)\circ\Phi(j)\circ\eta(x_2)\circ\iota$ that by lemma \ref{lemma:etaP} is 
$\eta_\P(x_1)\circ \mu(\Phi(j))\circ \eta_\P(x_2)$. 
Hence, $\zeta$ is well-defined and it is clearly an $R$-linear map. 
Let $\Phi_1,\Phi_2\in \mathscr{H}_R(G,\eta)$. 
For every $g\in G$ we have
\begin{align*}
\big((\mu\circ\Phi_1)*(\mu\circ\Phi_2)\big)(g)
&=\sum_{x\in G/J^1_\P}\mathsf{p}\circ\Phi_1(x)\circ\iota\circ\mathsf{p}\circ\Phi_2(x^{-1}g)\circ\iota\\
&=\sum_{y\in G/J^1}\sum_{z\in J^1/J^1_\P}\mathsf{p}\circ\Phi_1(yz)\circ\iota\circ\mathsf{p}\circ\Phi_2(z^{-1}y^{-1}g)\circ\iota\\
&=\sum_{y\in G/J^1}\mathsf{p}\circ\Phi_1(y)\circ
\bigg(\sum_{z\in J^1/J^1_\P} \eta(z)\circ\iota\circ\mathsf{p}\circ\eta(z^{-1})\bigg)\circ\Phi_2(y^{-1}g)\circ\iota\\
\text{\scriptsize{(lemma \ref{lemma:etaP})}}
&=\sum_{y\in G/J^1}\mathsf{p}\circ\Phi_1(y)\circ\Phi_2(y^{-1}g)\circ\iota=(\mu\circ(\Phi_1*\Phi_2))(g)
\end{align*}
and so $\zeta$ is a homomorphism of $R$-algebras.
Let $\Phi\in \mathscr{H}_R(G,\eta)$ such that 
$\mathsf{p}\circ\Phi(g)\circ\iota=0$ for every $g\in G$.
Then by lemma \ref{lemma:etaP}, for every $g'\in G$ we have
\begin{align*}
\Phi(g')&=
\sum_{j_1\in J^1/J^1_\P} \eta(j_1)\circ\iota\circ\mathsf{p}\circ\eta(j_1^{-1})\circ\Phi(g')\circ
\sum_{j_2\in J^1/J^1_\P} \eta(j_2)\circ\iota\circ\mathsf{p}\circ\eta(j_2^{-1})\\
&=\sum_{j_1,j_2\in J^1/J^1_\P} \eta(j_1)\circ\iota\circ(\mathsf{p}\circ\Phi(j_1^{-1}g'j_2)\circ\iota)\circ\mathsf{p}\circ\eta(j_2^{-1})=0
\end{align*}
and then $\zeta$ is injective.
Now, we know that $\mathscr{H}_R(G,\eta)\cong \End_G(\ind_{J^1}^G(\eta))$, $\mathscr{H}_R(G,\eta_\P)\cong \End_G(\ind_{J^1_\P}^G(\eta_\P))$ and  $\ind_{J^1_\P}^{J^1}(\eta_\P)\cong\eta$. Then by transitivity of the induction we have 
$\mathscr{H}_R(G,\eta)\cong \mathscr{H}_R(G,\eta_\P)$ and then $\zeta$ must be bijective.
Furthermore, if $\Phi\in \mathscr{H}_R(G,\eta)$ has support in $J^1xJ^1$ with $x\in B^{\times}$ then the support of $\zeta(\Phi)$ is in $J^1xJ^1\cap I_G(\eta_\P)
=J^1xJ^1\cap J^1_\P B^{\times}J^1_\P =J^1_\P xJ^1_\P$.
\end{proof}

\begin{lemma}\label{lemma:product}
Let $x_1,x_2\in B^{\times}$ and let $\widetilde f_i\in\mathscr{H}_R(G,\eta)_{J^1x_iJ^1}$ and $\widehat f_i=\zeta(\widetilde f_i)$ for $i\in\{1,2\}$.
\begin{enumerate}
\item If $x_1$ or $x_2$ normalizes $J^1_\P$ then the support of $\widehat f_1*\widehat f_2$ is in $J^1_\P x_1x_2J^1_\P$ and $$(\widehat f_1*\widehat f_2)(x_1x_2)=\widehat f_1(x_1)\circ \widehat f_2(x_2).$$
\item If $x_1$ or $x_2$ normalizes $J^1$ then the support of $\widehat f_1*\widehat f_2$ is in $J^1_\P x_1x_2J^1_\P$ and 
$$(\widehat f_1*\widehat f_2)(x_1x_2)=\mathsf{p}\circ\widetilde f_1(x_1)\circ \widetilde f_2(x_2)\circ\iota.$$
\end{enumerate}
\end{lemma}

\begin{proof}
First point follows by remark \ref{rmk:normal}.
If $x_1$ or $x_2$ normalizes $J^1$, by remark \ref{rmk:normal} the support of $\widetilde f_1*\widetilde f_2$ is in $J^1x_1x_2J^1$ and so the support of $\widehat f_1*\widehat f_2=\zeta(\widetilde f_1*\widetilde f_2)$ is in $J^1x_1x_2J^1\cap I_G(\eta_\P)=J^1_\P x_1x_2J^1_\P$ and moreover $(\widehat f_1*\widehat f_2)(x_1x_2)=\zeta(\widetilde f_1*\widetilde f_2)(x_1x_2)=\mathsf{p}\circ\widetilde f_1(x_1)\circ \widetilde f_2(x_2)\circ\iota.$
\end{proof}

\begin{lemma}\label{lemma:intertetaP}
For every $x\in B^{\times}\cap\M$ and every 
$y\in I_G(\eta_\P)$ which normalizes $J^1_\M$
we have $I_x(\eta_\P)=I_x(\eta_\M)$ and $I_y(\eta_\P)=I_y(\eta_\M)$.
Moreover, every non-zero element in $I_z(\eta_\P)$, with 
$z\in I_G(\eta_\P)$, is invertible.
\end{lemma}

\begin{proof}
For the first assertion, in both cases the $R$-vector spaces are $1$-dimensional and so it suffices to prove an inclusion. Since $\eta_\M$ is the restriction of $\eta_\P$ to $J^1_\M$,
for every $x'\in I_G(\eta_\P)$ we have  $I_{x'}(\eta_\P)\subseteq I_{x'}(\eta_\M)$. For the second assertion, we observe that $I_G(\eta_\P)=J^1_\P B^{\times}J^1_\P=
J^1_\P I_B\widetilde{W}I_B J^1_\P$. 
Now $I_B$ normalizes $J^1_\P$ since it is contained in 
$J^1(J\cap \P)$ while $\widetilde{W}$ normalizes $J^1_\M$. 
Take $z=z_1z_2z_3\in I_G(\eta_\P)$ with $z_1\in J^1_\P I_B$, $z_2\in \widetilde{W}$ and $z_3\in I_B J^1_\P$ and take a non-zero element $\gamma$ in $I_z(\eta_P)$. 
Let $\gamma_1$ and $\gamma_3$ two invertible elements in $I_{z_1^{-1}}(\eta_\P)$
and in $I_{z_3^{-1}}(\eta_\P)$ respectively. 
Then $\gamma_1\circ\gamma\circ\gamma_3$ is a non-zero element in $I_{z_2}(\eta_\P)=I_{z_2}(\eta_\M)$ and so it is invertible.
\end{proof}

\subsection{The isomorphism  $\mathscr{H}_R(J,\eta)\cong\mathscr{H}_R(K_B,K^1_B)$}\label{subsec:isomfinitealgebras}
In this paragraph we want to prove that the subalgebra 
$\mathscr{H}_R(K_B,K^1_B)$ of $\mathscr{H}_R(B^{\times},K^1_B)$ is isomorphic to the subalgebra $\mathscr{H}_R(J,\eta_\P)$ of $\mathscr{H}_R(G,\eta_\P)$ and so to $\mathscr{H}_R(J,\eta)$.

\smallskip
In accordance with chapter 2 of \cite{Chin1}, 
we denote by $f_x \in \mathscr{H}_{R}(B^{\times},K^1_B)$
the characteristic function of $K^1_Bx K^1_B$
for every $x\in B^{\times}$ and we write 
$\Phi_1 \Phi_2= \Phi_1 * \Phi_2$ for every $\Phi_1$ and $\Phi_2$ in $\mathscr{H}_{R}(B^{\times},K^1_B)$, in 
$\mathscr{H}_R(G,\eta)$ or in $\mathscr{H}_R(G,\eta_\P)$.

\smallskip
We observe that every element in $\mathscr{H}_R(J,\eta_\P)$ has support in $J\cap J^1_\P B^{\times} J^1_\P= J^1_\P(J\cap  B^{\times}) J^1_\P= J^1_\P K_B J^1_\P$ and so its image by $\zeta^{-1}$ has support in $J^1 K_B J^1$. 
This implies that $\zeta$ induces an algebra isomorphism from
$\mathscr{H}_R(J,\eta)$ to $\mathscr{H}_R(J,\eta_\P)$.
We also remark that $\mathscr{H}_R(K_B,K^1_B)$ is isomorphic to the group algebra $R[K_B/K^1_B]\cong R[J/J^1]$, then we can identify every 
$\Phi\in \mathscr{H}_R(K_B,K^1_B)$ to a function $\Phi\in \mathscr{H}_R(J,J^1)$.

\smallskip
From now on we fix a $\beta$-extension $\kappa$ of $\eta$.
We recall that $\res^J_{J^1}\kappa =\eta$, $I_G(\eta)=I_G(\kappa)=J^1B^{\times}J^1$ and for every $y\in B^{\times}$ we have $I_y(\eta)=I_y(\kappa)$ which is an 
$R$-vector space of dimension $1$. 
Then $V_\eta$ is also the $R$-vector space of $\kappa$ and 
$\kappa(j)\in I_j(\eta)$ for every $j\in J$.

\begin{lemma}\label{lemma:finitealgebras}
The map $\Theta':\mathscr{H}_R(K_B,K^1_B)\rightarrow\mathscr{H}_R(J,\eta)$ defined by $\Phi\mapsto \Phi\otimes\kappa$ for every $\Phi\in \mathscr{H}_R(K_B,K^1_B)$ is an algebra isomorphism. 
\end{lemma}

\begin{proof}
The map is well-defined since for every $\Phi\in \mathscr{H}_R(K_B,K^1_B)$ we have $\Phi\otimes\kappa:J\rightarrow \End_R(V_\eta)$ 
and $(\Phi\otimes\kappa)(j_1 jj'_1)=\Phi(j)\kappa(j_1 jj'_1)=\eta(j_1)\circ(\Phi(j)\kappa(j))\circ\eta(j'_1)$ for every $j\in J$ and $j_1,j'_1\in J^1$. 
It is clearly $R$-linear and 
\begin{align*}
\Theta'(\Phi_1*\Phi_2)(j)&=\sum_{x\in J/J^1} \Phi_1(x)\Phi_2(x^{-1}j)\kappa(j)=\sum_{x\in J/J^1} \Phi_1(x)\Phi_2(x^{-1}j)\kappa(x)\circ \kappa(x^{-1}j)\\
&=\sum_{x\in J/J^1} (\Phi_1(x)\kappa(x))\circ (\Phi_2(x^{-1}j)\kappa(x^{-1}j))=(\Theta'(\Phi_1)*\Theta'(\Phi_2))(j)
\end{align*}
for every $\Phi_1,\Phi_2\in \mathscr{H}_R(K_B,K^1_B)$ and $j\in J$. Hence, $\Theta'$ is an $R$-algebra homomorphism.
It is injective because $\kappa(j)\in GL(V_\eta)$ for every $j\in J$. 
Let $\widetilde{f}\in \mathscr{H}_R(J,\eta)$ and $j\in J$. 
Since $\widetilde f(j)\in I_j(\eta)=\Hom_{J^1}(\eta,\eta^j)$, which is of dimension $1$, we have $\widetilde f(j)\in R\kappa(j)$ and then we can write 
$\widetilde{f}(j)=\Phi(j)\kappa(j)$ with $\Phi:J\rightarrow R$.
Since $\widetilde{f}\in \mathscr{H}_R(J,\eta)$, for every $j_1\in J^1$ we have $\Phi(j_1j)\kappa(j_1j)=\widetilde{f}(j_1j)=\eta(j_1)\widetilde{f}(j)=\eta(j_1)\Phi(j)\kappa(j)=
\Phi(j)\kappa(j_1j)$ and so $\Phi\in \mathscr{H}_R(J,J^1)$. 
We conclude that $\Theta'$ is surjective and then it is an algebra isomorphism.
\end{proof}

Composing the restriction of $\zeta$ to $\mathscr{H}_R(J,\eta)$ with $\Theta'$ we obtain an algebra isomorphism
$\mathscr{H}_R(K_B,K^1_B)\rightarrow\mathscr{H}_R(J,\eta_\P)$. 
For every $x\in K_B$ let $\widetilde f_x=\Theta'(f_x)\in \mathscr{H}_R(J,\eta)$ which is given by 
$\widetilde f_x(y)=\kappa(y)$ for every $y\in J^1xJ^1=J^1x$ and let $\widehat f_x=\zeta(\widetilde f_x)\in \mathscr{H}_R(J,\eta_\P)$ which is given by $\widehat f_x(z)=\mathsf{p}\circ\kappa(z)\circ\iota$ for every $z\in J^1_\P xJ^1_\P$.

\subsection{Generators and relations of $\mathscr{H}_R(B^{\times},K^1_B)$}\label{subsec:decomp}
In this paragraph we introduce some notations and we recall the presentation by generators and relations of the algebra $\mathscr{H}_R(B^{\times},K^1_B)$ presented in \cite{Chin1}.

\smallskip
We denote $\Omega=K_B\cup\{\tau_0,\tau_0^{-1}\}\cup
\{\tau_\alpha\,|\,\alpha\in\Sigma\}$ 
and $\bm\Omega=\{f_\omega \,|\,\omega\in\Omega\}$ which is a finite set.
We define some subgroup of $G$, through its identification with $GL_{m'}(A(E))$. 
For every $\alpha=\alpha_{ij}\in\bm\Phi$ we denote by 
$\U_{\alpha}$ the subgroup of matrices $(a_{hk})\in G$ with $a_{hh}=1$ for every 
$h\in\{1,\dots,m'\}$, $a_{ij}\in A(E)$ and $a_{hk}=0$ if $h\neq k$ and $(h,k)\neq (i,j)$. 
For every $P\subset\Sigma$ we denote by $\M_{P}$ the standard Levi subgroup associated to $P$ and by
$\U_{P}^+$ (resp. $\U_{P}^-$) the unipotent radical of upper (resp. lower) standard parabolic subgroups with Levi component $\M_P$. We remark that $\M=\M_{\emptyset}$, $\U=\U_{\emptyset}$ and $\U^-=\U_{\emptyset}^-$. 
Thus, we have 
$\U_{P}^+=\prod_{\alpha\in\mathbf{\Psi}_P^+}\U_{\alpha}$ 
and $\U_{P}^-=\prod_{\alpha\in\mathbf{\Psi}_P^-}\U_{\alpha}$.
Furthermore, if $P_1\subset P_2\subset\Sigma$ then $\U_{P_2}^+$ is a subgroup of $\U_{P_1}^+$ and  $\U_{P_2}^-$ a subgroup of $\U_{P_1}^-$. 

\begin{rmk}\label{rmk:UintersectJ1H1}
By proposition \ref{prop:structureJ1H1}, if we take 
$\alpha=\alpha_{ij}\in\bm\Phi$ and
$(a_{hk})$ in $\U_\alpha\cap J^1$ (resp. $\U_\alpha\cap H^1$) then $a_{ij}$ is in $\mathfrak{J}^1_0$  (resp. $\mathfrak{H}^1_0$).
\end{rmk}

\begin{rmk}\label{rmk:UintersectK}
In accordance with paragraph 2.2 of \cite{Chin1} we denote $M_P=\M_P\cap K_B$, $U_P^+=\U_P^+\cap K_B$ and $U_P^-=\U_P^-\cap K_B$ for every $P\subset \Sigma$ and 
$U_\alpha=\U_\alpha\cap K_B$ for every $\alpha\in \bm\Phi$.
\end{rmk}

As in paragraph 2.3 of \cite{Chin1}, for every $\alpha=\alpha_{i,i+1}\in \Sigma$ and $w\in W$ we consider the following sets: 
$A(w,\alpha)=\{w(j)\,|\,i+1\leq j\leq m\}$, $B(w,\alpha)=\{w(j)-1\,|\,i+1\leq j\leq m\}$, $P'(w,\alpha)=A(w,\alpha)\setminus B(w,\alpha)$, $P(w,\alpha)=\{\alpha_{i,i+1}\in\Sigma\,|\, i\in P'(w,\alpha)\}$ and $Q(w,\alpha)=B(w,\alpha)\setminus A(w,\alpha)$.
We remark that $\tau_{P'(w,\alpha)}=\tau_{P(w,\alpha)}$ because $0\notin P'(w,\alpha)$ and $\tau_m=\mathbb{I}_{m}$.
Moreover, if $\alpha=\alpha_{i,i+1}\in \Sigma$, $w'\in W$ and $w$ is of minimal length in $w'W_{\widehat {\alpha}}\in W/W_{\widehat {\alpha}}$ then we have 
\[w'\tau_i w'^{-1}=w\tau_iw^{-1}=\prod_{h=i+1}^{m}w\tau_{h-1}\tau_{h}^{-1}w^{-1}=\prod_{h=i+1}^{m}\tau_{w(h)-1}\tau_{w(h)}^{-1}=\tau_{P(w,\alpha)}^{-1}\tau_{Q(w,\alpha)}.\]

\begin{lemma}\label{lemma:relations}
The algebra $\mathscr{H}_{R}(B^{\times},K^1_B)$ is the $R$-algebra generated by $\bm\Omega$ subject to the following relations 
\begin{enumerate}[1.]
	\item  $f_{k}=1$ for every $k\in K^1$ and $f_{k_1}f_{k_2}=f_{k_1k_2}$ for every $k_1,k_2\in K$;
	\item $f_{\tau_0}f_{\tau_0^{-1}}=1$
				and $f_{\tau_0^{-1}}f_{\omega}=f_{\tau_0^{-1}\omega\tau_0}f_{\tau_0^{-1}}$ for every $\omega\in \Omega$;
	\item $f_{\tau_{\alpha}}f_x=f_{\tau_{\alpha}x\tau_{\alpha}^{-1}}f_{\tau_{\alpha}}$ for every $\alpha\in\Sigma$ and $x\in M_{\widehat\alpha}$;
  \item $f_uf_{\tau_{\alpha}}=f_{\tau_{\alpha}}$ if $u\in U_{\alpha'}$ with $\alpha'\in \mathbf{\Psi}_{\widehat{\alpha}}^+$, for every $\alpha\in\Sigma$;
	\item $f_{\tau_{\alpha}}f_u=f_{\tau_{\alpha}}$ if $u\in U_{\alpha'}$ with $\alpha'\in \mathbf{\Psi}_{\widehat{\alpha}}^-$, for every $\alpha\in\Sigma$;
	\item $f_{\tau_{\alpha}}f_{\tau_{\alpha'}}=f_{\tau_{\alpha'}}f_{\tau_{\alpha}}$ for every $\alpha,\alpha'\in\Sigma$;
	\item $\displaystyle{\left(\prod_{\alpha'\in P(w,\alpha)}f_{\tau_{\alpha'}}\right)f_wf_{\tau_\alpha}f_{w^{-1}}=
	q^{\ell(w)}\left(\prod_{\alpha''\in Q(w,\alpha)}f_{\tau_{\alpha''}}\right)\left(\sum_{u}f_u\right)}\,\quad$ 	for every $\alpha\in\Sigma$ and $w$ of minimal length in $wW_{\widehat \alpha}\in W/W_{\widehat \alpha}$ and where $u$ 
	describes a system of representatives of $(U\cap wU^-	w^{-1})K^1_B/K^1_B$ in $U\cap wU^-	w^{-1}$. 
\end{enumerate}
\end{lemma}
\begin{proof}
The only difference between this presentation and these in \cite{Chin1} is the relation 3 which is equivalent to relations 3, 4 and 7 of definition 2.21 of \cite{Chin1} because $\M\cap K_B$, $U_{\alpha'}$ with $\alpha'\in\bm\Phi_{\widehat\alpha}$ and $W_{\widehat\alpha}$ generate $M_{\widehat\alpha}$.
\end{proof}

Hence, to define an algebra homomorphism from 
$\mathscr{H}_{R}(B^{\times},K^1_B)$ to $\mathscr{H}_R(G,\eta_\P)$, it is sufficient to choose elements 
$\widehat f_\omega\in\mathscr{H}_R(G,\eta_\P)$ for every $\omega\in\Omega$ such that $\widehat f_\omega$ respect the relations of lemma \ref{lemma:relations}. 
We remark that we can take $\widehat f_\omega\in 
\mathscr{H}_R(G,\eta_\P)_{J^1_\P \omega J^1_\P}$ for every $\omega\in\Omega$ and we recall that in paragraph \ref{subsec:isomfinitealgebras} we have just defined $\widehat f_k$ for every $k\in K_B$
as the image of $f_k$ by $\zeta\circ\Theta'$.

\subsection{Some decompositions of $J^1_\P$-double cosets}
In this paragraph we introduce some notations and some tools that we will use to construct elements in $\mathscr{H}_R(G,\eta_\P)_{J^1_\P \tau_i J^1_\P}$ with $i\in\{0,\dots,m'-1\}$.

\begin{lemma}\label{lemma:J1Ptau}
Let $\tau\in\bm\Delta$ and $P=P(\tau)$.
\begin{enumerate}
\item We have $J^1_\P=(J^1_\P \cap \U^-_P)(J^1_\P \cap \M_P)(J^1_\P \cap \U_P^+)=
(J^1_\P \cap \U_P^+)(J^1_\P \cap \M_P)(J^1_\P \cap \U_P^-)$.
\item We have 
$(J^1_\P \cap \U_P^+)^\tau\subset H^1\cap \U_P^+\subset J^1_\P \cap \U_P^+$,
$(J^1_\P \cap \U^-_P)^{\tau^{-1}}\subset 
(J^1 \cap \U^-_P)^{\tau^{-1}}\subset  
H^1 \cap \U^-_P=J^1_\P \cap \U^-_P$
 and 
$(J^1_\P \cap \M_P)^{\tau}=J^1_\P \cap \M_P$.
\item We have  
$(J^1_\P\cap \U)^\tau\subset J^1_\P\cap \U$,
$(J^1_\P\cap \U^-)^{\tau^{-1}}\subset J^1_\P\cap \U^-$ and  
$(J^1_\M)^{\tau}=J^1_\M$.
\end{enumerate}
\end{lemma}

\begin{proof}
The first point follows by remark \ref{rmk:decomposed2}.
To prove the second point we observe that
remark \ref{rmk:UintersectJ1H1} implies that
$(J^1_\P\cap \U_P^+)^\tau = 
(J^1\cap\prod_{\alpha\in\bm\Psi_{P}^+}\U_\alpha)^\tau$ is contained in 
$(\mathbb{I}_{m'}+\varpi \mathfrak{J}^1)\cap \U_P^+$ which is in $H^1\cap \U_P^+\subset J^1_\P\cap \U_P^+$.
Similarly we prove $(J^1 \cap \U^-_P)^{\tau^{-1}}\subset H^1 \cap \U^-_P$.
Moreover, since 
$\varpi^{-1}\mathfrak{J}^1_0\varpi= \mathfrak{J}^1_0$ and
$\varpi^{-1}\mathfrak{H}^1_0\varpi= \mathfrak{H}^1_0$, we have
$(J^1_\P\cap \M_P)^\tau= J^1_\P\cap \M_P$. 
To prove the third point, we observe that  
$(J^1_\P\cap \U)^\tau\subset \big((J^1_\P\cap \M_P)(J^1_\P\cap \U_P^+)\big)^{\tau}\cap \U$ which is in $(J^1_\P\cap \M_P)(J^1_\P\cap \U_P^+)\cap \U=J^1_\P\cap \U$. 
Similarly we prove $(J^1_\P\cap \U^-)^{\tau^{-1}}\subset J^1_\P\cap \U^-$. 
Finally, since $\varpi^{-1}\mathfrak{J}^1_0\varpi= \mathfrak{J}^1_0$ we obtain $(J^1_\M)^{\tau}=J^1_\M$.
\end{proof}

\begin{lemma}\label{lemma:J1Ptau:2}
If $\tau\in\bm\Delta$ then $J^1_\P\tau J^1_\P=(J^1_\P\cap \U_{P(\tau)}^-)\tau J^1_\P=J^1_\P\tau (J^1_\P\cap \U_{P(\tau)}^+)$ and 
$J^1_\P\tau^{-1} J^1_\P=(J^1_\P\cap \U_{P(\tau)}^+)\tau^{-1} J^1_\P=J^1_\P\tau^{-1} (J^1_\P\cap \U_{P(\tau)}^-)$.
\end{lemma}

\begin{proof}
Let $P=P(\tau)$. 
By lemma \ref{lemma:J1Ptau} we have 
$J^1_\P=(J^1_\P\cap \U_{P}^-)(J^1_\P\cap \M_{P})
(J^1_\P\cap \U_{P}^+)$
and so we obtain
$J^1_\P\tau J_\P^1=(J^1_\P\cap \U_{P}^-)\tau(J^1_\P\cap \M_{P})^\tau (J^1_\P\cap \U_{P}^+)^\tau J^1_\P$ which is equal to 
$(J^1_\P\cap \U_{P}^-)\tau J^1_\P$ by lemma \ref{lemma:J1Ptau}. 
Similarly we prove other equalities.
\end{proof}

\begin{lemma}\label{lemma:J1Pw}
If $w\in W$ then $(J^1_\P)^w J^1_\P=(J^1\cap \U^w\cap \U^-)J^1_\P$.
\end{lemma}

\begin{proof}
Since $(H^1\cap \U^-)^w\subset J^1_\P$ and $(J^1_\M)^w=J^1_\M$ we obtain $(J^1_\P)^w J^1_\P=(J^1\cap\U)^wJ^1_\P$. 
Moreover, we have $(J^1\cap\U)^w\cap\U\subset J^1_\P$ and so 
$(J^1_\P)^w J^1_\P=(J^1\cap\U^w\cap\U^-)J^1_\P$.
\end{proof}

\begin{lemma}\label{lemma:J1PU}
We have $J^1_\P \U^-J^1_\P\cap \U= J^1_\P\cap\U$ and $J^1_\P\U J^1_\P\cap \U^-= J^1_\P\cap\U^-$.
\end{lemma}

\begin{proof}
We have 
$J^1_\P\U^-J^1_\P\cap\U=
(J^1_\P\cap \U)(J^1_\P\cap\M)(J^1_\P\cap \U^-)
\U^-(J^1_\P\cap \U^-)(J^1_\P\cap\M)(J^1_\P\cap \U)\cap \U=
(J^1_\P\cap\U)\Big((J^1_\P\cap\M)\U^-(J^1_\P\cap\M)\cap \U\Big)(J^1_\P\cap \U)\subset 
(J^1_\P\cap\U)(\P^-\cap\U)(J^1_\P\cap \U)=J^1_\P\cap \U$ which is clearly contained in $J^1_\P\U^-J^1_\P\cap\U$.
Similarly we prove the second statement.
\end{proof}

\begin{lemma}\label{lemma:producttau}
Let $\tau,\tau'\in\bm\Delta$. Then 
$J^1_\P\tau J^1_\P\tau'J^1_\P=J^1_\P\tau\tau'J^1_\P$ and 
$(J^1_\P)^{\tau} J^1_\P\cap (J^1_\P)^{\tau'^{-1}}J^1_\P=J^1_\P.$
\end{lemma}

\begin{proof}
By lemma \ref{lemma:J1Ptau:2} we have 
$J^1_\P\tau J^1_\P\tau'J^1_\P=J^1_\P\tau(J^1_\P\cap \U_{P(\tau)}^+)\tau'J^1_\P
=J^1_\P\tau\tau'(J^1_\P\cap \U_{P(\tau)}^+)^{\tau'}J^1_\P
$.
By lemma \ref{lemma:J1Ptau} it is in 
$J^1_\P\tau\tau' (J^1_\P\cap \U)^{\tau'}J^1_\P\subset J^1_\P\tau\tau'J^1_\P$ and so $J^1_\P\tau J^1_\P\tau'J^1_\P=J^1_\P\tau\tau'J^1_\P$.
Now, by lemma \ref{lemma:J1Ptau:2}, the set $(J^1_\P)^{\tau} J^1_\P\cap (J^1_\P)^{\tau'^{-1}}J^1_\P$ is contained in $(J^1_\P\cap\U^-)^\tau J^1_\P \cap (J^1_\P
\cap\U)^{\tau'^{-1}}J^1_\P
=\Big((J^1_\P\cap\U^-)^{\tau} J^1_\P \cap (J^1_\P\cap\U)^{\tau'^{-1}}\Big)J^1_\P$ which is equal to $J^1_\P$ by lemma \ref{lemma:J1PU}.
\end{proof}

\begin{rmk}\label{rmk:J1}
We can prove similar results of lemmas \ref{lemma:J1Ptau}, \ref{lemma:J1Ptau:2}, \ref{lemma:J1PU} and \ref{lemma:producttau} by replacing $J^1_\P$ with $J^1$.
\end{rmk}

\begin{lemma}\label{lemma:roots}
Let $\alpha=\alpha_{i,i+1}\in\Sigma$, $w\in W$ and $P=P(w,\alpha)$. 
Then 
$\mathbf{\Psi}_{\widehat P}^+\cap w \mathbf{\Psi}_{\widehat\alpha}^-=\bm\Phi^+\cap w \mathbf{\Psi}_{\widehat\alpha}^-$ and 
$\mathbf{\Psi}_{\widehat P}^-\cap w \mathbf{\Psi}_{\widehat\alpha}^+=\bm\Phi^-\cap w \mathbf{\Psi}_{\widehat\alpha}^+$. If in addition $w$ is of minimal length in $w W_{\widehat\alpha}\in W/ W_{\widehat\alpha}$ then 
$\bm\Phi^+\cap w \mathbf{\Psi}_{\widehat\alpha}^-=\bm\Phi^+\cap w\bm\Phi^-$ and $\bm\Phi^-\cap w \mathbf{\Psi}_{\widehat\alpha}^+=\bm\Phi^-\cap w\bm\Phi^+$.
\end{lemma}

\begin{proof}
It follows by lemma 2.19 of \cite{Chin1}.
\end{proof}

From now on, we denote $\delta(\mathfrak{J}^1_0,\mathfrak{H}^1_0)=
[\mathfrak{J}^1_0:\mathfrak{H}^1_0]$ and
$\delta(\mathfrak{H}^1_0,\varpi\mathfrak{H}^1_0)=
[\mathfrak{H}^1_0:\varpi\mathfrak{H}^1_0]$.

\begin{rmk}\label{rmk:index}
By remark \ref{rmk:UintersectJ1H1} 
we have 
$\delta(\mathfrak{J}^1_0,\mathfrak{H}^1_0)=[J^1\cap\U_\alpha:H^1\cap\U_\alpha]$ and 
$\delta(\mathfrak{H}^1_0,\varpi\mathfrak{H}^1_0)=
[H^1\cap\U_{\alpha'}:(H^1\cap\U_{\alpha'})^{\tau_{\alpha'}}]=
[H^1\cap\U_{\alpha''}:(H^1\cap\U_{\alpha''})^{\tau_{\alpha''}^{-1}}]$
for every $\alpha\in\bm\Phi$, $\alpha'\in \bm\Phi^+$ and $\alpha''\in \bm\Phi^-$.
In particular $\delta(\mathfrak{J}^1_0,\mathfrak{H}^1_0)$ and 
$\delta(\mathfrak{H}^1_0,\varpi\mathfrak{H}^1_0)$ are powers of $p$ and so they are invertible in $R$.
\end{rmk}

\smallskip
From now on we fix $1\leq i\leq m'-1$ and we consider 
$\alpha=\alpha_{ii+1}$, $w$ of minimal length in 
$wW_{\widehat{\alpha}}$, $P=P(w,\alpha)$ and $Q=Q(w,\alpha)$. 

\begin{rmk}\label{rmk:roots}
Lemma \ref{lemma:roots} implies that 
$w \U_{\widehat\alpha}^-w^{-1}\cap \U_{\widehat{P}}^+=
w \U^- w^{-1}\cap \U^+$ and
$w \U_{\widehat\alpha}^+w^{-1}\cap \U_{\widehat{P}}^-=
w \U w^{-1}\cap \U^-$. 
Moreover, we have $\ell(w)=|\mathbf{\Psi}_{\widehat P}^+\cap w \mathbf{\Psi}_{\widehat\alpha}^-|=|\mathbf{\Psi}_{\widehat P}^-\cap w \mathbf{\Psi}_{\widehat\alpha}^+|$ by remark \ref{rmk:length}.
\end{rmk}

We define
\begin{equation}\label{eq:defV}
\V(w,\alpha)=
(J^1_\P\cap w \U_{\widehat\alpha}^+w^{-1}\cap \U_{\widehat{P}}^-)^{w\tau_{\alpha}^{-1}w^{-1}}
\end{equation}
which is a pro-$p$-group. 
We remark that it is equal to $(J^1_\P\cap w \U w^{-1}\cap \U^-)^{w\tau_{\alpha}^{-1}w^{-1}}$ by remark \ref{rmk:roots} and to
$(H^1\cap w \U_{\widehat\alpha}^+w^{-1}\cap \U_{\widehat{P}}^-)^{w\tau_{\alpha}^{-1}w^{-1}}$ since $J^1_\P\cap \U_{\widehat{P}}^-=H^1\cap \U_{\widehat{P}}^-$. Then $\V(w,\alpha)$ is equal to 
$$\prod_{\alpha'\in w\bm\Psi_{\widehat{\alpha}}^+\cap\bm\Psi_{\widehat{P}}^-}
(H^1\cap \U_{\alpha'})^{w\tau_{\alpha}^{-1}w^{-1}} 
=\prod_{\alpha''\in \bm\Psi_{\widehat{\alpha}}^+\cap w^{-1}\bm\Psi_{\widehat{P}}^-}
\hspace{-0,25cm}
(H^1\cap \U_{\alpha''})^{\tau_{\alpha}^{-1}w^{-1}} 
=\prod_{\alpha'\in w \bm\Psi_{\widehat{\alpha}}^+\cap \bm\Psi_{\widehat{P}}^-}
\hspace{-0,25cm}
(\mathbb{I}_{m'}+\varpi^{-1}\mathfrak{H}^1)\cap \U_{\alpha'}
$$
which is $(\mathbb{I}_{m'}+\varpi^{-1}\mathfrak{H}^1)\cap w \U_{\widehat\alpha}^+w^{-1}\cap \U_{\widehat{P}}^-.$

\begin{lemma}\label{lemma:V}
The group $wU_{\widehat\alpha}^+w^{-1}\cap U_{\widehat{P}}^-$ is in $\V(w,\alpha)$, it normalizes $\V(w,\alpha)\cap J^1_\P$ and 
$$(wU_{\widehat\alpha}^+w^{-1}\cap U_{\widehat{P}}^-)\cap (\V(w,\alpha)\cap J^1_\P)= wU_{\widehat\alpha}^+w^{-1}\cap U_{\widehat{P}}^-\cap K^1_B.$$
\end{lemma}

\begin{proof}
We recall that 
$wU_{\widehat\alpha}^+w^{-1}\cap U_{\widehat{P}}^-=w \U_{\widehat\alpha}^+w^{-1}\cap \U_{\widehat{P}}^-\cap K_B$ by remark \ref{rmk:UintersectK}.
Since $U_{\alpha'}=\tau_\alpha (K^1_B\cap U_{\alpha'})\tau_\alpha^{-1}$ for every $\alpha'\in \bm\Psi_{\widehat{\alpha}}^+$ (see lemma 2.9 of \cite{Chin1}), then we have 
$wU_{\widehat\alpha}^+w^{-1}\cap U_{\widehat{P}}^-
=(K^1_B\cap wU_{\widehat\alpha}^+w^{-1}\cap U_{\widehat{P}}^-)^{w\tau_\alpha^{-1}w^{-1}}$ which is in 
$\V(w,\alpha)$.
Moreover, the group  $wU_{\widehat\alpha}^+w^{-1}\cap U_{\widehat{P}}^-$ normalizes $\V(w,\alpha)\cap J^1_\P=\V(w,\alpha)\cap H^1$ because we have 
$wU_{\widehat\alpha}^+w^{-1}\cap U_{\widehat{P}}^- \subset K_B$ and $K_B$ normalizes $H^1$. 
Finally, since $K_B\cap H^1=K_B^1$, we have
$wU_{\widehat\alpha}^+w^{-1}\cap U_{\widehat{P}}^-\cap \V(w,\alpha)\cap J^1_\P
=wU_{\widehat\alpha}^+w^{-1}\cap U_{\widehat{P}}^-\cap H^1
= wU_{\widehat\alpha}^+w^{-1}\cap U_{\widehat{P}}^-\cap K_B\cap H^1
=wU_{\widehat\alpha}^+w^{-1}\cap U_{\widehat{P}}^-\cap K_B^1$.
\end{proof}

\noindent
By lemma \ref{lemma:V} the group 
$\V'=(wU_{\widehat\alpha}^+w^{-1}\cap U_{\widehat{P}}^-)(\V(w,\alpha)\cap J^1_\P)$ is a subgroup of  
$\V(w,\alpha)$.
We set $$d(w,\alpha)=[\V(w,\alpha):\V']\in R$$
which is non-zero because it is a power of $p$.

\begin{rmk}\label{rmk:index2} 
We have $\V(w,\alpha)\cap J^1_\P=H^1\cap w \U_{\widehat\alpha}^+w^{-1}\cap \U_{\widehat{P}}^-=\prod_{\alpha'\in w \bm\Psi_{\widehat{\alpha}}^+\cap \bm\Psi_{\widehat{P}}^-}H^1\cap \U_{\alpha'}$.
Hence, by remarks \ref{rmk:index} and \ref{rmk:roots} we have 
$[\V(w,\alpha):\V(w,\alpha)\cap J^1_\P]=
[\varpi^{-1}\mathfrak{H}^1_0:\mathfrak{H}^1_0]^{\ell(w)}=
\delta(\mathfrak{H}^1_0,\varpi \mathfrak{H}^1_0)^{\ell(w)}$.
On the other hand we have
$[\V(w,\alpha):\V(w,\alpha)\cap J^1_\P]
=d(w,\alpha)\; [\V':\V(w,\alpha)\cap J^1_\P]$
which is equal to 
$d(w,\alpha)\;
[(wU^+w^{-1}\cap U^-)(\V(w,\alpha)\cap J^1_\P):\V(w,\alpha)\cap J^1_\P]$
by remark \ref{rmk:roots}
and so to
$d(w,\alpha) [wUw^{-1}\cap U^-: wUw^{-1}\cap U^-\cap K_B^1]=d(w,\alpha) q^{\ell(w)}$
where $q$ is the cardinality of $\frack_{D'}$.
So, if we denote 
$\partial=\delta(\mathfrak{H}^1_0,\varpi\mathfrak{H}^1_0)/q\in R^{\times}$ then 
$d(w,\alpha)=\partial^{\ell(w)}$.
\end{rmk}

\begin{lemma}\label{lemma:supportrel9}
We have $(J_\P^1)^{\tau_{P}}J_\P^1 \cap (J_\P^1)^{w\tau_{\alpha}^{-1}w^{-1}}J_\P^1
=\V(w,\alpha)J_\P^1.$
\end{lemma}

\begin{proof}
We have 
$(J^1_\P)^{w\tau_\alpha^{-1}w^{-1}}=
(H^1\cap w^{-1}\U^-w)^{\tau_\alpha^{-1}w^{-1}} (J^1_\M)^{w\tau_\alpha^{-1}w^{-1}}(J^1\cap w^{-1}\U w)^{\tau_\alpha^{-1}w^{-1}}$. Now we consider the decompositions 
$H^1\cap w^{-1}\U^-w=(H^1\cap w^{-1}\U^-w\cap\U)(H^1\cap w^{-1}\U^-w\cap\U^-)$ and 
$J^1\cap w^{-1}\U w=(J^1\cap w^{-1}\U w\cap\U^-)(J^1\cap w^{-1}\U w\cap\U)$. 
By lemma \ref{lemma:roots} we have 
$J^1\cap w^{-1}\U w\cap\U^-=J^1\cap w^{-1}\U w\cap\U^-_{\widehat\alpha}$ and so by lemma \ref{lemma:J1Ptau} we obtain
$(J^1\cap w^{-1}\U w\cap\U^-)^{\tau_\alpha^{-1}w^{-1}}\subset
(J^1\cap \U^-_{\widehat\alpha})^{\tau_\alpha^{-1}w^{-1}}\subset(H^1\cap \U^-_{\widehat\alpha})^{w^{-1}}\subset J^1_\P$
and $(H^1\cap w^{-1}\U^-w\cap\U^-)^{\tau_\alpha^{-1}w^{-1}}\subset (H^1\cap\U^-)^{\tau_\alpha^{-1}w^{-1}}\subset(H^1\cap\U^-)^{w^{-1}}\subset J^1_\P$. 
Then, since  $(J^1_\M)^{w\tau_\alpha^{-1}w^{-1}}=J^1_\M$ by lemma \ref{lemma:J1Ptau}
and since $(H^1\cap \U^-\cap w\U w^{-1})^{w\tau_\alpha^{-1}w^{-1}}=\V(w,\alpha)$, we obtain 
$(J^1_\P)^{w\tau_\alpha^{-1}w^{-1}}\subset
\V(w,\alpha) J^1_\P (J^1\cap \U \cap w\U w^{-1})^{w\tau_\alpha^{-1}w^{-1}}.$
By lemma \ref{lemma:J1Ptau:2} and by previous calculations we have 
\[
(J_\P^1)^{\tau_{P}}J_\P^1 \cap (J_\P^1)^{w\tau_{\alpha}^{-1}w^{-1}}J_\P^1= \Big((J_\P^1\cap \U_{\widehat{P}}^-)^{\tau_{P}} \cap \V(w,\alpha) J^1_\P (J^1\cap \U \cap w\U w^{-1})^{w\tau_\alpha^{-1}w^{-1}} J^1_\P \Big)J_\P^1.
\]
Now, since $w\tau_{\alpha}^{-1}w^{-1}=\tau_Q^{-1}\tau_P$, the group $\V(w,\alpha)$ is contained both in  
$(\U^-_{\widehat{P}})^{\tau_Q^{-1}\tau_P}=
(\U^-_{\widehat{P}})^{\tau_P}$ and in 
$(J^1_\P\cap \U^-)^{\tau_Q^{-1}\tau_P}\subset (J^1_\P\cap \U^-)^{\tau_P} \subset (J^1_\P)^{\tau_P}$ by lemma \ref{lemma:J1Ptau}. This implies 
$\V(w,\alpha)\subset (J^1_\P\cap \U^-_{\widehat{P}})^{\tau_P}$
and so 
$(J_\P^1)^{\tau_{P}}J_\P^1 \cap (J_\P^1)^{w\tau_{\alpha}^{-1}w^{-1}}J_\P^1= \V(w,\alpha)\Big((J_\P^1\cap \U_{\widehat{P}}^-)^{\tau_{P}} \cap  J^1_\P (J^1\cap \U \cap w\U w^{-1})^{w\tau_\alpha^{-1}w^{-1}} J^1_\P \Big)J_\P^1.$
Now we have $(J_\P^1\cap \U_{\widehat{P}}^-)^{\tau_{P}} \cap  J^1_\P (J^1\cap \U \cap w\U w^{-1})^{w\tau_\alpha^{-1}w^{-1}} J^1_\P\subset \U^-\cap J^1_\P\U J^1_\P$ that is in $J^1_\P$ by lemma \ref{lemma:J1PU}.
\end{proof}

\subsection{The group $\widetilde{W}$}
In this paragraph we use a presentation by generators and relations of $\widetilde{W}$ to find a subgroup of $\Aut_R(V_\M)$ isomorphic to a quotient of $\widetilde{W}$.

\begin{rmk}\label{rmk:presentationW}
We know that the Iwahori-Hecke algebra (see 3.14 of \cite{Vig2}) is a deformation of the $R$-algebra $R[\widetilde{W}]$ and so it is not difficult to show that $\widetilde W$ is the group generated by $s_1,\dots, s_{m'-1}$ and $\tau_{m'-1}$ subject to relations $s_is_j=s_js_i$ for every $i,j$ such that $|i-j|>1$, $s_is_{i+1}s_i=s_{i+1}s_is_{i+1}$ for every $i\neq m'-1$, $s_i^2=1$ for every $i$, $\tau_{m'-1} s_i=s_i\tau_{m'-1}$ for every $i\neq m'-1$ and $\tau_{m'-1} s_{m'-1}\tau_{m'-1} s_{m'-1}=s_{m'-1}\tau_{m'-1} s_{m'-1}\tau_{m'-1}$.
\end{rmk}

\begin{lemma}\label{lemma:wtauw}
Let $i\in\{1,\dots,m'-1\}$, $\alpha=\alpha_{i,i+1}$, $w\in W$ of minimal length in $wW_{\widehat{\alpha}}$ and $\Phi\in\mathscr{H}_R(G,\eta_\P)_{J^1_\P \tau_i J^1_\P}$. 
Then the support of $\widehat{f}_w\Phi\widehat{f}_{w^{-1}}$ is in $J^1_\P w\tau_iw^{-1}J^1_\P$ and 
$$(\widehat{f}_w\Phi\widehat{f}_{w^{-1}})(w\tau_iw^{-1})=\delta(\mathfrak{J}^1_0,\mathfrak{H}^1_0)^{\ell(w)}\widehat{f}_w(w)\circ \Phi(\tau_i) \circ \widehat{f}_{w^{-1}}(w^{-1}).$$
\end{lemma}

\begin{proof}
Since $w$ and $w^{-1}$ normalize $J^1$, by lemma \ref{lemma:product} the support of $\widehat{f}_w\Phi\widehat{f}_{w^{-1}}$ is in $J^1_\P w\tau_iw^{-1}J^1_\P$. 
We recall that 
$$(\widehat{f}_w\Phi\widehat{f}_{w^{-1}})(w\tau_iw^{-1})=\sum_{x\in G/J^1_\P}(\widehat{f}_w\Phi)(w\tau_i x)\widehat{f}_{w^{-1}}(x^{-1}w^{-1}).$$
By lemma \ref{lemma:J1Pw}, the support of $x\mapsto (\widehat{f}_w\Phi)(w\tau_ix)
\widehat{f}_{w^{-1}}(x^{-1}w^{-1})$ is in 
$(J^1_\P)^{w\tau_i}J^1_\P\cap(J^1_\P)^{w}J^1_\P
=(J^1_\P)^{w\tau_i}J^1_\P\cap(J^1\cap\U^w\cap\U^-)J^1_\P 
.$
Since $w$ is of minimal length in $wW_{\widehat{\alpha}}$, by lemma \ref{lemma:roots} we have $J^1\cap\U^w\cap\U^-=J^1\cap\U^w\cap\U^-_{\widehat{\alpha}}$ which is included in 
$(J^1_\P)^{w\tau_i}$ because 
$(J^1\cap\U^w\cap\U^-_{\widehat{\alpha}})^{\tau_i^{-1}w^{-1}}
=((J^1\cap\U^-_{\widehat{\alpha}})^{\tau_i^{-1}} \cap\U^w)^{w^{-1}}$ that by lemma \ref{lemma:J1Ptau} is included in 
$(H^1\cap\U^-_{\widehat{\alpha}})^{w^{-1}} \cap\U$ and so in $J^1_\P$. 
Hence, we obtain $(J^1_\P)^{w\tau_i}J^1_\P\cap(J^1_\P)^{w}J^1_\P=(J^1\cap\U^w\cap\U^-)J^1_\P$.
Now since $(J^1\cap\U^w\cap\U^-)^{w^{-1}}$ and $(J^1\cap\U^w\cap\U^-)^{\tau_i^{-1}w^{-1}}$ are contained in $J^1\cap\U$ and so in the kernel of $\eta_\P$ and since we have
$[(J^1\cap\U^w\cap\U^-)J^1_\P:J^1_\P]
=[J^1\cap\U^w\cap\U^-:H^1\cap\U^w\cap\U^-]=\delta(\mathfrak{J}^1_0,\mathfrak{H}^1_0)^{\ell(w)}$
 we obtain
 $(\widehat{f}_w\Phi\widehat{f}_{w^{-1}})(w\tau_i w^{-1})=\delta(\mathfrak{J}^1_0,\mathfrak{H}^1_0)^{\ell(w)}(\widehat{f}_w\Phi)(w\tau_i)\circ \widehat{f}_{w^{-1}}(w^{-1})$. 
To conclude we observe that by lemma \ref{lemma:product} the support of $\widehat{f}_w\Phi$ is contained in $J^1_\P w\tau_iJ^1_\P$ and by lemmas \ref{lemma:J1Ptau:2} and \ref{lemma:J1Pw}
the support of $x\mapsto (\widehat{f}_w)(wx)
\Phi(x^{-1}\tau_i)$ is in 
$(J^1_\P)^wJ^1_\P\cap (J^1_\P)^{\tau_i^{-1}}J^1_\P
=(J^1\cap\U^w\cap\U^-)J^1_\P\cap (J^1_\P\cap \U_{P(\tau_i)}^+)^{\tau_i^{-1}}J^1_\P$
that is contained in 
$(\U J^1_\P \cap\U^-)J^1_\P=J^1_\P$ by lemma \ref{lemma:J1PU}. Hence, $(\widehat{f}_w\Phi)(w\tau_i)=\widehat{f}_w(w)\circ \Phi(\tau_i)$.
\end{proof}

\begin{lemma}\label{lemma:kappa}
Let $w\in W$ and $\alpha\in\Sigma$. Then
\begin{equation*}
\mathsf{p}\circ\kappa(w)\circ\iota\circ\mathsf{p}\circ\kappa(s_\alpha)\circ\iota=\left\{
\begin{array}{ll}
\mathsf{p}\circ\kappa(ws_\alpha)\circ\iota & \text{if } w\alpha>0\\
\delta(\mathfrak{J}^1_0,\mathfrak{H}^1_0)^{-1} \mathsf{p}\circ\kappa(ws_\alpha)\circ\iota & \text{if } w\alpha<0.
\end{array}
\right.
\end{equation*}
\end{lemma}

\begin{proof}
By lemma \ref{lemma:finitealgebras} we have $\widehat{f}_w\widehat{f}_{s_\alpha}=\widehat{f}_{ws_\alpha}$ and then  
$(\widehat{f}_w\widehat{f}_{s_\alpha})(ws_\alpha)=\mathsf{p}\circ\kappa(ws_\alpha)\circ\iota$.
On the other hand we have
$$(\widehat{f}_w\widehat{f}_{s_\alpha})(ws_\alpha)=
\sum_{x\in G/J^1_\P}(\widehat{f}_w)(wx)\widehat{f}_{s_\alpha}(x^{-1}s_\alpha).$$
By lemma \ref{lemma:J1Pw} the support of 
$x\mapsto \widehat{f}_{w}(wx)\widehat{f}_{s_\alpha}(x^{-1}
s_\alpha)$ is contained in 
$(J^1_\P)^wJ^1_\P \cap (J^1_\P)^{s_\alpha}J^1_\P
=(J^1_\P)^wJ^1_\P \cap (J^1\cap \U^{s_\alpha}\cap \U^{-1})J^1_\P
=\big((J^1_\P)^wJ^1_\P \cap J^1\cap \U_{-\alpha}\big)J^1_\P
$ which is equal to $J^1_\P$ if $w(-\alpha)<0$ and to $(J^1\cap\U_{-\alpha})J^1_\P$ if $w(-\alpha)>0$.
Hence, if $w\alpha>0$ we obtain 
$(\widehat{f}_w\widehat{f}_{s_\alpha})(ws_\alpha)=\mathsf{p}\circ\kappa(w)\circ\iota\circ\mathsf{p}\circ\kappa(s_\alpha)\circ\iota$ while if $w\alpha<0$ since $(J^1\cap\U_{-\alpha})^{w^{-1}}$ and 
$(J^1\cap\U_{-\alpha})^{s_\alpha}$ are contained in $J^1\cap\U$ and so in the kernel of $\eta_\P$ and since $[(J^1\cap\U_{-\alpha})J^1_\P:J^1_\P]=[J^1\cap\U_{-\alpha}:H^1\cap \U_{-\alpha}]=\delta(\mathfrak{J}^1_0,\mathfrak{H}^1_0)$ we obtain 
$(\widehat{f}_w\widehat{f}_{s_\alpha})(ws_\alpha)=\delta(\mathfrak{J}^1_0,\mathfrak{H}^1_0)\;\mathsf{p}\circ\kappa(w)\circ\iota\circ\mathsf{p}\circ\kappa(s_\alpha)\circ\iota$.
\end{proof}

From now on we fix a non-zero element $\gamma\in I_{\tau_{m'-1}}(\eta_\P)$ which is invertible by lemma \ref{lemma:intertetaP} and we consider the function $\widehat f_{\tau_{m'-1}}\in\mathscr{H}_R(G,\eta_\P)_{J^1_\P \tau_{m'-1}J^1_\P}$ defined by $\widehat f_{\tau_{m'-1}}(j_1\tau_{m'-1}j_2)=\eta_\P(j_1)\circ\gamma\circ\eta_\P(j_2)$ for every $j_1,j_2\in J^1_\P$.

\smallskip
From now on we fix a square root $\delta(\mathfrak{J}^1_0,\mathfrak{H}^1_0)^{1/2}$ of $\delta(\mathfrak{J}^1_0,\mathfrak{H}^1_0)$ in $R$.
We consider the subgroup $\widetilde{\mathcal{W}}$ of $\Aut_R(V_\M)$ generated by $\gamma$ and by $\delta(\mathfrak{J}^1_0,\mathfrak{H}^1_0)^{1/2}\mathsf{p}\circ\kappa(s_i)\circ\iota$ with $i\in\{1,\dots,m'-1\}$.

\begin{lemma}\label{lemma:isomWtilde}
The function that maps $s_i$ to $\delta(\mathfrak{J}^1_0,\mathfrak{H}^1_0)^{1/2}\mathsf{p}\circ\kappa(s_i)\circ\iota$ for every $i\in\{1,\dots,m'-1\}$ and $\tau_{m'-1}$ to $\gamma$ extends to a surjective group homomorphism 
$\varepsilon:\widetilde{W}\rightarrow \widetilde{\mathcal W}$.
\end{lemma}

\begin{proof}
Let $\delta=\delta(\mathfrak{J}^1_0,\mathfrak{H}^1_0)$.
To prove that $\varepsilon$ is a group homomorphism we use the presentation of $\widetilde W$ given in remark \ref{rmk:presentationW}.
For every $i,j\in\{1,\dots,m'-1\}$ such that $|i-j|>1$ we have 
$\varepsilon(s_i)\varepsilon(s_j)=\delta\;\mathsf{p}\circ\kappa(s_i)\circ\iota\circ\mathsf{p}\circ\kappa(s_j)\circ\iota$ that by lemma \ref{lemma:kappa} is equal to 
$\delta\;\mathsf{p}\circ\kappa(s_is_j)\circ\iota
=\delta\;\mathsf{p}\circ\kappa(s_js_i)\circ\iota=\varepsilon(s_j)\varepsilon(s_i)$.
For every $i\neq m'-1$ we have 
$\varepsilon(s_i)\varepsilon(s_{i+1})\varepsilon(s_i)=
\delta^{3/2}\;\mathsf{p}\circ\kappa(s_i)\circ\iota\circ\mathsf{p}\circ\kappa(s_{i+1})\circ\iota\circ\mathsf{p}\circ\kappa(s_i)\circ\iota$ that by lemma \ref{lemma:kappa} is equal to 
$\delta^{3/2}\;\mathsf{p}\circ\kappa(s_is_{i+1}s_i)\circ\iota
=\delta^{3/2}\;\mathsf{p}\circ\kappa(s_{i+1}s_is_{i+1})\circ\iota=\varepsilon(s_{i+1})\varepsilon(s_i)\varepsilon(s_{i+1})$.
For every $i$ we have $\varepsilon(s_i)^2=\delta\;\mathsf{p}\circ\kappa(s_i)\circ\iota\circ\mathsf{p}\circ\kappa(s_i)\circ\iota$ that by lemma \ref{lemma:kappa} is equal to 
$\mathsf{p}\circ\kappa(s_is_i)\circ\iota$ that is the identity of $\Aut_R(V_\M)$.
Let $\tau=\tau_{m'-1}$ and $\widehat{f}_{\tau}=\widehat{f}_{\tau_{m'-1}}$.
For every $i\neq m'-1$ we have
$\varepsilon(\tau)\varepsilon(s_i)=\delta^{1/2}\gamma\circ\mathsf{p}\circ\kappa(s_i)\circ\iota$
that is equal to $\delta^{1/2}(\widehat{f}_{\tau}\widehat{f}_{s_i})(\tau s_i)$ since the support of 
$x\mapsto \widehat{f}_{\tau}(\tau x)\widehat{f}_{s_i}(x^{-1}s_i)$ is contained in $(J^1_\P)^{\tau}J^1_\P\cap (J^1_\P)^{s_i}J^1_\P=
((J^1_\P\cap \U_{P(\tau)}^-)^{\tau}J^1_\P\cap J^1_\P\cap \U_{\alpha_{i+1,i}})J^1_\P=J^1_\P$. 
Hence, by lemma \ref{lemma:product} we have
$\varepsilon(\tau)\varepsilon(s_i)=\delta^{1/2}\mathsf{p}\circ\zeta^{-1}(\widehat{f}_\tau)(\tau)\circ\kappa(s_i)\circ\iota$. 
Since $\zeta^{-1}(\widehat{f}_\tau)(\tau)\in I_{\tau}(\eta)=I_{\tau}(\kappa)$ and $s_i\in J\cap J^{\tau}$ we obtain $\varepsilon(\tau)\varepsilon(s_i)=\delta^{1/2}\mathsf{p}\circ\kappa(s_i)\circ\zeta^{-1}(\widehat{f}_\tau)(\tau)\circ\iota
=\delta^{1/2}(\widehat{f}_{s_i}\widehat{f}_{\tau})(s_i\tau)
$ that is equal to 
$\delta^{1/2}\mathsf{p}\circ\kappa(s_i)\circ\iota\circ\gamma=\varepsilon(s_i)\varepsilon(\tau)$ since 
the support of 
$x\mapsto \widehat{f}_{s_i}(s_i x)\widehat{f}_{\tau}(x^{-1}\tau)$ is contained in $(J^1_\P)^{s_i}J^1_\P\cap (J^1_\P)^{\tau^{-1}}J^1_\P=
(J^1_\P\cap \U_{\alpha_{i+1,i}}\cap(J^1_\P\cap \U_{P(\tau)}^+)^{\tau^{-1}}J^1_\P)J^1_\P=J^1_\P$. 
It remains to prove the last relation.
Let $s=s_{m'-1}$. 
Then $\tau s\tau s=\tau_{m'-2}=s\tau s\tau$ and
by lemma \ref{lemma:product} we have 
$(\widehat{f}_{\tau}\widehat{f}_{s}\widehat{f}_{\tau}\widehat{f}_{s})(\tau s\tau s)=
\mathsf{p}\circ\zeta^{-1}(\widehat{f}_{\tau}\widehat{f}_{s}\widehat{f}_{\tau})(\tau s \tau)\circ \kappa(s)\circ \iota$.
Now, since $\zeta^{-1}(\widehat{f}_{\tau}\widehat{f}_{s}\widehat{f}_{\tau})(\tau s \tau)\in I_{\tau s \tau}(\kappa)$ and since $s=s^{\tau s\tau}\in J\cap J^{\tau s \tau}$, we obtain  
$(\widehat{f}_{\tau}\widehat{f}_{s}\widehat{f}_{\tau}\widehat{f}_{s})(\tau_{m'-2})=\mathsf{p}\circ \kappa(s)\circ\zeta^{-1}(\widehat{f}_{\tau}\widehat{f}_{s}\widehat{f}_{\tau})(\tau s \tau)\circ \iota
=(\widehat{f}_{s}\widehat{f}_{\tau}\widehat{f}_{s}\widehat{f}_{\tau})(\tau_{m'-2})$.
On the other hand we have 
$$(\widehat{f}_{\tau}\widehat{f}_{s}\widehat{f}_{\tau}\widehat{f}_{s})(\tau_{m'-2})=(\widehat{f}_{\tau}\widehat{f}_{s}\widehat{f}_{\tau}\widehat{f}_{s})(\tau s\tau s)=\sum_{x\in G/J^1_\P}
\widehat{f}_{\tau}(\tau x)(\widehat{f}_{s}\widehat{f}_{\tau}\widehat{f}_{s})(x^{-1}s\tau s).$$
The support of 
$x\mapsto \widehat{f}_{\tau}(\tau x)(\widehat{f}_{s}\widehat{f}_{\tau}\widehat{f}_{s})(x^{-1}s\tau s)$ is in 
$(H^1\cap\U_{\alpha'})^\tau J^1_\P$ with $\alpha'=\alpha_{m',m'-1}$ by lemma \ref{lemma:supportrel9}.
For every $x\in (H^1\cap\U_{\alpha'})^\tau$ the elements $x^{\tau^{-1}}$ and $(x^{-1})^{s\tau s}$ are in $H^1\cap \U$ and so in the kernel of $\eta_\P$. 
Then 
$(\widehat{f}_{\tau}\widehat{f}_{s}\widehat{f}_{\tau}\widehat{f}_{s})(\tau_{m'-2})=(\widehat{f}_{s}\widehat{f}_{\tau}\widehat{f}_{s}\widehat{f}_{\tau})(\tau_{m'-2})$ is equal to 
$\delta(\mathfrak{H}^1_0,\varpi \mathfrak{H}^1_0) \;\gamma \circ (\widehat{f}_{s}\widehat{f}_{\tau}\widehat{f}_{s})(s\tau s)$ and by lemma \ref{lemma:wtauw} it is also equal to 
$\delta(\mathfrak{H}^1_0,\varpi \mathfrak{H}^1_0)\;\varepsilon(\tau)\varepsilon(s)\varepsilon(\tau)\varepsilon(s)$.
Now, if $\alpha''=\alpha_{m'-2,m'-1}$ then  
$\alpha'\notin \bm\Psi_{\widehat{\alpha''}}^+\cup\bm\Psi_{\widehat{\alpha''}}^- $ and so we have
$(J^1_\P)^sJ^1_\P\cap(J^1_\P)^{\tau_{m'-2}^{-1}}J^1_\P=J^1_\P=(J^1_\P)^{\tau_{m'-2}}J^1_\P\cap(J^1_\P)^{s}J^1_\P$. 
Hence, $(\widehat{f}_{s}\widehat{f}_{\tau}\widehat{f}_{s}\widehat{f}_{\tau}\widehat{f}_{s})(s\tau s\tau s)$ is equal both to 
$$\widehat{f}_{s}(s)\circ(\widehat{f}_{\tau}\widehat{f}_{s}\widehat{f}_{\tau}\widehat{f}_{s})(\tau_{m'-2})
=\delta(\mathfrak{H}^1_0,\varpi \mathfrak{H}^1_0)\delta(\mathfrak{J}^1_0,\mathfrak{H}^1_0)^{-1/2}\varepsilon(s)\varepsilon(\tau)\varepsilon(s)\varepsilon(\tau)\varepsilon(s)$$ and also to 
\begin{align*}
(\widehat{f}_{s}\widehat{f}_{\tau}\widehat{f}_{s}\widehat{f}_{\tau})(\tau_{m'-2})\circ \widehat{f}_{s}(s)
&=\delta(\mathfrak{H}^1_0,\varpi \mathfrak{H}^1_0)\;\delta(\mathfrak{J}^1_0,\mathfrak{H}^1_0)^{-1/2}
\varepsilon(\tau)\varepsilon(s)\varepsilon(\tau)\varepsilon(s)^2\\
&=\delta(\mathfrak{H}^1_0,\varpi \mathfrak{H}^1_0)\;\delta(\mathfrak{J}^1_0,\mathfrak{H}^1_0)^{-1/2}
\varepsilon(\tau)\varepsilon(s)\varepsilon(\tau).
\end{align*}
This implies $\varepsilon(\tau)\varepsilon(s)\varepsilon(\tau)\varepsilon(s)=\varepsilon(s)\varepsilon(\tau)\varepsilon(s)\varepsilon(\tau)$ since both $\delta(\mathfrak{H}^1_0,\varpi \mathfrak{H}^1_0)$ and $\delta^{-1/2}$ are invertible in $R$.
We conclude that $\varepsilon$ is a group homomorphism and it is clearly surjective.
\end{proof}

\begin{rmk}\label{rmk:xiw}
For every $w\in W$ we have 
$\varepsilon(w)=\delta(\mathfrak{J}^1_0,\mathfrak{H}^1_0)^{\ell(w)/2}\mathsf{p}\circ\kappa(w)\circ\iota$.
\end{rmk}

\begin{lemma}\label{lemma:xibijective}
For every $\widetilde{w}\in\widetilde{W}$ we have 
$\varepsilon(\widetilde{w})\in I_{\widetilde w}(\eta_\P)$.
\end{lemma}

\begin{proof}
Since $\eta_\M$ is the restriction of $\eta_\P$ to the group $J^1_\M$, we have 
$\varepsilon(w)=\delta(\mathfrak{J}^1_0,\mathfrak{H}^1_0)^{\ell(w)/2}\widehat{f}_w(w)\in I_w(\eta_\M)$ for every $w\in W$ and $\gamma\in  I_{\tau_{m'-1}}(\eta_\M)$. 
Then, since every $w\in W$ and $\tau_{m'-1}$ normalize $J^1_\M$, we have 
$\varepsilon(\widetilde{w})\in I_{\widetilde w}(\eta_\M)$ for every $\widetilde{w}\in\widetilde{W}$ and so $\varepsilon(\widetilde{w})\in I_{\widetilde w}(\eta_\P)$ by lemma \ref{lemma:intertetaP}.
\end{proof}

\begin{lemma}\label{lemma:gammacommute}
For every $\tau',\tau''\in\bm\Delta$, 
$\gamma'\in I_{\tau'}(\eta_\P)$ and $\gamma''\in I_{\tau''}(\eta_\P)$ we have $\gamma'\circ\gamma''=\gamma''\circ\gamma'$.
\end{lemma}

\begin{proof}
We recall that $I_{\tau}(\eta_\P)$ is $1$-dimensional for every $\tau\in\bm\Delta$ and so there exist $c',c''\in R$ such that $\gamma'=c'\varepsilon(\tau')$ and $\gamma''=c''\varepsilon(\tau'')$. We obtain 
$\gamma'\circ\gamma''=c'c''\varepsilon(\tau')\circ\varepsilon(\tau'')=
c'c''\varepsilon(\tau'\tau'')=c'c''\varepsilon(\tau''\tau')=\gamma''\circ\gamma'$.
\end{proof}

\subsection{The isomorphisms $\mathscr{H}_R(G,\eta_P)\cong\mathscr{H}_R(B^{\times},K^1_B)$}
In this paragraph we define elements $\widehat f_{\tau_{i}}\in\mathscr{H}_R(G,\eta_\P)_{J^1_\P \tau_{i}J^1_\P}$ for every $i\in\{0,\dots,m'-1\}$ and we prove that $\widehat f_\omega$ with $\omega\in\Omega$ respect relations of lemma \ref{lemma:relations} obtaining an algebra homomorphism from $\mathscr{H}_{R}(B^{\times},K^1_B)$ to $\mathscr{H}_R(G,\eta_\P)$.

\smallskip
For every $i\in\{0,\dots,m'-1\}$ we denote
 $\gamma_i=\partial^{(m'-i)(m'-i-1)/2}\varepsilon(\tau_i)$ 
where $\partial$ is the power of $p$ defined in remark \ref{rmk:index2}. 
Then $\gamma_i$ is an invertible element in $I_{\tau_i}(\eta_\P)$ and $\gamma_{m'-1}=\gamma$. 

\begin{lemma}\label{lemma:gammai}
We have $\gamma_{i-1}\circ\gamma_i^{-1}=\partial^{m'-i}\varepsilon(\tau_{i-1}\tau_i^{-1})$ and $\gamma_i=\prod_{h={i+1}}^{m'}\partial^{m'-h}\varepsilon(\tau_i)$ for every $i\in\{1,\dots,m'-1\}$.
\end{lemma}

\begin{proof}
Since $((m'-(i-1))(m'-(i-1)-1)-(m'-i)(m'-i-1))/2=m'-i$ we have $\gamma_{i-1}\circ\gamma_i^{-1}=\partial^{m'-i} \varepsilon(\tau_{i-1})\varepsilon(\tau_{i})^{-1}=\partial^{m'-i}\varepsilon(\tau_{i-1}\tau_i^{-1})$.
The second statement is true because $\sum_{h=i+1}^{m'}m'-h=\sum_{j=0}^{m'-i-1}j=(m'-i)(m'-i-1)/2$.
\end{proof}

For every $i\in\{0,\dots,m'-1\}$ we consider the function 
$\widehat f_{\tau_{i}}\in\mathscr{H}_R(G,\eta_\P)_{J^1_\P \tau_{i}J^1_\P}$ defined by $\widehat f_{\tau_{i}}(j_1\tau_{i}j_2)=\eta_\P(j_1)\circ\gamma_i\circ\eta_\P(j_2)$ for every $j_1,j_2\in J^1_\P$.
We remark that in general $\widehat f_{\tau_{i}}$ is not invertible but since $\tau_0$ normalizes $J^1_\P$ the function
 $\widehat{f}_{\tau_0}$ is invertible in  
 $\mathscr{H}_R(G,\eta_\P)$ with inverse 
$\widehat{f}_{\tau_0^{-1}}: \tau_0^{-1}J^1_\P\rightarrow \End_R(V_\M)$ defined by
$\widehat{f}_{\tau_0^{-1}}(\tau_0^{-1}j)=\gamma_0^{-1}\circ\eta_\P(j)$ for every $j\in J^1_\P$. 

\begin{lemma}\label{lemma:isomHecke}
The map 
$\Theta'':\bm\Omega  \rightarrow \mathscr{H}_R(G,\eta_P)$
given by $f_\omega  \mapsto  \widehat f_\omega$ for every 
$f_\omega\in\bm\Omega$
is well-defined.
\end{lemma}

\begin{proof}
The map is well-defined on $f_k$ with $k\in K_B$ because 
$\Theta'$ is a homomorphism and it is well-defined on $\tau_i$ with  $i\in\{0,\dots,m'-1\}$ because 
$K^1_B\tau_iK^1_B=K^1_B\tau_jK^1_B$ implies $i=j$.
\end{proof}

\begin{lemma}\label{lemma:producttau2}
For every $i,j\in\{0,\dots,m'-1\}$ the function 
$\widehat f_{\tau_i}\widehat f_{\tau_j}$ is in  
$\mathscr{H}_R(G,\eta_\P)_{J^1_\P\tau_i\tau_j J^1_\P}$ and
$(\widehat f_{\tau_i}\widehat f_{\tau_j})(\tau_i\tau_j)=\gamma_i\circ\gamma_j$.
\end{lemma}

\begin{proof}
If $i$ or $j$ is $0$ then it follows by lemma \ref{lemma:product} since $\tau_0$ normalizes $J^1_\P$.
Otherwise, by lemma \ref{lemma:producttau} the support of 
$\widehat f_{\tau_i}\widehat f_{\tau_j}$ is in 
$J^1_\P\tau_iJ^1_\P\tau_jJ^1_\P=J^1_\P\tau_i\tau_jJ^1_\P$ and the support of 
$x\mapsto \widehat f_{\tau_i}(\tau_ix)\widehat f_{\tau_j}(x^{-1}\tau_j)$ is in $(J^1_\P)^{\tau_i} J^1_\P \cap (J^1_\P)^{\tau_j^{-1}}J^1_\P=J^1_\P$. We obtain 
$(\widehat f_{\tau_i}\widehat f_{\tau_j})(\tau_i\tau_j)=
\sum_{x\in G/J^1_\P}\widehat f_{\tau_i}(\tau_ix)\widehat f_{\tau_j}(x^{-1}\tau_j)=
\widehat f_{\tau_i}(\tau_i)\circ\widehat f_{\tau_j}(\tau_j)=\gamma_i\circ\gamma_j$. 
\end{proof}

By lemmas \ref{lemma:producttau2} and \ref{lemma:gammacommute} we obtain $\widehat f_{\tau_i}\widehat f_{\tau_j}=\widehat f_{\tau_j}\widehat f_{\tau_i}$ for every $i,j\in\{0,\dots,m'-1\}$. 
So, if $P\subset\{0,\dots,m'-1\}$ we denote by $\gamma_P$ the composition of $\gamma_i$ with $i\in P$, which is well-defined by lemma \ref{lemma:gammacommute}, and by 
$\widehat{f}_{\tau_P}$ the product of $\widehat{f}_{\tau_i}$ with $i\in P$, which is well-defined because the $\widehat f_{\tau_i}$ commute.
Furthermore, by lemma \ref{lemma:producttau} we obtain that the support of $\widehat{f}_{\tau_P}$ is $J^1_\P\tau_P J^1_\P$ and by lemma \ref{lemma:producttau2} we have $\widehat{f}_{\tau_P}(\tau_P)=\gamma_P$.

\begin{lemma}\label{lemma:rel347}
We have $\widehat f_{\tau_i}\widehat f_{x}=\widehat f_{\tau_i x\tau_{i}^{-1}}\widehat f_{\tau_i}$
for every $i\in\{0,\dots,m'-1\}$ and every 
$x\in M_{\widehat{\alpha_{i,i+1}}}=K_B\cap \M_{\widehat{\alpha_{i,i+1}}}$ if $i\neq 0$ or $x\in K_B$ if $i=0$.
\end{lemma}

\begin{proof}
Since $x$ normalizes $J^1$ by lemma \ref{lemma:product} 
the supports of $\widehat f_{\tau_i}\widehat f_{x}$ and of 
$\widehat f_{\tau_i x\tau_{i}^{-1}}\widehat f_{\tau_i}$ is in $J^1_\P \tau_i x J^1_\P$ and 
$(\widehat f_{\tau_i}\widehat f_{x})(\tau_i x)=
\mathsf{p}\circ \zeta^{-1}(\widehat f_{\tau_i})(\tau_i)\circ\kappa(x)\circ\iota$ that is equal to
$\mathsf{p}\circ \kappa(\tau_i x\tau_i^{-1})\circ\zeta^{-1}(\widehat f_{\tau_i})(\tau_i)\circ\iota
=(\widehat f_{\tau_i x\tau_{i}^{-1}}\widehat f_{\tau_i})(\tau_i x)$ because 
$\zeta^{-1}(\widehat f_{\tau_i})(\tau_i)\in I_{\tau_i}(\kappa)$ and $x\in J\cap J^{\tau_i}$.
\end{proof}

\begin{lemma}\label{lemma:rel56}
Let $i\in\{1,\dots, m'-1\}$ and $\alpha\in \mathbf{\Psi}_{\widehat{\alpha_{ii+1}}}^+$. Then for every 
$u\in U_{\alpha}$ and $u'\in U_{-\alpha}$ we have 
$\widehat f_{u}\widehat f_{\tau_i}=\widehat f_{\tau_i}$ and
$\widehat f_{\tau_i}\widehat f_{u'}=\widehat f_{\tau_i}$.
\end{lemma}

\begin{proof}
The elements $\tau_i^{-1}u\tau_i$ and $\tau_iu'\tau_i^{-1}$ are in $K^1_B\subset J^1_\P$ and so, since $u$ and $u'$ normalize $J^1$, by lemma \ref{lemma:product} the supports of $\widehat f_{u}\widehat f_{\tau_i}$ and of $\widehat f_{\tau_i}\widehat f_{u'}$ are in $J^1_\P u\tau_iJ^1_\P=J^1_\P\tau_i J^1_\P=J^1_\P\tau_i u'J^1_\P$.
Now since $\zeta^{-1}(\widehat f_{\tau_i})(\tau_i)\in I_{\tau_i}(\eta)=I_{\tau_i}(\kappa)$ and $u\in J\cap J^{\tau_i^{-1}}$, by lemma \ref{lemma:product} we  have 
$(\widehat f_{u}\widehat f_{\tau_i})(u\tau_i)=
\mathsf{p}\circ\kappa(u)\circ\zeta^{-1}(\widehat f_{\tau_i})(\tau_i)\circ\iota=
\mathsf{p}\circ\zeta^{-1}(\widehat f_{\tau_i})(\tau_i)\circ\eta(\tau_i^{-1}u\tau_i)\circ\iota$. 
By lemma \ref{lemma:etaP} we obtain 
$(\widehat f_{u}\widehat f_{\tau_i})(u\tau_i)=\mathsf{p}\circ\zeta^{-1}(\widehat f_{\tau_i})(\tau_i)\circ\iota\circ\eta_\P(\tau_i^{-1}u\tau_i)=\widehat f_{\tau_i}(\tau_i)\circ\eta_\P(\tau_i^{-1}u\tau_i)=\widehat f_{\tau_i}(u\tau_i)$. 
Similarly we have 
$\widehat f_{\tau_i}(\tau_iu')=
\mathsf{p}\circ\zeta^{-1}(\widehat f_{\tau_i})(\tau_i)\circ\kappa(u')\circ\iota=
\mathsf{p}\circ\eta(\tau_iu'\tau_i^{-1})\circ\zeta^{-1}(\widehat f_{\tau_i})(\tau_i)\circ\iota=
\eta_\P(\tau_iu'\tau_i^{-1})\circ\mathsf{p}\circ\zeta^{-1}(\widehat f_{\tau_i})(\tau_i)\circ\iota=
\eta_\P(\tau_iu'\tau_i^{-1})\circ\widehat f_{\tau_i}(\tau_i)=
\widehat f_{\tau_i}(\tau_iu')$. 
\end{proof}

We introduce some subgroup of $G$ through its identification with $GL_{m'}(A(E))$ in order to find the support of 
$\widehat f_{\tau_{P}}\widehat f_w\widehat f_{\tau_{\alpha}}\widehat f_{w^{-1}}$. We recall that $\mathfrak{A}(E)$ is the unique hereditary order normalized by $E^{\times}$ in $A(E)$ and $\mathfrak{P}(E)$ is its radical.
\begin{itemize}[$\bullet$]
\item Let $\mathcal{Z}$ be the set of matrices $(z_{ij})$ such that $z_{ii}=1$, $z_{ij}\in \varpi^{-1}\mathfrak{P}(E)$ if 
$i<j$ and $z_{ij}=0$ if $i>j$. 
\item Let $\V$ be the group  
$\displaystyle{(J^1 \cap w \U_{\widehat{\alpha}}^- w^{-1}\cap \U_{\widehat{P}}^+)^{w\tau_\alpha w^{-1}}=
\mathbb{I}_{m'}+\prod_{\alpha'\in w\bm\Psi_{\widehat{\alpha}}^-\cap\bm\Phi_{\widehat{P}}^+}(\varpi^{-1}\mathfrak{J}^1\cap \U_{\alpha'})\subset\mathcal{Z}.}$
We remark that it is different from $\V(w,\alpha)$ defined by (\ref{eq:defV}).
\item Let $\tilde{I}^1$ be the group of matrices 
$(m_{ij})$ such that $m_{ii}\in 1+\mathfrak{P}(E)$, $m_{ij}\in \mathfrak{A}(E)$ if 
$i<j$ and $m_{ij}\in\mathfrak{P}(E)$ if $i>j$. 
\item Let $\mathbf{W}=W\ltimes M$ be the subgroup of $B^{\times}$ of monomial matrices with coefficients in  $\ent_{D'}^{\times}$. Then $B^{\times}$ is the disjoint union of $I_{B}(1)wI_B(1)$ with $w\in \mathbf{W}$, where $I_B(1)=K^1U$ is the \emph{pro-$p$-Iwahori subgroup of $K_{B}$}, i.e. the pro-$p$-radical of $I_B$.
\end{itemize}

\begin{lemma}\label{lemma:supportrel9:2}
We have
$J^1_\P \tau_P J^1_\P w\tau_\alpha w^{-1}J^1_\P
=J^1_\P \tau_Q \mathcal{V}J^1_\P.$
\end{lemma}

\begin{proof}
We proceed similarly to the beginning of proof of lemma \ref{lemma:supportrel9}:
we can prove that
$J_\P^1 w\tau_{\alpha}w^{-1}J_\P^1=
(J^1_\P\cap w\U_{\widehat\alpha}^- w^{-1}) w\tau_\alpha w^{-1}J^1_\P$.
Now we consider the decomposition of the group $(J^1_\P\cap w\U_{\widehat{\alpha}}^-w^{-1})$ into the product
$(J^1_\P\cap w\U_{\widehat{\alpha}}^-w^{-1}\cap\U^-)
(J^1_\P\cap w\U_{\widehat{\alpha}}^-w^{-1}\cap\U)$.
By lemma \ref{lemma:J1Ptau} we have 
$(J^1_\P\cap w\U_{\widehat{\alpha}}^-w^{-1}\cap\U^-)^{\tau_P^{-1}}\subset J^1_\P$ and by lemma \ref{lemma:roots} we have 
$J^1_\P\cap w\U_{\widehat{\alpha}}^-w^{-1}\cap\U=
J^1_\P\cap w\U_{\widehat{\alpha}}^-w^{-1}\cap\U_{\widehat{P}}^+$.
\end{proof}

\begin{lemma}\label{lemma:supportrel9:3}
Let $\tau\in\bm\Delta$. If $z\in\mathcal{Z}$ is such that 
$\tilde{I}^1\tau z\tilde{I}^1\cap \mathbf{W}\neq\emptyset$ then
$\tilde{I}^1\tau z\tilde{I}^1\cap \mathbf{W}=\{\tau\}$.
\end{lemma}

\begin{proof}
For every $r\in\{1,\dots, m'\}$ we denote $\bm\Delta_{(r)}$, $\mathcal{Z}_{(r)}$, $\tilde{I}^1_{(r)}$ and 
$\mathbf{W}_{(r)}$ the subsets of $GL_{r}(A(E))$ similar to those defined for $GL_{m'}(A(E))$. 
We prove the statement of the lemma by induction on $r$. 
If $r=1$ we have 
$\bm\Delta_{(1)}=\varpi^\Z$, $\mathcal{Z}_{(1)}=\{1\}$, $\tilde{I}^1_{(1)}=1+\mathfrak{P}(E)$ and 
$\mathbf{W}_{(1)}=\varpi^{\Z}$ and we have
$(1+\mathfrak{P}(E))\varpi^a(1+\mathfrak{P}(E))\cap \varpi^{\Z}=\varpi^a(1+\mathfrak{P}(E))\cap \varpi^{\Z}=\{\varpi^a\}$ for every $a\in\Z$.
Now we suppose the statement true for every $r<m'$.
Let $x,y\in\tilde{I}^1$ such that $x\tau z y\in\mathbf{W}$. 
We proceed by steps.\\
\emph{First step}. We consider the decomposition
$\tilde{I}^1=(\tilde{I}^1\cap\U^-)(\tilde{I}^1\cap\U)(\tilde{I}^1\cap\M)$ and we write $x=x_1x_2x_3$ with $x_1\in \tilde{I}^1\cap\U^-$, $x_2\in \tilde{I}^1\cap\U$ and $x_3\in \tilde{I}^1\cap\M$. Then we have
$$x\tau z y=x_1\tau\big((\tau^{-1}x_2\tau)(\tau^{-1}x_3\tau)z(\tau^{-1}x_3^{-1}\tau)\big)(\tau^{-1}x_3\tau)y.$$ 
We observe that $\tau^{-1}x_3\tau$ is a diagonal matrix with coefficients in $1+\mathfrak{P}(E)$ and the conjugate of $z$ by this element is in $\mathcal{Z}$.
Moreover, $\tau^{-1}x_2\tau$ is in $\tilde{I}^1\cap\U$ and if we multiply it by an element of $\mathcal{Z}$ we obtain another element of $\mathcal{Z}$. 
If we set $z_1=\tau^{-1}x_2x_3\tau z\tau^{-1}x_3^{-1}\tau\in \mathcal{Z}$ then $\tilde{I}^1\tau z\tilde{I}^1=\tilde{I}^1\tau  z_1\tilde{I}^1$ and $(\tilde{I}^1\cap\U^-)\tau z_1\tilde{I}^1\cap \mathbf{W}\neq\emptyset$.
Hence, we can suppose $x\in \tilde{I}^1\cap\U^-$.\\
\emph{Second step}. Let $a_1\leq\dots\leq a_{m'}\in\N$ such that $\tau=\mathrm{diag}(\varpi^{a_i})$ and let $s\in\N^*$ such that $a_1=\cdots=a_s$ and $a_1< a_{s+1}$. 
We want to prove $z_{ij}\in \mathfrak{A}(E)$ for every $i\in\{1,\dots,s\}$ so we assume the opposite and we look for a contradiction. 
Let $\mathrm{v}$ be the valuation on $A(E)$ associated to $\mathfrak{P}(E)$ and let
\begin{align*}
b&=\min\{\mathrm{v}(\varpi^{a_1}z_{ij})\,|\,1\leq i\leq s, 1\leq j\leq m'\}\\
k&=\min\{1\leq j\leq m'\,|\, \text{there exists }z_{ij} \text{ with }1\leq i\leq s \text{ such that }\mathrm{v}(\varpi^{a_1}z_{ij})=b\}.
\end{align*}
Let $1\leq h\leq s$ be such that $\mathrm{v}(\varpi^{a_1}z_{hk})=b$.
By hypothesis the element $z_{hk}$ is not in $\mathfrak{A}(E)$ and so $h<k$ and 
\begin{equation}\label{eq:conditionb}
(a_1-1)\mathrm{v}(\varpi)<b<a_1\mathrm{v}(\varpi).
\end{equation}
We observe that for every $i\in\{1,\dots,m'\}$ and $j>i$ we have $\mathrm{v}(\varpi^{a_i}z_{ij})\geq b$: if $i\leq s$ by definition of $b$ and if 
$i>s$ because $\mathrm{v}(\varpi^{a_i}z_{ij})=a_i\mathrm{v}(\varpi)+\mathrm{v}(z_{ij})>(a_i-1)\mathrm{v}(\varpi)\geq a_1\mathrm{v}(\varpi)>b$. 
We consider the coefficient at position $(h,k)$ of $x\tau z y$ which is equal to 
$$\sum_{e=1}^{m'}\sum_{f=1}^{m'}x_{he}\varpi^{a_e}z_{ef}y_{fk}
=\sum_{e=1}^{h}\sum_{f=e}^{m'}x_{he}\varpi^{a_1}z_{ef}y_{fk}.$$
since $x_{he}=0$ if $e>h$ and $z_{ef}=0$ if $f<e$.
Now,
\begin{itemize}[$\bullet$]
\item if $e=h$ and $f=k$ then $\mathrm{v}(x_{hh}\varpi^{a_1}z_{hk}y_{kk})=b$  because $x_{hh}=1$, and $y_{kk}\in 1+\mathfrak{P}(E)$;
\item if $e=h$ and $f<k$ then $\mathrm{v}(x_{hh}\varpi^{a_1}z_{hf}y_{fk})>b$  by definition of $k$;
\item if $e=h$ and $f>k$ then $\mathrm{v}(x_{hh}\varpi^{a_1}z_{hf}y_{fk})>b$ because $y_{fk}\in\mathfrak{P}(E)$;
\item if $e<h$ then $\mathrm{v}(x_{he}\varpi^{a_1}z_{ef}y_{fk})>b$ because $x_{he}\in\mathfrak{P}(E)$.
\end{itemize}
We obtain an element of valuation $b$. 
Then $b$ must be a multiple of $\mathrm{v}(\varpi)$ because $x\tau z y\in\mathbf{W}$ but this in contradiction with (\ref{eq:conditionb}). 
Hence, $z_{ij}\in \mathfrak{A}(E)$ for every $i\in\{1,\dots,s\}$. Now, we can write $z=z'z''$ with 
$z'_{ii}=1$, $z'_{ij}=z_{ij}$ if $i\in\{s+1,\dots,m'\}$ and $j>i$ and $z'_{ij}=0$ otherwise and 
$z''_{ii}=1$, $z''_{ij}=z_{ij}$ if $i\in\{1,\dots,s\}$ and $j>i$ and $z''_{ij}=0$ otherwise. 
Then $z''\in\tilde{I}^1$ and so $\tilde{I}^1\tau z\tilde{I}^1=\tilde{I}^1\tau  z'\tilde{I}^1$ and 
$(\tilde{I}^1\cap\U^-)\tau z'\tilde{I}^1\cap \mathbf{W}\neq\emptyset$.
Then we can suppose $z$ of the form $\left(\begin{smallmatrix}\mathbb{I}_{s}&0\\0&\hat{z}\end{smallmatrix}\right)$ 
with $\hat{z}\in \mathcal{Z}_{(m'-s)}$.\\
\emph{Third step}. We write $x=x'x''$ with 
$x'_{ii}=1$, $x'_{ij}=x_{ij}$ if $i\in\{s+1,\dots,m'\}$ and $j<i$ and $x'_{ij}=0$ otherwise and  
$x''_{ii}=1$, $x''_{ij}=x_{ij}$ if $i\in\{1,\dots,s\}$ and $j<i$ and $x''_{ij}=0$ otherwise. Then $\tau^{-1}x''\tau\in \tilde{I}^1$ and it commutes with $z$. 
Then we can suppose $x$ of the form  
$\left(\begin{smallmatrix}\mathbb{I}_{s}&0\\ x'''&\hat{x}\end{smallmatrix}\right)$ with $x'''\in M_{(m'-s)\times s}(\mathfrak{P}(E))$ and
$\hat{x}\in \tilde{I}^1_{(m'-s)}$.\\
\emph{Fourth step}.
Let $\tau=\left(\begin{smallmatrix}\varpi^{a_1}\mathbb{I}_{s}&0\\ 0&\hat{\tau}\end{smallmatrix}\right)$ 
with $\hat\tau\in\bm\Delta_{(m'-s)}$ and 
$y=\left(\begin{smallmatrix}y_1&y_2\\y_3&\hat{y}\end{smallmatrix}\right)$ with $y_1\in \tilde{I}^1_{(s)}$,
$y_2\in M_{s\times (m'-s)}(\mathfrak{A}(E))$,
$y_3\in M_{(m'-s)\times s}(\mathfrak{P}(E))$ and
$\hat{y}\in \tilde{I}^1_{(m'-s)}$. 
Then the product $x\tau z y$ is
$$\begin{pmatrix}\mathbb{I}_s&0\\x'''&\hat{x}\end{pmatrix}
\begin{pmatrix}\varpi^{a_1}\mathbb{I}_s&0\\ 0&\hat{\tau}\end{pmatrix}
\begin{pmatrix}\mathbb{I}_s&0\\ 0&\hat{z}\end{pmatrix}
\begin{pmatrix}y_1&y_2\\y_3&\hat{y}\end{pmatrix}
=\begin{pmatrix}\varpi^{a_1}y_1&\varpi^{a_1}y_2\\ t&x'''\varpi^{a_1}y_2+\hat{x}\hat{\tau}\hat{z}\hat{y}\end{pmatrix}
$$
where $t=x'''\varpi^{a_1}y_1+\hat{x}\hat{\tau}\hat{z}y_3$. 
Since $x\tau z y$ is in $\mathbf{W}$ and since $y_1$ is invertible then $\varpi^{a_1}y_1$ must be in $\mathbf{W}_{(s)}$ and so $\varpi^{a_1}y_2=t=0$. 
This implies $y_1=\mathbb{I}$ and so $x\tau z y=\left(\begin{smallmatrix}\varpi^{a_1}\mathbb{I}_s& 0\\ 0&\hat{x}\hat{\tau}\hat{z}\hat{y}\end{smallmatrix}\right)$ with 
$\hat{x}\hat{\tau}\hat{z}\hat{y}\in\mathbf{W}_{(m'-s)}$.
Now, since $\tilde{I}^1_{(m'-s)}\hat\tau\hat{z}\tilde{I}^1_{(m'-s)}\cap \mathbf{W}_{(m'-s)}\neq\emptyset$, by inductive hypothesis we have $\hat{x}\hat{\tau}\hat{z}\hat{y}=\hat{\tau}$ and so $x\tau z y=\tau$.
\end{proof}

\begin{lemma}\label{lemma:supportrel9:4}
We have $J^1_\P \tau_P J^1_\P w\tau_\alpha w^{-1}J^1_\P
\cap J^1_\P B^{\times}J^1_\P=J^1_\P \tau_Q(U\cap wU^-w^{-1})J^1_\P$.
\end{lemma}

\begin{proof}
By lemma \ref{lemma:supportrel9:2} we have 
$J^1_\P\tau_P J^1_\P w\tau_\alpha w^{-1}J^1_\P =J^1_\P\tau_Q \mathcal{V} J^1_\P$.
Now, since 
$\mathfrak{J}^1\subset M_{m'}(\mathfrak{P}(E))$ we have 
$\V\subset \mathcal{Z}$ and $J^1_\P\subset \tilde{I}^1$ and so we obtain 
\begin{align*}
J^1_\P\tau_P J^1_\P w\tau_\alpha w^{-1}J^1_\P \cap B^{\times}
&\subset\tilde{I^1}\tau_Q \mathcal{Z} \tilde{I^1} \cap K^1_BU\mathbf{W}UK^1_B=K^1_BU(\tilde{I^1}\tau_Q \mathcal{Z} \tilde{I^1} \cap \mathbf{W})UK^1_B\\
\text{\scriptsize{(lemma \ref{lemma:supportrel9:3})}} &=K^1_BU\tau_Q UK^1_B=K^1_B\tau_Q UK^1_B.
\end{align*}
This implies $J^1_\P\tau_P J^1_\P w\tau_\alpha w^{-1}J^1_\P \cap B^{\times}= J^1_\P\tau_Q \mathcal{V} J^1_\P\cap K^1_B\tau_Q UK^1_B$. 
Let now $v\in\mathcal{V}$ be such that 
$J^1_\P\tau_Q v J^1_\P\cap K^1_B\tau_Q UK^1_B\neq\emptyset$. 
Then 
$v\in \tau_Q^{-1}J^1_\P K^1_B \tau_Q U K^1_B J^1_\P \cap\V\subset \tau_Q^{-1}J^1_\P \tau_Q U J^1_\P \cap\U$. 
Now $U=K_B\cap \U\subset J\cap \P $ normalizes $J^1_\P$ and so
$v\in \tau_Q^{-1}J^1_\P \tau_QJ^1_\P U\cap\U$ which is in $
\big(\tau_Q^{-1}(J^1_\P \cap\U^-_{\widehat{Q}})\tau_Q J^1_\P \cap\U\big)U$ by lemma \ref{lemma:J1Ptau:2}.
Hence, by lemma \ref{lemma:J1PU} we obtain
$v\in UJ^1_\P \cap\V\subset UJ^1\cap\V$. 
By lemma \ref{lemma:roots} we have 
$U\cap wU^-w^{-1}=U_{\widehat{P}}^+\cap wU_{\widehat{\alpha}}^-w^{-1}$ and proceeding similarly to proof of lemma \ref{lemma:V} we can prove $U_{\widehat{P}}^+\cap wU_{\widehat{\alpha}}^-w^{-1}\subset \V$. 
We obtain
\begin{align*}
UJ^1\cap\V&=(U\cap wU^-w^{-1})(U\cap wUw^{-1})J^1 \cap \V =(U\cap wU^-w^{-1})\Big(J^1 (U\cap wUw^{-1}) \cap \V\Big)\\
&=(U\cap wU^-w^{-1})\Big(J^1 (w^{-1}Uw\cap U) \cap \V^{w}\Big)^{w^{-1}}
\end{align*}
By definition of $\V$ we have  
$\V^{w}=(J^1_\P \cap w \U_{\widehat{\alpha}}^- w^{-1}\cap \U_{\widehat{P}}^+)^{w\tau_\alpha}
\subset (\U_{\widehat{\alpha}}^-)^{\tau_\alpha}\subset \U^-$ and then
$UJ^1\cap\V\subset (U\cap wU^-w^{-1})\big(J^1 \U \cap \U^-\big)^{w^{-1}}$ that by remark \ref{rmk:J1} is equal to $(U\cap wU^-w^{-1})J^1$.
We obtain $v$ in $(U\cap wU^-w^{-1})J^1\cap UJ^1_\P=
(U\cap wU^-w^{-1})(J^1\cap U)J^1_\P\subset
(U\cap wU^-w^{-1})K^1_B J^1_\P=(U\cap wU^-w^{-1}) J^1_\P$.
Hence, we have $J^1\tau_P J^1_\P w\tau_\alpha w^{-1}J^1_\P \cap J^1_\P B^{\times}J^1_\P=J^1_\P\tau_Q(U\cap wU^-w^{-1})J^1_\P$.
\end{proof}

\begin{lemma}\label{lemma:rel9:1}
For every $u\in U\cap wU^-w^{-1}$ we have $$(\widehat f_{\tau_{P}}\widehat f_w\widehat f_{\tau_{\alpha}}\widehat f_{w^{-1}})(\tau_Q u)=
q^{\ell(w)}
d(w,\alpha)\delta(\mathfrak{J}^1_0,\mathfrak{H}^1_0)^{\ell(w)}
\gamma_P\circ
\mathsf{p}\circ\kappa(w)\circ\iota\circ \gamma_i \circ \mathsf{p}\circ\kappa(w^{-1})\circ\iota\circ\mathsf{p}\circ\kappa(u)\circ\iota.$$
\end{lemma}

\begin{proof}
By lemma \ref{lemma:supportrel9:4} the support of 
$\widehat f_{\tau_{P}}\widehat f_w\widehat f_{\tau_{\alpha}}\widehat f_{w^{-1}}$ is contained in  
$J^1_\P\tau_{Q}(U\cap wU^-w^{-1})J^1_\P$. 
Let $u\in U\cap wU^-w^{-1}$. 
By lemma \ref{lemma:roots} we have $U\cap wU^-w^{-1}=U_{\widehat{P}}^+\cap wU_{\widehat\alpha}^-w^{-1}$, by lemma \ref{lemma:rel56} we have 
$\widehat f_{\tau_{\alpha}}=\widehat f_{\tau_{\alpha}}\widehat f_{w^{-1}u w}$ and by lemma \ref{lemma:finitealgebras} we have 
$\widehat f_{w^{-1}uw}\widehat f_{w^{-1}}=\widehat f_{w^{-1}}\widehat f_{u}$. 
Since $u$ is in $U=K_B\cap\U\subset J\cap\P$, it normalizes $J^1_\P$ and then by lemma \ref{lemma:product} we obtain 
$(\widehat f_{\tau_{P}}\widehat f_w\widetilde f_{\tau_{\alpha}}\widehat f_{w^{-1}})(\tau_Qu)
=(\widehat f_{\tau_{P}}\widehat f_w\widehat f_{\tau_{\alpha}}\widehat f_{w^{-1}}\widehat f_{u})(\tau_Qu)
=(\widehat f_{\tau_{P}}\widehat f_w\widehat f_{\tau_{\alpha}}\widehat f_{w^{-1}})(\tau_Q)\circ\mathsf{p}\circ\kappa(u)\circ\iota.$
It remains to calculate
$$(\widehat f_{\tau_{P}}\widehat f_w\widehat f_{\tau_{\alpha}}\widehat f_{w^{-1}})(\tau_Q)
=\sum_{x\in G/J^1_\P}\widehat f_{\tau_{P}}(\tau_P x)(\widehat f_w\widehat f_{\tau_{\alpha}}\widehat f_{w^{-1}})(x^{-1}w\tau_\alpha w^{-1}).$$ 
By lemma \ref{lemma:supportrel9} the support of function
$x\mapsto \widehat f_{\tau_{P}}(\tau_P x)(\widehat f_w\widehat f_{\tau_{\alpha}}\widehat f_{w^{-1}})(x^{-1}w\tau_{\alpha}w^{-1})$ is in $\V(w,\alpha)J^1_\P$.
Now, since for every $x\in \V(w,\alpha)=(J^1_\P\cap w\U_{\widehat\alpha}^+w^{-1} \cap \U_{\widehat{P}}^-)^{w\tau_\alpha^{-1} w^{-1}}$ we have $(x^{-1})^{w\tau_\alpha w^{-1}}\in J^1_\P\cap \U^-$ and 
$x^{\tau_\P^{-1}}\in (J^1_\P\cap w\U_{\widehat\alpha}^+w^{-1} \cap \U_{\widehat{P}}^-)^{\tau_Q^{-1}}\subset (J^1_\P\cap \U^-)^{\tau_Q^{-1}}$ which is in $J^1_\P\cap\U^-$ by lemma \ref{lemma:J1Ptau}, then $(x^{-1})^{w\tau_\alpha w^{-1}}$ and $x^{\tau_\P^{-1}}$ are in the kernel of $\eta_\P$. We obtain
\begin{align*}
(\widehat f_{\tau_{P}}\widehat f_w\widehat f_{\tau_{\alpha}}\widehat f_{w^{-1}})(\tau_Q)
&=[\V(w,\alpha):\V(w,\alpha)\cap H^1]\:
\widehat f_{\tau_{P}}(\tau_P)\circ
(\widehat f_w\widehat f_{\tau_{\alpha}}\widehat f_{w^{-1}})(w\tau_{\alpha}w^{-1})\\
\text{\scriptsize{(remark \ref{rmk:index2})}}
&=d(w,\alpha)q^{\ell(w)}
\widehat f_{\tau_{P}}(\tau_P)\circ
(\widehat f_w\widehat f_{\tau_{\alpha}}\widehat f_{w^{-1}})(w\tau_{\alpha}w^{-1})\\
\text{\scriptsize{(lemma \ref{lemma:wtauw})}}&=
d(w,\alpha)q^{\ell(w)}\delta(\mathfrak{J}^1_0,\mathfrak{H}^1_0)^{\ell(w)}
\gamma_P\circ
\mathsf{p}\circ\kappa(w)\circ\iota\circ \gamma_i \circ \mathsf{p}\circ\kappa(w^{-1})\circ\iota
\end{align*}
and so the result.
\end{proof}

\begin{lemma}\label{lemma:rel9:2}
We have 
$\gamma_Q=d(w,\alpha)\delta(\mathfrak{J}^1_0,\mathfrak{H}^1_0)^{\ell(w)}
\gamma_P\circ\mathsf{p}\circ\kappa(w)\circ\iota\circ \gamma_i \circ 
\mathsf{p}\circ\kappa(w^{-1})\circ\iota.$
\end{lemma}

\begin{proof}
By definition of $P(w,\alpha)$ and $Q(w,\alpha)$ (see paragraph \ref{subsec:decomp}) we have 
$\tau_P^{-1}\tau_Q=w\tau_i w^{-1}=\prod_{h=i+1}^{m'}\tau_{w(h)}^{-1}\tau_{w(h)-1}$ and so
\begin{align*}
\gamma_P^{-1}\gamma_Q&=\prod_{h=i+1}^{m'}\gamma_{w(h)}^{-1}\gamma_{w(h)-1}\\
\text{\scriptsize{(lemma \ref{lemma:gammai})}}
&=\prod_{h=i+1}^{m'}\partial^{m'-w(h)} \varepsilon\big(\tau_{w(h)}^{-1}\tau_{w(h)-1}\big)
=\Big(\prod_{h=i+1}^{m'}\partial^{m'-w(h)}\Big) \varepsilon(w\tau_{i}w^{-1})\\
\text{\scriptsize{(lemma \ref{lemma:gammai})}}
&=\Big(\prod_{h=i+1}^{m'}\partial^{m'-w(h)}\Big) \Big(\prod_{h=i+1}^{m'}\partial^{h-m'}\Big) \varepsilon(w)\circ\gamma_{i}\circ\varepsilon(w^{-1})\\
\text{\scriptsize{(remark \ref{rmk:xiw})}}
&=\Big(\prod_{h=i+1}^{m'}\partial^{m'-w(h)}\Big) \Big(\prod_{h=i+1}^{m'}\partial^{h-m'}\Big) \delta(\mathfrak{J}^1_0,\mathfrak{H}^1_0)^{\ell(w)}
\mathsf{p}\circ\kappa(w)\circ\iota\circ\gamma_{i}\circ\mathsf{p}\circ\kappa(w^{-1})\circ\iota\\
&=\Big(\prod_{h=i+1}^{m'}\partial^{h-w(h)}\Big) \delta(\mathfrak{J}^1_0,\mathfrak{H}^1_0)^{\ell(w)}
\mathsf{p}\circ\kappa(w)\circ\iota\circ\gamma_{i}\circ\mathsf{p}\circ\kappa(w^{-1})\circ\iota.
\end{align*}
It remains to prove $d(w,\alpha)=\prod_{h=i+1}^{m'}\partial^{h-w(h)}$.
Since by remark \ref{rmk:index2} we have $d(w,\alpha)=\partial^{\ell(w)}$, it is sufficient to prove 
$\sum_{h=i+1}^{m'}h-w(h)=\ell(w)$. 
We prove this statement by induction on $\ell(w)$. 
If $\ell(w)=1$, since $w$ is of minimal length in $wW_{\widehat{\alpha}}$, we have $w=s_\alpha=(i,i+1)$ and  
$\sum_{h=i+1}^{m'}h-w(h)=i+1-w(i+1)+\sum_{h=i+2}^{m'}h-w(h)=i+1-i+0=1.$ Let now $w$ of length $\ell(w)=n>1$. 
By lemma 2.12 of \cite{Chin1} there exists $\alpha_{jj+1}\in P$ and $w'\in W$ of length $n-1$ such that $w=s_jw'$.
Then $w'$ is of minimal length in $w'W_{\widehat\alpha}$ and so we can use inductive hypothesis. Moreover, by definition of  
$P$, there exist $\hat{h}\in\{i+1,\dots,m'\}$ such that  $j=w(\hat{h})$ and $j+1\neq w(h)$ for every 
$h\in\{i+1,\dots,m'\}$ and then $w(h)=w'(h)$ for every $h\in\{i+1,\dots,m'\}$ different from $\hat{h}$.
We obtain 
$\sum_{h=i+1}^{m'}h-w(h)= \sum_{h\neq\hat{h}}(h-w(h))+\hat{h}-w'(\hat{h})+w'(\hat{h})-w(\hat{h})
=\sum_{h\neq\hat{h}}(h-w'(h))+\hat{h}-w'(\hat{h})+(s_j(j))-j
=\sum_{h=i+1}^{m'}h-w'(h) +j+1-j=\ell(w')+1=\ell(w)$.
\end{proof}

\begin{lemma}\label{lemma:rel9:3}
We have 
$\widehat f_{\tau_{P}}\widehat f_w\widehat f_{\tau_{\alpha}}\widehat f_{w^{-1}}
=q^{\ell(w)}\widehat{f}_{\tau_Q}\sum_{u} \widehat{f}_u$
where $u$ describes a system of representatives of 
$(U\cap wU^-w^{-1})K^1/K^1$ in $U\cap wU^-w^{-1}$.
\end{lemma}

\begin{proof}
By lemma \ref{lemma:supportrel9:4} the support of 
$\widehat f_{\tau_{P}}\widehat f_w\widehat f_{\tau_{\alpha}}\widehat f_{w^{-1}}$ is contained in  
$J^1_\P\tau_{Q}(U\cap wU^-w^{-1})J^1_\P$. 
For every $u'\in U\cap wU^-w^{-1}$, by lemmas \ref{lemma:rel9:1} and \ref{lemma:rel9:2} we have
$(\widehat f_{\tau_{P}}\widehat f_w\widehat f_{\tau_{\alpha}}\widehat f_{w^{-1}})(\tau_Q u')
=q^{\ell(w)}
d(w,\alpha)\delta(\mathfrak{J}^1_0,\mathfrak{H}^1_0)^{\ell(w)}
\gamma_P\circ
\mathsf{p}\circ\kappa(w)\circ\iota\circ \gamma_i \circ \mathsf{p}\circ\kappa(w^{-1})\circ\iota\circ\mathsf{p}\circ\kappa(u')\circ\iota
=q^{\ell(w)}\gamma_Q\circ \mathsf{p}\circ\kappa(u')\circ\iota$.
To conclude we observe that 
$\big(\widehat{f}_{\tau_Q}\sum_{u} \widehat{f}_u\big)(\tau_Qu')=
(\widehat{f}_{\tau_Q} \widehat{f}_{u'})(\tau_Qu')
=\gamma_{Q}\circ \mathsf{p}\circ\kappa(u')\circ\iota$
\end{proof}

\begin{prop}\label{prop:isomHecke}
The map $\Theta''$ of lemma \ref{lemma:isomHecke} respect relations of lemma \ref{lemma:relations}. 
\end{prop}

\begin{proof}
By lemma \ref{lemma:finitealgebras} the map $\Theta''$ respects relation 1.
By lemma \ref{lemma:rel347} it respects relation 3 and $\widehat f_{\tau_0^{-1}}\widehat f_k=\widehat f_{\tau_0^{-1}k\tau_0}\widehat f_{\tau_0^{-1}}$ for every $k\in K_B$ and by lemmas \ref{lemma:producttau2} and \ref{lemma:gammacommute} it respects relations 2 and 6.
Moreover it respects relations 4 and 5 by lemma \ref{lemma:rel56} and relation 7 by lemma \ref{lemma:rel9:3}. 
\end{proof}

\begin{teor}\label{thm:isomHecke2}
For every non-zero $\gamma\in I_{\tau_{m'-1}}(\eta)$ and every $\beta$-extension $\kappa$ of $\eta$ there exists an algebra isomorphism 
$\Theta_{\gamma,\kappa}:\mathscr{H}_R(B^{\times},K^1_B)\rightarrow \mathscr{H}_R(G,\eta)$.
\end{teor}

\begin{proof}
By proposition \ref{prop:isomHecke} and by lemma \ref{lemma:isomHeckeetaP} there exists an algebra homomorphism from $\mathscr{H}_R(B^{\times},K^1_B)$ to $\mathscr{H}_R(G,\eta)$ which depends on the choice of 
a $\beta$-extension of $\eta$ and of an element in $I_{\tau_{m'-1}}(\eta_\P)$ which is isomorphic to $I_{\tau_{m'-1}}(\eta)$
by lemma \ref{lemma:isomHeckeetaP}.
Let $\Xi$ be a set of representatives of $K^1_B$-double cosets of $B^{\times}$. Then 
$\{f_x \,|\; x\in \Xi\}$ is a basis of $\mathscr{H}_R(B^{\times},K^1_B)$ as $R$-vector space and, since $I_G(\eta)=J^1B^{\times}J^1$ and $\mathrm{dim}_R(I_y(\eta))=1$ for every $y\in I_G(\eta)$, the set  
$\{\Theta_{\gamma,\kappa}(f_x) \,|\; x\in \Xi\}$ is a set of generators of $\mathscr{H}_R(G,\eta)$ as $R$-vector space and so 
$\Theta_{\gamma,\kappa}$ is surjective. 
Moreover, the set 
$\{\Theta_{\gamma,\kappa}(f_x) \,|\; x\in \Xi\}$ is linearly independent and so $\Theta_{\gamma,\kappa}$ is also injective. 
\end{proof}

\begin{rmk}\label{rmk:kappa'}
Let $\kappa$ and $\kappa'$ be two $\beta$-extensions of $\eta$. By paragraph \ref{subsec:Heisenberg} there exists a character $\chi$  of $\ent_{E}^{\times}$  trivial on $1+\wp_E$ such that $\kappa'=\kappa\otimes (\chi\circ N_{B/E})$. If we denote $\widetilde\chi=\infl_{\ent_E^{\times}}^{E^\times}\chi\circ N_{B/E}$, which is a character of $B^{\times}$, then $\Theta_{\gamma,\kappa}^{-1}\circ\Theta_{\gamma,\kappa'}$ maps 
$f_x$ to $\widetilde\chi f_x=\widetilde\chi(x)f_x$ for every $x\in B^{\times}$.
\end{rmk}

\section{Semisimple types}\label{sec:semisimpletypes}
Using notations of section \ref{sec:maximaltypes}, in this section we present the construction of semisimple types of $G$ with coefficients in $R$. We refer to sections 2.8-9 of \cite{MS} for more details.

\smallskip
Let $r\in\N^*$ and let $(m_1,\dots,m_r)$ be a family of strictly positive integers such that $\sum_{i=1}^r m_i=m$.
For every $i\in\{1,\dots,r\}$ we fix a maximal simple type $(J_i,\lambda_i)$ of $GL_{m_i}(D)$ and a simple stratum  $[\Lambda_i,n_i,0,\beta_i]$ of $A_{i}=M_{m_i}(D)$ such that $J_i=J(\beta_i,\Lambda_i)$. 
Then, the centralizer $B_i$ of $E_i=F[\beta_i]$ in 
$A_{i}$ is isomorphic to  $M_{m'_i}(D_i')$ for a suitable $E_i$-division algebra $D'_i$ of reduced degree $d'_i$ and a suitable $m'_i\in \N^*$. 
Moreover, $U(\Lambda_i)\cap B_i^{\times}$ is a maximal compact open subgroup of $B_i^{\times}$ that we identify with $GL_{m'_i}(\ent_{D'_i})$. 

\smallskip
Let $M$ be the standard Levi subgroup of $G$ of block diagonal matrices of sizes $m_1,\dots,m_r$. 
The pair $(J_M,\lambda_M)$ with $J_M=\prod_{i=1}^rJ_i$ and $\lambda_M=\bigotimes_{i=1}^r\lambda_i$ is called \emph{maximal simple type} of $M$.

\smallskip
For every $i\in\{1,\dots,r\}$ we fix a simple character $\theta_i\in \mathscr{C}_R(\Lambda_i,0,\beta_i)$  contained in $\lambda_i$ and we observe that this choice does not depend on the choice of the $\beta$-extension $\kappa_i$ such that $\lambda_i=\kappa_i\otimes\sigma_i$. Grouping $\theta_i$ according their  endo-classes, we obtain a partition $\{1,\dots,r\}=\bigsqcup_{j=1}^l I_j$ with $l\in\N^*$. 
Up to renumbering the $(J_i,\lambda_i)$ we can suppose that there exist integers 
$0=a_0<a_1<\cdots<a_l=r$ such that we have 
$I_j=\{i\in \N\,|\,a_{j-1}<i\leq a_j\}$. 
For every $j\in\{1,\dots,l\}$ we denote $m^j=\sum_{i\in I_j}m_i$ and $m'^j=\sum_{i\in I_j}m'_i$ and we consider the standard Levi subgroup $L$ of $G$ containing $M$ of block diagonal matrices of sizes $m^1,\dots,m^l$.

\smallskip
Let $j\in\{1,\dots,l\}$.  We choose a simple stratum  
$[\Lambda^j,n^j,0,\beta^j]$ of $M_{m^j}(D)$ as in paragraph 2.8 of \cite{MS}. 
If we denote by $B^j$ the centralizer of $E^j=F[\beta^j]$ in $M_{m^j}(D)$, there exist a $E^j$-division algebra $D'^j$ and an isomorphism that identifies $B^j$ to $M_{m'^{j}}(D'^j)$ and $U(\Lambda^j)\cap B^{j\times}$ to a standard parabolic subgroup of $GL_{m'^{j}}(\ent_{D'^j})$ associated to $m'_i$ with $i\in I_j$. 
We denote by $\theta^j$ the transfer of $\theta_i$ with $i\in I_j$ to $\mathscr{C}_R(\Lambda^j,0,\beta^j)$ that does not depend on $i$ and we fix a $\beta$-extension $\kappa^j$ of $\theta^j$.
In section 2.8 of \cite{MS} are defined two compact open subgroups  
$\mathbf{J}_j\subset J(\beta^j,\Lambda^j)$ and $\mathbf{J}^1_j\subset J^1(\beta^j,\Lambda^j)$ of $G$ such that  
$\mathbf{J}_j/\mathbf{J}_j^1\cong \prod_{i\in I_j}J_i/J^1_i$,  
and representations $\bm\kappa_j$ of $\mathbf{J}_j$ and $\bm\eta_j$ of $\mathbf{J}^1_j$ such that 
$\ind_{\mathbf{J}_j^1}^{J^1(\beta^j,\Lambda^j)}\bm\eta_j\cong \res^{J(\beta^j,\Lambda^j)}_{J^1(\beta^j,\Lambda^j)}\kappa^j$,
$\ind_{\mathbf{J}_j}^{J(\beta^j,\Lambda^j)}\bm\kappa_j\cong \kappa^j$, 
$\mathbf{J}_j\cap M=\prod_{i\in I_j}J_i$ and 
$\res^{\mathbf{J}_j}_{\mathbf{J}_j\cap M}\bm\kappa_j=\bigotimes_{i\in I_j}\kappa_i$ where $\kappa_i\in\mathcal{B}(\theta_i)$ for every $i\in I_j$.
We denote $\eta_i$ the restriction of $\kappa_i$ to $J^1(\beta_i,\Lambda_i)$ for every $i\in I_j$.
We obtain a decomposition $\lambda_i=\kappa_i\otimes\sigma_i$ for every $i\in I_j$ where $\sigma_i$ is a representation of $J_i$ trivial on $J_i^1$. 
We denote $\bm\sigma_j$ the representation $\bigotimes_{i\in I_j}\sigma_i$ viewed as a representation of $\mathbf{J}_j$ trivial on  $\mathbf{J}_j^1$ and we set $\bm\lambda_j=\bm\kappa_j\otimes \bm\sigma_j$.
Then $(\mathbf{J}_j,\bm\lambda_j)$ is a cover of $(\prod_{i\in I_j}J_i,\bigotimes_{i\in I_j}\lambda_i)$ (proposition 2.26 of \cite{MS}), $(\mathbf{J}_j,\bm\kappa_j)$ is  decomposed above $(\prod_{i\in I_j}J_i,\bigotimes_{i\in I_j}\kappa_i)$
and $(\mathbf{J}^1_j,\bm\eta_j)$ is a cover of $(\prod_{i\in I_j}J^1_i,\bigotimes_{i\in I_j}\eta_i)$ (proposition 2.27 of \cite{MS}).

\smallskip
We set $J^1_M=\prod_{i=1}^rJ^1_i$, $\kappa_M=\bigotimes_{i=1}^r\kappa_i$, $\eta_M=\bigotimes_{i=1}^r\eta_i$,
$\mathbf{J}_L=\prod_{j=1}^l \mathbf{J}_j$, $\mathbf{J}^1_L=\prod_{j=1}^l \mathbf{J}^1_j$, 
$\bm\lambda_L=\bigotimes_{j=1}^l\bm\lambda_j$, $\bm\kappa_L=\bigotimes_{j=1}^l\bm\kappa_j$, 
$\bm\eta_L=\bigotimes_{j=1}^l\bm\eta_j$ and $\bm\sigma_L=\bigotimes_{j=1}^l\bm\sigma_j$. 
By construction $(\mathbf{J}_L,\bm\lambda_L)$ and $(\mathbf{J}^1_L,\bm\eta_L)$ are covers of $(J_M,\lambda_M)$ and $(J_M^1,\eta_M)$ respectively and  
$(\mathbf{J}_L,\bm\kappa_L)$ is decomposed above $(J_M,\kappa_M)$. 

\smallskip
Proposition 2.28 of \cite{MS} defines a cover $(\mathbf{J},\bm\lambda)$ of $(\mathbf{J}_L,\bm\lambda_L)$ and so of $(J_M,\lambda_M)$, that we call \emph{semisimple type} of $G$. If the $(J_i,\lambda_i)$ are maximal simple supertypes, we call $(\mathbf{J},\bm\lambda)$ \emph{semisimple supertype} of $G$.
The semisimple type $(\mathbf{J},\bm\lambda)$ is associated to a stratum 
$[\bm\Lambda,\mathbf{n},0,\bm\beta]$ of $A$, not necessarily simple (section 2.9 of \cite{MS}). 
We denote by $B$ the centralizer of $\bm\beta$ in $A$, $B_L^{\times}=B^{\times}\cap L=\prod_{j=1}^lB^{j\times}$ and $\mathbf{J}^1=\mathbf{J}\cap U_1(\bm\Lambda)$. 
By propositions 2.30 and 2.31 of \cite{MS} there exists a unique pair $(\mathbf{J}^1,\bm\eta)$ decomposed above $(\mathbf{J}^1_L,\bm\eta_L)$ and so above $(J^1_M,\eta_M)$. Its intertwining set is $I_G(\bm\eta)=\mathbf{J}B^{\times}_L\mathbf{J}$ and for every $y\in B_L^{\times}$ the $R$-vector space $I_y(\bm\eta)$ is $1$-dimensional. We also have the isomorphisms
$$\mathbf{J}/\mathbf{J}^1\cong \mathbf{J}_L/\mathbf{J}^1_L\cong\prod_{i=1}^r J_i/J^1_i\cong \prod_{i=1}^r GL_{m'_i}(\frack_{D'_i}).$$
We can identify $\bm\sigma_L$ to an irreducible representation $\bm\sigma$ of $\mathbf{J}$ trivial on 
$\mathbf{J}^1$. 
By proposition 2.33 of \cite{MS} there exists a unique pair $(\mathbf{J},\bm\kappa)$ decomposed above $(\mathbf{J}_L,\bm\kappa_L)$ and so above $(J_M,\kappa_M)$. 
Moreover, we have 
$\bm\eta=\res^{\mathbf{J}}_{\mathbf{J}^1}\bm\kappa$, $\bm\lambda=\bm\kappa\otimes\bm\sigma$ and $I_G(\bm\kappa)=\mathbf{J}B^{\times}_L\mathbf{J}$.
We denote $\mathscr{M}$ the finite group 
$\prod_{i=1}^r GL_{m'_i}(\frack_{D'_i})$.
Then we can identify $\bm\sigma$ to a cuspidal (supercuspidal if $(\mathbf{J},\bm\lambda)$ is a semisimple supertype)  representation of $\mathscr{M}$. 

\begin{rmk}\label{rmk:choix1}
The choice of $\beta$-extensions $\kappa^j\in\mathcal{B}(\theta^j)$ for every $j\in\{1,\dots,l\}$ determines $\kappa_i\in\mathcal{B}(\theta_i)$ for every $i\in\{1,\dots,r\}$, $\bm\kappa^j$ for every $j\in\{1,\dots,l\}$, $\bm\kappa_L$ and $\bm\kappa$ and so the decompositions 
$\lambda_i=\kappa_i\otimes\sigma_i$, $\bm\lambda_j=\bm\kappa_j\otimes\bm\sigma_j$ and $\bm\lambda=\bm\kappa\otimes\bm\sigma$.
\end{rmk}

\subsection{The representation $\bm\eta_{max}$}\label{subsec:etamax}

In this paragraph we associate to every semisimple supertype $(\mathbf{J},\bm\lambda)$ of $G$ an irreducible projective representation $\bm\eta_{max}$ of a compact open subgroup of $G$ and we prove that the algebra 
$\mathscr{H}_R(G,\bm\eta_{max})$ is isomorphic to $\mathscr{H}_R(B^{\times}_L,K^1_L)$ where $K^1_L$ is the pro-$p$-radical of the maximal compact open subgroup of $B^{\times}_L$.

\smallskip
For every $j\in\{1,\dots,l\}$ we choose a simple stratum $[\Lambda_{max,j},n_{max,j},0,\beta^j]$ of $M_{m^j}(D)$ such that $U(\Lambda_{max,j})\cap B^{j\times}$ is a maximal compact open subgroup of $B^{j\times}$ containing 
$U(\Lambda^j)\cap B^{j\times}$ as in paragraph 6.2 of \cite{SeSt1}. 
Then we can identify $U(\Lambda_{max,j})\cap B^{j\times}$ to $GL_{m'^j}(\ent_{D'^j})$. 
Let $J_{max,j}=J(\beta^j,\Lambda_{max,j})$ and $J^1_{max,j}=J^1(\beta^j,\Lambda_{max,j})$.
We can also choose $\theta_{max,j}\in\mathscr{C}_R(\Lambda_{max,j},0,\beta^j)$ such that its transfer to $\mathscr{C}_R(\Lambda^j,0,\beta^j)$ is $\theta^j$. We fix a $\beta$-extension $\kappa_{max,j}$ of $\theta_{max,j}$ and we denote $\eta_{max,j}$ its restriction to $J^1_{max,j}$.
By (5.2) of \cite{SeSt1}, there exists a unique 
$\kappa^j\in \mathcal{B}(\theta^j)$ such that 
\begin{equation}\label{eq:kappamax}
\ind_{J(\beta^j,\Lambda^j)}^{(U(\Lambda_j)\cap B^{j\times})U_1(\Lambda^j)}\kappa^j  \cong \ind_{(U(\Lambda_j)\cap B^{j\times})J^1_{max,j}}^{(U(\Lambda^j)\cap B^{j\times})U_1(\Lambda^j)}\kappa_{max,j}
\end{equation}
and so by remark \ref{rmk:choix1} the choice of 
$\kappa_{max,j}$ determines $\bm\kappa_j$.
We denote  $J_{max}=\prod_{j=1}^lJ_{max,j}$, $J^1_{max}=\prod_{j=1}^lJ^1_{max,j}$,
$\kappa_{max}=\bigotimes_{j=1}^l\kappa_{max,j}$,
$\eta_{max}=\bigotimes_{j=1}^l\eta_{max,j}$,
$K_L= \prod_{j=1}^l U(\Lambda_{max,j})\cap B^{j\times}$ and
$K^1_L= \prod_{j=1}^l U_1(\Lambda_{max,j})\cap B^{j\times}$. 
If we denote $\mathscr{G}$ the finite group $\prod_{j=1}^l GL_{m'^j}(\frack_{D'^j})$, we obtain $J_{max}/J^1_{max}\cong K_L/K^1_L\cong\mathscr{G}$ and $(\mathscr{M},\bm\sigma)$ is a supercuspidal pair of 
$\mathscr{G}$.

\smallskip
As before in this section, by propositions 2.30, 2.31 and 2.33 of \cite{MS} we can define two compact open subgroups $\mathbf{J}_{max}$ and $\mathbf{J}^1_{max}$ of 
$G$ such that $\mathbf{J}_{max}/\mathbf{J}_{max}^1\cong J_{max}/J^1_{max}\cong\mathscr{G}$ and pairs 
$(\mathbf{J}_{max},\bm\kappa_{max})$ and 
$(\mathbf{J}^1_{max},\bm\eta_{max})$ decomposed above 
$(J_{max},\kappa_{max})$ and $(J_{max}^1,\eta_{max})$ respectively. Then we have 
$I_G(\bm\kappa_{max})=I_G(\bm\eta_{max})=\mathbf{J}_{max}B_L^{\times}\mathbf{J}_{max}$ 
and the $R$-vector spaces $I_y(\bm\eta_{max})$ and $I_y(\bm\kappa_{max})$ have dimension 1 for every 
$y\in B_L^{\times}$.

\begin{rmk}\label{rmk:choix2}
Since for every $j\in\{1,\dots,l\}$ the choice of 
$\kappa_{max,j}\in\mathcal{B}(\theta_{max,j})$ determine  $\bm\kappa_j$, the choice of $\kappa_{max}$ determine 
$\bm\kappa$ and $\bm\kappa_{max}$ and so the decomposition 
$\bm\lambda=\bm\kappa\otimes\bm\sigma$. 
On the other hand $\bm\eta_{max}$, the group $\mathscr{G}$
and the conjugacy class of $\mathscr{M}$ are uniquely determined by the semisimple supertype 
$(\mathbf{J},\bm\lambda)$, independently by the choice of $\kappa_{max}$ or of $\bm\kappa$.
\end{rmk}

\begin{prop}\label{prop:isomorphismetamaxj}
The algebras $\mathscr{H}_R(G,\bm\eta_{max})$ and $\bigotimes_{j=1}^{l}\mathscr{H}_R(GL_{m^j}(D),\eta_{max,j})$ are isomorphic.
\end{prop}

\begin{proof}
By lemma 2.4, proposition 2.5 of \cite{Gui} and lemma \ref{lemma:isomalgebras} there exists an algebra isomorphism 
$\bigotimes_{j=1}^{l}\mathscr{H}_R(GL_{m^j}(D),\eta_{max,j})\rightarrow \mathscr{H}_R(L,\eta_{max})$.
Now, since $I_G(\bm\eta_{max})\subset \mathbf{J}_{max}L\mathbf{J}_{max}$ the subalgebra 
$\mathscr{H}_R(\mathbf{J}_{max}L\mathbf{J}_{max},\bm\eta_{max})$ of $\mathscr{H}_R(G,\bm\eta_{max})$ of functions with support in $\mathbf{J}_{max}L\mathbf{J}_{max}$ is equal to $\mathscr{H}_R(G,\bm\eta_{max})$ and so by sections II.6-8 of \cite{Vig1} there exists an algebra isomorphism $\mathscr{H}_R(L,\eta_{max})\rightarrow \mathscr{H}_R(G,\bm\eta_{max})$ which preserves the support.  
\end{proof}

\begin{corol}\label{corol:isomHeckealgebras}
The $R$-algebras $\mathscr{H}_R(B_L^{\times},K^1_L)$ and $ \mathscr{H}_R(G,\bm\eta_{max})$ are isomorphic. 
\end{corol}

\begin{proof}
By remark 1.5 of \cite{Chin1} (see also theorem 6.3 of \cite{Krieg}) we have $\mathscr{H}_R(B_L^{\times},K^1_L)\cong \bigotimes_{j=1}^{l}\mathscr{H}_R(B^{j\times},U_1(\Lambda_{max,j})\cap B^{j\times})$ and by theorem \ref{thm:isomHecke2} we have 
$\mathscr{H}_R(B^{j\times},U_1(\Lambda_{max,j})\cap B^{j\times})\cong \mathscr{H}_R(GL_{m^j}(D),\eta_{max,j})$ for every $j\in\{1,\dots,l\}$. 
\end{proof}

\begin{rmk}\label{rmk:isomHeckealgebras}
By theorem \ref{thm:isomHecke2} the isomorphism of corollary \ref{corol:isomHeckealgebras} depends on the choice  of a $\beta$-extension $\kappa_{max,j}$ of $\eta_{max,j}$ and of an intertwining element of $\eta_{max,j}$ for every $j\in\{1,\dots,l\}$. Using proposition \ref{prop:isomorphismetamaxj}, the tensor product of these intertwining elements becomes an intertwining element of $\bm\eta_{max}$.
\end{rmk}

\begin{rmk}\label{rmk:corrTheta}
The procedure that associates $\bm\eta_{max}$ to $(\mathbf{J},\bm\lambda)$ depends on several non-canonical choices, for example the choice of the isomorphism $B_L^{\times}\rightarrow \prod GL_{m'^j}(D'^j)$.
To obtain a canonical correspondence, we denote  $\bm\Theta_i$ the endo-class of $\theta_i$ with $i\in\{1,\dots,r\}$ and we canonically associate to $(\mathbf{J},\bm\lambda)$ the formal sum $\bm\Theta(\mathbf{J},\bm\lambda)=\bm\Theta=\sum_{i=1}^r \frac{m_id}{[E_i:F]}\bm\Theta_i$. 
Furthermore, the group $\mathscr{G}$ and the $\mathscr{G}$-conjugacy class of $\mathscr{M}$ depend only on $(\mathbf{J},\bm\lambda)$ and actually the group $\mathscr{G}$ depends only on $\bm\Theta$ because $m'^j[\frack_{D'^j}:\frack_{E^j}]=\frac{m^jd}{[E^j:F]}=\sum_{i\in I_j}\frac{m_id}{[E_i:F]}$ which is the coefficient of $\bm\Theta_i$ in $\bm\Theta$.
We refer to paragraph 6.3 of \cite{SeSt1} for more details.
\end{rmk}

\section{The category equivalence $\mathscr{R}(G,\bm\eta_{max})\simeq\mathscr{R}(B_L^{\times},K_L^1)$}\label{sec:categoryequivalence}
Using notations of section \ref{sec:semisimpletypes}, in this section we prove that there exists an equivalence of categories between $\mathscr{R}(G,\bm\eta_{max})$ and $\mathscr{R}(B_L^{\times},K_L^1)$.
This allows to reduce the description of a positive-level block of $\mathscr{R}_R(G)$ to the description of a level-$0$ block of $\mathscr{R}_R(B_L^{\times})$.

\subsection{The category $\mathscr{R}(\mathbf{J},\bm\lambda)$}\label{subsec:categoryJlambda}
In this paragraph we associate to a semisimple supertype $(\mathbf{J},\bm\lambda)$ of $G$ a subcategory of $\mathscr{R}_R(G)$. 
We refer to \cite{SeSt1} for more details.

\smallskip
From now on we fix an extension $\bm\kappa_{max}$ of $\bm\eta_{max}$ to $\mathbf{J}_{max}$, as in paragraph \ref{subsec:etamax}. 
This uniquely determines a decomposition $\bm\lambda=\bm\kappa\otimes\bm\sigma$ where $\bm\kappa$ is an irreducible representation of $\mathbf{J}$ and $\bm\sigma$ is a supercuspidal representation of $\mathscr{M}$ viewed as an irreducible representation of $\mathbf{J}$ trivial on $\mathbf{J}^1$. 
We consider the functor 
$\textbf{\textsf{K}}_{\bm\kappa_{max}}:\mathscr{R}_R(G)\rightarrow\mathscr{R}(\mathbf{J}_{max}/\mathbf{J}^1_{max})=\mathscr{R}_R(\mathscr{G})$
given by $\textbf{\textsf{K}}_{\bm\kappa_{max}}(\pi)=\Hom_{\mathbf{J}^1_{max}}(\bm\eta_{max},\pi)$ for every representation $\pi$ of $G$ with $\mathbf{J}_{max}$ that acts on $\textbf{\textsf{K}}_{\bm\kappa_{max}}(\pi)$ by 
\begin{equation}\label{eq:actionJ}
x.\varphi=\pi(x)\circ \varphi\circ\bm\kappa_{max}(x)^{-1} 
\end{equation}
for every $x\in\mathbf{J}_{max}$. 
We denote $\pi(\bm\kappa_{max})$ this representation of $\mathscr{G}$. 
We remark that if $V_1$ and $V_2$ are representations of $G$ and $\phi\in\Hom_G(V_1,V_2)$ then $\textbf{\textsf{K}}_{\bm\kappa_{max}}(\phi)$ maps $\varphi$ to $\phi\circ\varphi$ for every $\varphi\in \Hom_G(\rho,V_1)$.
To more details on this functor see section 5 of \cite{MS} and \cite{SeSt1}.

\smallskip
We recall that we have $\bm\sigma=\bigotimes_{i=1}^r\sigma_i$ where $\sigma_i$ is a supercuspidal representation of $GL_{m'_{i}}(\frack_{D'_i})$. 
We denote $\Gamma_{\mathscr{M}}=\prod_{j=1}^l \mathrm{Gal}(\frack_{D'^j}/\frack_{E^j})^{|I_j|}$.
The equivalence class of $(\mathscr{M},\bm\sigma)$ (see definition 1.14 of \cite{SeSt1}) is the set, denoted by $[\mathscr{M},\bm\sigma]$, of supercuspidal pairs $(\mathscr{M}',\bm\sigma')$ of $\mathscr{G}$  such that there exists $\epsilon\in\Gamma_{\mathscr{M}}$ such that $(\mathscr{M}',\bm\sigma')$ is $\mathscr{G}$-conjugated to $(\mathscr{M},\bm\sigma^\epsilon)$.

\smallskip
Let $\bm\Theta=\bm\Theta(\mathbf{J},\bm\lambda)$.
For every representation $V$ of $G$ let $V[\bm\Theta,\bm\sigma]$ be the subrepresentation of $V$ generated by the maximal subspace of $\textbf{\textsf{K}}_{\bm\kappa_{max}}(V)$
such that every irreducible subquotient has supercuspidal support in $[\mathscr{M},\bm\sigma]$ and let $V[\bm\Theta]$ be the subrepresentation of $V$ generated by $\textbf{\textsf{K}}_{\bm\kappa_{max}}(V)$ (see paragraph 9.1 of \cite{SeSt1}).

\begin{defin}\label{def:RJlambda}
Let $\mathscr{R}(\mathbf{J},\bm\lambda)$ be the full subcategory of $\mathscr{R}_R(G)$ of representations $V$ such that $V=V[\bm\Theta,\bm\sigma]$. This does not depend on the choice of $\bm\kappa_{max}$ (see paragraph 10.1 of \cite{SeSt1}).
\end{defin}

\begin{rmk}\label{rmk:RJlambda}
For every representation $V$ of $G$ we have $V[\bm\Theta,\bm\sigma][\bm\Theta,\bm\sigma]=V[\bm\Theta,\bm\sigma]$ (see lemma 9.1 of \cite{SeSt1}) and so $V[\bm\Theta,\bm\sigma]$ is an object of $\mathscr{R}(\mathbf{J},\bm\lambda)$.
\end{rmk}

We call \emph{equivalence class of} $(\mathbf{J},\bm\lambda)$ the set $[\mathbf{J},\bm\lambda]$ of semisimple supertypes 
$(\widetilde{\mathbf{J}},\widetilde{\bm\lambda})$ of $G$ such that
$\ind_{\widetilde{\mathbf{J}}}^G(\widetilde{\bm\lambda})\cong\ind_{\mathbf{J}}^G(\bm\lambda)$.

\begin{teor}\label{thm:blocks}
The category $\mathscr{R}(\mathbf{J},\bm\lambda)$ depends only on the class $[\mathbf{J},\bm\lambda]$ and it is a block of $\mathscr{R}_R(G)$.
\end{teor}

\begin{proof}
It follows by propositions 10.2 and 10.5 and theorem 10.4 of \cite{SeSt1}.
\end{proof}

\begin{rmk}
The proof in \cite{SeSt1} of theorem \ref{thm:blocks} use the notions of inertial class of a supercuspidal pair of $G$ and the notion of supercuspidal support (see 1.1.3, 2.1.2 and 2.1.3 of \cite{MS2}).
These notions are very important in the study of representations of $GL_m(D)$ but in this article they are not used explicitly.
\end{rmk}

\subsection{The category equivalence}\label{subsec:categoryequivalence}

Let $(\mathbf{J},\bm\lambda)$ be a semisimple supertype of $G$ and let $\bm\Theta=\bm\Theta(\mathbf{J},\bm\lambda)$ be the formal sum of endo-classes associated to it.
In general there exist several semisimple supertypes of $G$ associated to $\bm\Theta$. 
We denote 
$\mathbf{X}=\mathbf{X}_{\bm\Theta}=\{[\mathbf{J}',\bm\lambda']\,|\,\bm\Theta(\mathbf{J}',\bm\lambda')=\bm\Theta\}.$
In this paragraph we prove that the sum $\bigoplus_{[\mathbf{J}',\bm\lambda']\in \mathbf{X}}\mathscr{R}(\mathbf{J}',\bm\lambda')$ is equivalent to the level-$0$ subcategory of $\mathscr{R}_R(B^{\times}_L)$.

\smallskip
Let $\mathbf{Y}=\mathbf{Y}_{\bm\Theta}$ be the set of equivalence classes of supercuspidal pairs of $\mathscr{G}$, that is uniquely determined by $\bm\Theta$ by remark \ref{rmk:corrTheta}.
Let $\bm\kappa_{max}$ be a fixed extension of $\bm\eta_{max}$ to $\mathbf{J}_{max}$ as in paragraph \ref{subsec:etamax} and let $\textbf{\textsf{K}}=\textbf{\textsf{K}}_{\bm\kappa_{max}}$.
By proposition 10.6 of \cite{SeSt1} there exists a bijection
\begin{equation}\label{eq:phikappa}
\phi_{\bm\kappa_{max}}: \mathbf{X}\rightarrow \mathbf{Y}
\end{equation}
given by $\phi_{\bm\kappa_{max}}([\mathbf{J}',\bm\lambda'])=[\mathscr{M},\bm\sigma]$ if the supercuspidal supports of irreducible subquotients of  $\textbf{\textsf{K}}(V)$ are in $[\mathscr{M},\bm\sigma]$ for every (or equivalently for one) object $V$ of $\mathscr{R}(\mathbf{J}',\bm\lambda')$.
This is equivalent to say that there exists $\bm\kappa$ as in section \ref{sec:semisimpletypes} (which depends on $\bm\kappa_{max}$) such that $\bm\lambda'=\bm\kappa\otimes\bm\sigma'$ with $(\mathscr{M},\bm\sigma')\in[\mathscr{M},\bm\sigma]$. 

\begin{prop}[Corollary 9.4 of \cite{SeSt1}]\label{prop:SS1}
For every representation $V$ of $G$ we have
\begin{equation}\label{eq:SS1}
V[\bm\Theta]=\bigoplus_{[\mathscr{M}',\bm\sigma']\in \mathbf{Y}}V[\bm\Theta,\bm\sigma'].
\end{equation}
\end{prop}

\begin{prop}[Lemma 10.3 of \cite{SeSt1}]\label{prop:SS2}
If $[\mathbf{J}',\bm\lambda']\in \mathbf{X}$ and $W$ is a simple object of $\mathscr{R}(\mathbf{J}',\bm\lambda')$ then $\textbf{\textsf{K}}(W)\neq 0$.
\end{prop}

Since $\mathbf{J}^1_{max}$ has a pro-order invertible in $R^{\times}$, the representation $\bm\eta_{max}$ is projective and so we can use notations and results of section \ref{subsec:categories}. 
We have defined the functor 
$$\mathbf{M}_{\bm\eta_{max}}:\mathscr{R}_R(G)\rightarrow \Mod-\mathscr{H}_R(G,\bm\eta_{max})$$ 
by $\mathbf{M}_{\bm\eta_{max}}(V)=\Hom_{G}(\ind_{\mathbf{J}^1_{max}}^G(\bm\eta_{max}),V)$ and $\mathbf{M}_{\bm\eta_{max}}(\phi):\varphi\mapsto \varphi\circ\phi$ for every representations $V$ and $V_1$ of $G$, $\phi\in\Hom_G(V,V_1)$ and $\varphi\in \Hom_{G}(\ind_{\mathbf{J}^1_{max}}^G(\bm\eta_{max}),V)$.

\begin{rmk}\label{rmk:forget}
Frobenius reciprocity induces a natural isomorphism between the functor $\mathbf{M}_{\bm\eta_{max}}$ composed with forget-functor $\Mod-\mathscr{H}_R(G,\bm\eta_{max})\rightarrow \Mod_R$ and the functor $\textbf{\textsf{K}}_{\bm\kappa_{max}}$ composed with the forget-functor $\mathscr{R}_R(\mathscr{G})\rightarrow \Mod_R$. 
This implies that for every representation $V$ of $G$ the subrepresentation $V[\bm\Theta]$ of $V$ is the subrepresentation $V[\bm\eta_{max}]$ defined in paragraph \ref{subsec:categories}.
\end{rmk}

\noindent
We have also defined the full subcategories $\mathscr{R}_{\bm\eta_{max}}(G)$ and $\mathscr{R}(G,\bm\eta_{max})$ of $\mathscr{R}_R(G)$. We recall that $\mathscr{R}(G,\bm\eta_{max})$ is the category of $V$ such that $V=V[\bm\Theta]$ and $\mathscr{R}_{\bm\eta_{max}}(G)$ is the category of $V$ such that  $\mathbf{M}_{\bm\eta_{max}}(V')\neq 0$ for every irreducible subquotient $V'$ of $V$. 

\begin{lemma}\label{lemma:categoryequality}
We have $\mathscr{R}(G,\bm\eta_{max})= \mathscr{R}_{\bm\eta_{max}}(G)$.
\end{lemma}

\begin{proof}
Thanks to remark \ref{rmk:categoryinclusion} it is sufficient to prove $\mathscr{R}(G,\bm\eta_{max})\subset \mathscr{R}_{\bm\eta_{max}}(G)$.
Let $V$ be a representation in $\mathscr{R}(G,\bm\eta_{max})$. 
By proposition \ref{prop:SS1} we have $V=\bigoplus_{\mathbf{Y}}V[\bm\Theta,\bm\sigma']$ and by remark \ref{rmk:RJlambda} the representation $V[\bm\Theta,\bm\sigma']$ is an object of $\mathscr{R}(\mathbf{J}',\bm\lambda')$ where $[\mathbf{J}',\bm\lambda']=\phi_{\bm\kappa_{max}}^{-1}([\mathscr{M},\bm\sigma'])\in \mathbf{X}$.   
Hence, we obtain the inclusion $\mathscr{R}(G,\bm\eta_{max})\subset \bigoplus_{\mathbf{X}}\mathscr{R}(\mathbf{J}',\bm\lambda')$.
Let now $W$ be an object of $\bigoplus_{\mathbf{X}}\mathscr{R}(\mathbf{J}',\bm\lambda')$ and $W'$ an irreducible subquotient of $W$.
Then $W'$ is an irreducible object of $\mathscr{R}(\mathbf{J}',\bm\lambda')$ for a $[\mathbf{J}',\bm\lambda']\in \mathbf{X}$ and so by proposition \ref{prop:SS2} we have $\textbf{\textsf{K}}_{\bm\kappa_{max}}(W)\neq 0$. 
Therefore, by remark \ref{rmk:forget} we have $\mathbf{M}_{\bm\eta_{max}}(W')\neq 0$ which implies $\bigoplus_{\mathbf{X}}\mathscr{R}(\mathbf{J},\bm\lambda')\subset \mathscr{R}_{\bm\eta_{max}}(G)$.
\end{proof}

\begin{rmk}\label{rmk:decblocks}
We have proved that $\mathscr{R}(G,\bm\eta_{max})=\mathscr{R}_{\bm\eta_{max}}(G)=\bigoplus_{[\mathbf{J},\bm\lambda]\in\mathbf{X}}\mathscr{R}(\mathbf{J},\bm\lambda).$
Moreover, by proposition \ref{prop:TFAE}, a representation $V$ of $G$ is in this category if and only if it verifies one of the following equivalent conditions: $V=V[\bm\Theta]$, for every subquotient $Z$ of $V$ we have $Z=Z[\bm\Theta]$, for every irreducible subquotient $U$ of $V$ we have $\mathbf{M}_{\bm\eta_{max}}(U)\neq 0$ or for every non-zero subquotient $W$ of $V$ we have $\mathbf{M}_{\bm\eta_{max}}(W)\neq 0$.
\end{rmk}

\begin{teor}\label{thm:equivalenceM}
The functor $\mathbf{M}_{\bm\eta_{max}}$ is an equivalence of categories between $\mathscr{R}(G,\bm\eta_{max})$ and $\Mod-\mathscr{H}_R(G,\bm\eta_{max})$.
\end{teor}

\begin{proof}
We apply theorem \ref{thm:equivalence} with $\mathtt{G}=G$ and $\sigma=\bm\eta_{max}$.
\end{proof}

\begin{rmk}\label{rmk:level0}
We recall that a level-$0$ representation of $B_L^{\times}$ is a representation generated by its 
$K^1_L$-invariant vectors. It is equivalent to say that all irreducible subquotients have non-zero $K^1_L$-invariant vectors (see section 3 of \cite{Chin1}).
The category $\mathscr{R}(B_L^{\times},K^1_L)$ is called \emph{level-$0$ subcategory} of $\mathscr{R}_R(B_L^{\times})$. 
By section 3 of \cite{Chin1} and theorem \ref{thm:equivalence}, the $K^1_L$-invariant functor $\inv_{K^1_L}$ induces an equivalence of categories between $\mathscr{R}(B_L^{\times},K_L^1)$ and $\Mod-\mathscr{H}_R(B_L^{\times},K_L^1)$ whose quasi-inverse is
$W\mapsto W\otimes_{\mathscr{H}_R(B_L^{\times},K^1_L)}\ind_{K^1_L}^{B_L^{\times}}(1)$.
We recall that if $(\varrho,Z)$ is a representation of $B^{\times}_L$ then the action of $\Phi\in \mathscr{H}_R(B_L^{\times},K_L^1)$ on $z\in Z^{K^1_L}$ is given by $z.\Phi=\sum_{x\in K^1_L\backslash B_L^{\times}}\Phi(x)\varrho(x^{-1})z$.
\end{rmk}

\begin{corol}\label{corol:equiv}
There exists an equivalence of categories between  $\mathscr{R}(G,\bm\eta_{max})$ and $\mathscr{R}(B_L^{\times},K_L^1)$.
\end{corol}

\begin{proof}
By corollary \ref{corol:isomHeckealgebras} the algebras $\mathscr{H}_R(B_L^{\times},K^1_L)$ and $\mathscr{H}_R(G,\bm\eta_{max})$ are isomorphic. We obtain an equivalence of categories between $\Mod-\mathscr{H}_R(G,\bm\eta_{max})$ and $\Mod-\mathscr{H}_R(B_L^{\times},K_L^1)$ and so between $\mathscr{R}(G,\bm\eta_{max})$ and $\mathscr{R}(B_L^{\times},K_L^1)$ by theorem \ref{thm:equivalenceM} and remark \ref{rmk:level0}.
\end{proof}

Now we want to describe the functor that induces this equivalence of categories.
We recall that we have fixed an isomorphism $B_L^{\times}\cong\prod GL_{m'^j}(D'^j)$ and an extension $\bm\kappa_{max}$ of $\bm\eta_{max}$.
We also fix a non-zero intertwining element $\gamma$ of $\bm\eta_{max}$ as in remark \ref{rmk:isomHeckealgebras}.
By corollary \ref{corol:isomHeckealgebras} we have an isomorphism 
$\Theta_{\gamma,\bm\kappa_{max}}:\mathscr{H}_R(B_L^{\times},K^1_L)\rightarrow \mathscr{H}_R(G,\bm\eta_{max})$ which 
induces an equivalence of categories $\Theta_{\gamma,\bm\kappa_{max}}^*:\Mod-\mathscr{H}_R(G,\bm\eta_{max})\rightarrow \Mod-\mathscr{H}_R(B_L^{\times},K_L^1)$. 
We obtain the diagram
\begin{equation}\label{diagram:equivalence}
\xymatrix{
\mathscr{R}(G,\bm\eta_{max})
\ar[rr]^{\text{Corollary \ref{corol:equiv}}}
\ar[d]^{\mathbf{M}_{\bm\eta_{max}}}
& &\mathscr{R}(B_L^{\times},K_L^1)
\\
\Mod-\mathscr{H}_R(G,\bm\eta_{max})
\ar[rr]^{\Theta_{\gamma,\bm\kappa_{max}}^*}
& &\Mod-\mathscr{H}_R(B_L^{\times},K_L^1).
\ar[u]_{\text{Remark \ref{rmk:level0}}}
}
\end{equation}
The functor $\mathbf{M}_{\bm\eta_{max}}:\mathscr{R}(G,\bm\eta_{max})\rightarrow \Mod-\mathscr{H}_R(G,\bm\eta_{max})$
is an equivalence of categories by theorem  \ref{thm:equivalenceM}. 
By lemma \ref{lemma:isomalgebras} the right action of  $\mathscr{H}_R(G,\bm\eta_{max})$ on $\mathbf{M}_{\bm\eta_{max}}(V)$ is given by $(m.\Psi)(f)=m(\Psi*f)$ for every $m\in \mathbf{M}_{\bm\eta_{max}}(V)$, $\Psi\in \mathscr{H}_R(G,\bm\eta_{max})$ and 
$f\in \ind^G_{\mathbf{J}^1_{max}}(\bm\eta_{max})$.
The right action of $\Phi\in \mathscr{H}_R(B_L^{\times},K_L^1)$ on a $\mathscr{H}_R(G,\bm\eta_{max})$-module $N$ is given by $N.\Phi=N.\Theta_{\gamma,\bm\kappa_{max}}(\Phi)$.
By remark \ref{rmk:level0} the functor 
$W\mapsto W\otimes_{\mathscr{H}_R(B_L^{\times},K^1_L)}\ind_{K^1_L}^{B_L^{\times}}(1)$
is a category equivalence between $\Mod-\mathscr{H}_R(B_L^{\times},K_L^1)$ and $\mathscr{R}(B_L^{\times},K_L^1)$ where, by lemma \ref{lemma:isomalgebras}, the left action of $\Phi\in\mathscr{H}_R(B_L^{\times},K^1_L)$ on $f\in\ind_{K^1_L}^{B_L^{\times}}(1)$ is given by $\Phi.f=\Phi*f$. 
Moreover, the left action of $x\in B_L^{\times}$ on $w\otimes f\in W\otimes_{\mathscr{H}_R(B_L^{\times},K^1_L)}\ind_{K^1_L}^{B_L^{\times}}(1)$ is given by $x.(w\otimes f)=w\otimes(x.f)$.

\smallskip
Composing these three functors we obtain the equivalence of categories of corollary \ref{corol:equiv} which we denote $\mathbf{F}_{\gamma,\bm\kappa_{max}}$ and that is given by 
\begin{equation}\label{eq:equiv}
\mathbf{F}_{\gamma,\bm\kappa_{max}}(\pi,V)=\mathbf{M}_{\bm\eta_{max}}(\pi,V)\otimes_{\mathscr{H}_R(B^{\times}_L,K^1_L)}\ind_{K^1_L}^{B_L^{\times}}(1_{K^1_L})
\end{equation}
for every $(\pi,V)$ in $\mathscr{R}(G,\bm\eta_{max})$, where the right action of $\Phi\in\mathscr{H}_R(B^{\times}_L,K^1_L)$ on $m\in\mathbf{M}_{\bm\eta_{max}}(\pi,V)$ is given by  
$(m.\Phi)(f)= m(\Theta_{\gamma,\bm\kappa_{max}}(\Phi)*f)$ for every $f\in\ind_{\mathbf{J}^1_{max}}^G(\bm\eta_{max})$.
We remark that if $V_1$ and $V_2$ are in $\mathscr{R}(G,\bm\eta_{max})$ and $\phi\in\Hom_G(V_1,V_2)$ then $\mathbf{F}_{\gamma,\bm\kappa_{max}}(\phi)$ maps $m\otimes f$ to $(\phi\circ m)\otimes f$ for every $m\in\mathbf{M}_{\bm\eta_{max}}(V_1)$ and $f\in\ind_{K^1_L}^{B_L^{\times}}(1_{K^1_L})$.

\subsection{Correspondence between blocks}\label{subsec:correspondence}
In this paragraph we discuss the correspondence among blocks of $\mathscr{R}(B_L^{\times},K_L^1)$ and those of $\mathscr{R}(G,\bm\eta_{max})$ induced by the equivalence of categories $\mathbf{F}_{\gamma,\bm\kappa_{max}}$ defined in (\ref{eq:equiv}).

\smallskip
We consider the functor 
$\textbf{\textsf{K}}_{K_L}:\mathscr{R}(B^{\times}_L,K^1_L)\rightarrow \mathscr{R}_R(K_L/K_L^1)=\mathscr{R}_R(\mathscr{G})$
given by $\textbf{\textsf{K}}_{K_L}(Z)=Z^{K^1_L}$
and $\textbf{\textsf{K}}_{K_L}(\phi)=\phi_{|Z^{K^1_L}}$ for every representations $(\varrho,Z)$ and $(\varrho_1,Z_1)$ of $B^{\times}_L$ and every $\phi\in\Hom_{B_L^{\times}}(Z,Z_1)$,
where $x\in K_L$ acts on $z\in Z^{K^1_L}$ by $x.z=\varrho(x)z$.
It is the functor presented in paragraph \ref{subsec:categoryJlambda} when we replace $G$ by $B^{\times}_L$ and $\bm\kappa_{max}$ by trivial representation of $K_L$.
We also consider the functor $\mathbf{H}:\Mod-\mathscr{H}_R(B_L^{\times},K^1_L)\rightarrow \mathscr{R}_R(K_L/K_L^1)$ 
given by $\mathbf{H}(W)=(\varrho',W)$ and $\mathbf{H}(\phi)=\phi$ for every $\mathscr{H}_R(B_L^{\times},K^1_L)$-modules $W$ and $W_1$ and every $\phi\in \Hom_{\mathscr{H}_R(B_L^{\times},K^1_L)}(W,W_1)$, where $\varrho'(k)w=w.f_{k^{-1}}$ for every $k\in K_L$ and $w\in W$.

\begin{rmk}\label{rmk:downarrow}
The functor $\textbf{\textsf{K}}_{K_L}$ is the composition of $\inv_{K^1_L}$ (see remark \ref{rmk:level0}) 
and the functor $\mathbf{H}$. 
Actually if $(\varrho,Z)$ is an object of $\mathscr{R}(B^{\times}_L,K^1_L)$ then 
$\mathbf{H}(\inv_{K^1_L}(Z))=\mathbf{H}(Z^{K^1_L})=(\varrho',Z^{K^1_L})$ where 
$\varrho'(k)z=z.f_{k{-1}}=\sum_{x\in K^1_L\backslash B_L^{\times}}f_{k^{-1}}(x)\varrho(x^{-1})z=\varrho(k)z$ for every $z\in Z^{K^1_L}$ and $k\in K_L$.
\end{rmk}

We obtain the diagram
\begin{equation}\label{diagram:naturalisom}
\begin{gathered}
\xymatrix{
\mathscr{R}(G,\bm\eta_{max})
\ar[rr]^{\mathbf{F}_{\gamma,\bm\kappa_{max}}}
\ar@/_0.5pc/[dr]^{\quad\Theta_{\gamma,\bm\kappa_{max}}^*\circ\mathbf{M}_{\bm\eta_{max}}}
\ar@/_1pc/[ddr]_{\textbf{\textsf{K}}_{\bm\kappa_{max}}}
&&\mathscr{R}(B_L^{\times},K_L^1)
\ar@/^1pc/[dl]_{\inv_{K^1_L}}
\ar@/^1pc/[ddl]^{\textbf{\textsf{K}}_{K_L}}
\\
&\Mod-\mathscr{H}_R(B_L^{\times},K^1_L)
\ar[d]^{\mathbf{H}}
&
\\ 
&\mathscr{R}_R(\mathscr{G})&
}
\end{gathered}
\end{equation}

\begin{prop}\label{thm:naturalisom}
There exists a natural isomorphism between $\textbf{\textsf{K}}_{K_L}\circ \mathbf{F}_{\gamma,\bm\kappa_{max}}$ and $\textbf{\textsf{K}}_{\bm\kappa_{max}}$.
\end{prop}

\begin{proof}
By remark \ref{rmk:downarrow} we have $\textbf{\textsf{K}}_{K_L}\circ \mathbf{F}_{\gamma,\bm\kappa_{max}}=\mathbf{H}\circ \inv_{K^1_L}\circ \mathbf{F}_{\gamma,\bm\kappa_{max}}$ and by diagram (\ref{diagram:equivalence}) we have a natural isomorphism between $\inv_{K^1_L}\circ \mathbf{F}_{\gamma,\bm\kappa_{max}}$ and $\Theta_{\gamma,\bm\kappa_{max}}^*\circ\mathbf{M}_{\bm\eta_{max}}$
so it is sufficient to find a natural isomorphism 
$\mathfrak{Z}:\mathbf{H}\circ\Theta_{\gamma,\bm\kappa_{max}}^*\circ\mathbf{M}_{\bm\eta_{max}}\rightarrow \textbf{\textsf{K}}_{\bm\kappa_{max}}$. 
For every object $(\pi,V)$ of $\mathscr{R}(G,\bm\eta_{max})$, let $\mathfrak{Z}_V:\mathbf{M}_{\bm\eta_{max}}(V)\rightarrow \textbf{\textsf{K}}_{\bm\kappa_{max}}(V)$ be the isomorphism of $R$-modules given by remark \ref{rmk:forget}.
The action of $x\in K_L/K_L^1\cong\mathscr{G}$ on $m\in \mathbf{M}_{\bm\eta_{max}}(\pi,V)$ is given by $x.m=m.\Theta_{\gamma,\bm\kappa_{max}}(f_{x^{-1}})=m.\widetilde{f}_{x^{-1}}$ where
$\widetilde{f}_{x^{-1}}\in \mathscr{H}_R(G,\bm\eta_{max})$ has support $x^{-1}\mathbf{J}^1_{max}$ and $\widetilde{f}_{x^{-1}}(x^{-1})=\bm\kappa_{max}(x^{-1})$ while the action of $x\in \mathbf{J}_{max}/\mathbf{J}^1_{max}\cong\mathscr{G}$ on $\varphi\in \textbf{\textsf{K}}_{\bm\kappa_{max}}(V)$ is given by (\ref{eq:actionJ}).
We have to prove that 
$\mathfrak{Z}_V(x.m)=x.\mathfrak{Z}_V(m)$ for every 
$m\in \mathbf{M}_{\bm\eta_{max}}(\pi,V)$ and $x\in \mathscr{G}$.
We recall that in paragraph \ref{subsec:Hecke} we have defined elements $i_v:\mathbf{J}^1_{max}\rightarrow V_{\bm\eta_{max}}$ with $v\in V_{\bm\eta_{max}}$, which generate  $\ind_{\mathbf{J}^1_{max}}^G(\bm\eta_{max})$ as representation of $G$, such that $m(i_v)=\mathfrak{Z}_V(m)(v)$.
Then for every $v\in V_{\bm\eta_{max}}$ we have 
$\mathfrak{Z}_V(x.m)(v)=(x.m)(i_v)=
(m.\widetilde{f}_{x^{-1}})(i_v)=m(\widetilde{f}_{x^{-1}}*i_v)$. 
The support of $\widetilde{f}_{x^{-1}}*i_v$ is $\mathbf{J}_{max}^1x^{-1}$ and 
$(\widetilde{f}_{x^{-1}}*i_v)(x^{-1})=\widetilde{f}_{x^{-1}}(x^{-1})v=\bm\kappa_{max}(x^{-1})v$. 
Hence, we obtain 
$\mathfrak{Z}_V(x.m)(v)=m(x.i_{\bm\kappa_{max}(x^{-1})v})=\pi(x)\big(m(i_{\bm\kappa_{max}(x^{-1})v})\big)=\pi(x)\big(\mathfrak{Z}_V(m)(\bm\kappa_{max}(x^{-1})v)\big)=(x.\mathfrak{Z}_V(m))(v)$.
Now, let $V_1$ and $V_2$ be two objects of $\mathscr{R}(G,\bm\eta_{max})$ and let $\phi\in\Hom_G(V_1,V_2)$.
Then for every $m\in \mathbf{M}_{\bm\eta_{max}}(V_1)$ and every $v\in V_{\bm\eta_{max}}$ we have
$\mathfrak{Z}_{V_2}\left(\mathbf{H}(\Theta_{\gamma,\bm\kappa_{max}}^*(\mathbf{M}_{\bm\eta_{max}}(\phi)))(m)\right)(v)=\mathfrak{Z}_{V_2}(\phi\circ m)(v)$ which is equal to $(\phi\circ m)(i_v)$ by Frobenius reciprocity. On the other hand we have  
$\textbf{\textsf{K}}_{\bm\kappa_{max}}(\phi)(\mathfrak{Z}_{V_1}(m))(v)=\phi(\mathfrak{Z}_{V_1}(m)(v))$ which is equal to $\phi(m(i_v))$ by Frobenius reciprocity. This shows that $\mathfrak{Z}$ is a natural isomorphism.
\end{proof}

Now we look for a block decomposition of $\mathscr{R}(B^{\times}_L,K^1_L)$.
Let $[\mathscr{M},\bm\sigma]\in \mathbf{Y}$. 
Then $\mathscr{M}=\prod_{j=1}^l\mathscr{M}_j$ and $\bm\sigma=\bigotimes_{j=1}^l\bm\sigma_j$ where 
$\mathscr{M}_j\cong\mathbf{J}_j/\mathbf{J}^1_j$ and $[\mathscr{M}_j,\bm\sigma_j]$ is class of supercuspidal pairs of $GL_{m'^j}(\frack_{D'^j})$.
For every $j\in\{1,\dots,l\}$, replacing $G$ by $B^{j\times}$ and $\bm\kappa_{max}$ by the trivial character of $U(\Lambda_{max,j})\cap B^{j\times}$ in definition \ref{def:RJlambda}, we obtain an abelian full subcategory 
$\mathscr{R}(U(\Lambda_{max,j})\cap B^{j\times},\bm\sigma_j)$ of $\mathscr{R}_R(B^{j\times})$ whose objects are representations $V_j$ of $B^{j\times}$ generated by the maximal subspace of $V_j^{U_1(\Lambda_{max,j})\cap B^{j\times}}$ for which every irreducible subquotient has supercuspidal support in $[\mathscr{M}_j,\bm\sigma_j]$. 
We obtain a full subcategory $\mathscr{R}(K_L,\bm\sigma)$ of $\mathscr{R}_R(B^{\times}_L)$ (and of $\mathscr{R}(B^{\times}_L,K^1_L)$) whose objects are representations $V$ of $B^{\times}_L$ generated by the maximal subspace of $V^{K^1_L}$ such that every irreducible subquotient has supercuspidal support in $[\mathscr{M},\bm\sigma]$. 
Theorem \ref{thm:blocks} and remark \ref{rmk:decblocks} give a block decomposition of $\mathscr{R}(B^{j\times},U_1(\Lambda_{max,j})\cap B^{j\times})$ for every $j\in\{1,\dots,l\}$ and so we obtain a block decomposition 
$$\mathscr{R}(B^{\times}_L,K^1_L)=\bigoplus_{[\mathscr{M},\bm\sigma]\in \mathbf{Y}} \mathscr{R}(K_L,\bm\sigma).$$ 
We recall that we have a block decomposition 
$\mathscr{R}(G,\bm\eta_{max})=\bigoplus_{[\mathbf{J},\bm\lambda]\in \mathbf{X}} \mathscr{R}(\mathbf{J},\bm\lambda)$ by remark \ref{rmk:decblocks} and a bijection $\phi_{\bm\kappa_{max}}:\mathbf{X}\rightarrow\mathbf{Y}$ defined in (\ref{eq:phikappa}) which depends on the choice of $\bm\kappa_{max}$.

\begin{teor}
Let $[\mathbf{J},\bm\lambda]\in\mathbf{X}$ and $[\mathscr{M},\bm\sigma]=\phi_{\bm\kappa_{max}}([\mathbf{J},\bm\lambda])\in\mathbf{Y}$.
Then $\mathbf{F}_{\gamma,\bm\kappa_{max}}$ induces an equivalence of categories between the block $\mathscr{R}(\mathbf{J},\bm\lambda)$ of $\mathscr{R}_R(G)$ and the block $\mathscr{R}(K_L,\bm\sigma)$ of $\mathscr{R}_R(B_L^{\times})$.
\end{teor}

\begin{proof}
If $V$ is an object of $\mathscr{R}(\mathbf{J},\bm\lambda)$, by proposition \ref{thm:naturalisom} there exists an isomorphism of representations of $\mathscr{G}$ between $\textbf{\textsf{K}}_{K_L}(\mathbf{F}_{\gamma,\bm\kappa_{max}}(V))$ and $\textbf{\textsf{K}}_{\bm\kappa_{max}}(V)$. 
Then irreducible subquotients of $(\mathbf{F}_{\gamma,\bm\kappa_{max}}(V))^{K^1_L}$ have supercuspidal support in $[\mathscr{M},\bm\sigma]$ and so $\mathbf{F}_{\gamma,\bm\kappa_{max}}(V)$ is in $\mathscr{R}(K_L,\bm\sigma)$.
\end{proof}

We remark that this correspondence does not depend on the choice of the intertwining element $\gamma$ of $\bm\eta_{max}$.

\subsection{Dependence on the choice of $\bm\kappa_{max}$}
In this paragraph we discuss the dependence of results of paragraphs \ref{subsec:categoryJlambda}, \ref{subsec:categoryequivalence} and \ref{subsec:correspondence} on the choice of the extension of $\bm\eta_{max}$ to $\textbf{J}_{max}$. 

\smallskip
Let $(\mathbf{J},\bm\lambda)$ be a semisimple supertype of $G$. 
We have just seen in remark \ref{rmk:corrTheta} that the group $\mathscr{G}$ depends only on 
$\bm\Theta(\mathbf{J},\bm\lambda)$ and by remark \ref{rmk:corrTheta} and theorem \ref{thm:blocks} the $\mathscr{G}$-conjugacy class of $\mathscr{M}$ and the category $\mathscr{R}(\mathbf{J},\bm\lambda)$ do not depend on the choice of the extension of $\bm\eta_{max}$ to $\textbf{J}_{max}$. 
Moreover, the sum (\ref{eq:SS1}) does not depend on this choice because a different one permutes the terms $V[\bm\Theta,\bm\sigma']$ in $V[\bm\Theta]$. 
Then $V[\bm\Theta]$, the equalities $\mathscr{R}(G,\bm\eta_{max})=\mathscr{R}_{\bm\eta_{max}}(G)=\bigoplus_{[\mathbf{J},\bm\lambda]\in\mathbf{X}}\mathscr{R}(\mathbf{J},\bm\lambda)$ and the equivalence of theorem \ref{thm:equivalenceM} do not depend on the choice of the extension of $\bm\eta_{max}$.

\smallskip
Let $\gamma$ be a fixed non-zero intertwining element of $\bm\eta_{max}$ as in remark \ref{rmk:isomHeckealgebras}.
Using notation of paragraph \ref{subsec:etamax}, let $\bm\kappa_{max}$ and $\bm\kappa'_{max}$ be two extensions of $\bm\eta_{max}$ to $\textbf{J}_{max}$ and let $\kappa_{max}=\bigotimes_{j=1}^l\kappa_{max,j}$ and $\kappa'_{max}=\bigotimes_{j=1}^l\kappa'_{max,j}$ be the restrictions to $J_{max}$ of $\bm\kappa_{max}$ and $\bm\kappa'_{max}$   respectively. 
Then, for every $j\in\{1,\dots,l\}$, $\kappa_{max,j}$ and $\kappa'_{max,j}$ are $\beta$-extensions of $\theta_{max,j}$  and so by paragraph \ref{subsec:Heisenberg} there exists a character $\chi_j$ of $\ent_{E^j}^{\times}$ trivial on $1+\wp_{E^j}$ such that $\kappa'_{max,j}=\kappa_{max,j}\otimes (\chi_j\circ N_{B^j/E^j})$.
Let $\chi$ and $\overline\chi$ be the character $\bigotimes_{j=1}^{l}(\chi_j\circ N_{B^j/E^j})$ viewed as characters of
$\textbf{J}_{max}$ trivial on $\textbf{J}^1_{max}$ and of $\mathscr{G}$ respectively and let $\widetilde\chi=\bigotimes_{j=1}^{l}\Big(\big(\infl_{\ent_{E^j}^{\times}}^{(E^j)^{\times}}\chi_j\big)\circ N_{B^j/E^j}\Big)$ viewed as a character of $B_L^{\times}$.

\smallskip
We consider the functors $\widetilde{\mathfrak{X}}:\mathscr{R}(B_L^{\times},K_L^1)\rightarrow \mathscr{R}(B_L^{\times},K_L^1)$ and $\overline{\mathfrak{X}}:\mathscr{R}_R(\mathscr{G})\rightarrow \mathscr{R}_R(\mathscr{G})$ given by 
$\widetilde{\mathfrak{X}}(\varrho)=\varrho\otimes\widetilde{\chi}^{-1}$, $\widetilde{\mathfrak{X}}(\widetilde\phi)=\widetilde\phi$, $\overline{\mathfrak{X}}(\tau)=\tau\otimes\overline{\chi}^{-1}$ 
and $\overline{\mathfrak{X}}(\overline\phi)=\overline\phi$ for every $\varrho$, $\varrho_1$ in $\mathscr{R}(B_L^{\times},K_L^1)$, every $\widetilde\phi\in\Hom_{B_L^{\times}}(\varrho,\varrho_1)$, every representations $\tau$ and $\tau_1$ of $\mathscr{G}$ and every $\overline{\phi}\in\Hom_{\mathscr{G}}(\tau,\tau_1)$.
We consider the following diagram.
\begin{equation}\label{diagram:dependence}
\begin{gathered}
\xymatrix{
\mathscr{R}(B_L^{\times},K_L^1)
\ar[rr]^{\textbf{\textsf{K}}_{K_L}}
\ar[dd]_{\widetilde{\mathfrak{X}}}
&&
\mathscr{R}_R(\mathscr{G})
\ar[dd]^{\overline{\mathfrak{X}}}
\\
&\mathscr{R}(G,\bm\eta_{max})
\ar[ul]_{\mathbf{F}_{\gamma,\bm\kappa_{max}}}
\ar[dl]_{\mathbf{F}_{\gamma,\bm\kappa'_{max}}}
\ar[dr]^{\textbf{\textsf{K}}_{\bm\kappa'_{max}}}
\ar[ur]^{\textbf{\textsf{K}}_{\bm\kappa_{max}}}
&
\\
\mathscr{R}(B_L^{\times},K_L^1)
\ar[rr]^{\textbf{\textsf{K}}_{K_L}}
&&
\mathscr{R}_R(\mathscr{G}).
}
\end{gathered}
\end{equation}

\begin{lemma}\label{lemma:naturalisom1}
We have $\textbf{\textsf{K}}_{\bm\kappa'_{max}}=\overline{\mathfrak{X}}\circ\textbf{\textsf{K}}_{\bm\kappa_{max}}$ and so for every representation $(\pi,V)$ in $\mathscr{R}(G,\bm\eta_{max})$ we have $\pi(\bm\kappa'_{max})=\pi(\bm\kappa_{max})\otimes \overline\chi^{-1}$.
\end{lemma}

\begin{proof}
The space of $\textbf{\textsf{K}}_{\bm\kappa'_{max}}(V)$ and of $\overline{\mathfrak{X}}(\textbf{\textsf{K}}_{\bm\kappa_{max}}(V))$ is 
$\Hom_{\mathbf{J}^1_{max}}(\bm\eta_{max},V)$.
Let $\varphi$ in this space and $x\in \mathbf{J}_{max}$.
Let $Q$ be the standard parabolic subgroup of $G$ with Levi component $L$, let $N$ be the unipotent radical of $Q$ such that $Q=LN$ and let $N^-$ be the unipotent radical opposite to $N$. 
We choose $x_1\in \mathbf{J}_{max}\cap N^-$, $x_2\in J_{max}$ and $x_3\in \mathbf{J}_{max}\cap N$ such that $x=x_1x_2x_3$.
Since $(\bm\kappa_{max},\mathbf{J}_{max})$ and $(\bm\kappa'_{max},\mathbf{J}_{max})$ are decomposed above $(\kappa_{max},J_{max})$ and $(\kappa'_{max},J_{max})$ respectively, we obtain
$\pi(\bm\kappa'_{max})(x)(\varphi)=\pi(x)\circ\varphi\circ \bm\kappa'_{max}(x^{-1})=\pi(x)\circ\varphi\circ \kappa'_{max}(x_2^{-1})=\pi(x)\circ\varphi\circ \kappa_{max}(x_2^{-1})\chi(x_2^{-1})=\pi(\bm\kappa_{max})(x)(\varphi)\chi(x_2)^{-1}$. Since $\mathbf{J}_{max}\cap N=\mathbf{J}^1_{max}\cap N$ and $\mathbf{J}_{max}\cap N^-=\mathbf{J}^1_{max}\cap N^-$ we obtain $\chi(x_2)^{-1}=\chi(x)^{-1}$.
Now, let $V_1$ and $V_2$ be two objects of $\mathscr{R}(G,\bm\eta_{max})$ and let $\phi\in\Hom_G(V_1,V_2)$. Then for every $\varphi\in\Hom_{\mathbf{J}^1_{max}}(\bm\eta_{max},V_1)$
we have $\textbf{\textsf{K}}_{\bm\kappa'_{max}}(\phi)(\varphi)=\phi\circ\varphi=\overline{\mathfrak{X}}(\textbf{\textsf{K}}_{\bm\kappa_{max}}(\phi))(\varphi)$.
\end{proof}

\begin{lemma}\label{lemma:naturalisom2}
We have $\textbf{\textsf{K}}_{K_L}\circ\widetilde{\mathfrak{X}}=\overline{\mathfrak{X}}\circ\textbf{\textsf{K}}_{K_L}$.
\end{lemma}

\begin{proof}
Let $(\varrho,Z)$ be in $\mathscr{R}(B_L^{\times},K_L^1)$. The space of $\textbf{\textsf{K}}_{K_L}(\widetilde{\mathfrak{X}}(Z))$ and of $\overline{\mathfrak{X}}(\textbf{\textsf{K}}_{K_L}(Z))$ is $Z^{K^1_L}$. 
Let $x\in K_L$ and let $\overline{x}$ the projection of $x$ in $K_L/K^1_L\cong\mathscr{G}$.
For every $z\in Z^{K^1_L}$ we have $\textbf{\textsf{K}}_{K_L}(\widetilde{\mathfrak{X}}(\varrho))(\overline{x})(z)=\widetilde{\chi}(x^{-1})\varrho(x)v$ while 
$\overline{\mathfrak{X}}(\textbf{\textsf{K}}_{K_L}(\varrho))(\overline{x})(z)=\overline{\chi}(\overline{x}^{-1})\varrho(x)v$.
Now, let $Z_1$ and $Z_2$ be two objects of $\mathscr{R}(B_L^{\times},K_L^1)$ and let $\phi\in\Hom_{B_L^{\times}}(Z_1,Z_2)$. 
Then we have 
$\textbf{\textsf{K}}_{K_L}(\widetilde{\mathfrak{X}}(\phi))=\phi_{|Z_1^{K^1_L}}=\overline{\mathfrak{X}}(\textbf{\textsf{K}}_{K_L}(\phi))$.
\end{proof}

We remark that by proposition \ref{thm:naturalisom}, 
lemma \ref{lemma:naturalisom1} and lemma \ref{lemma:naturalisom2}, the functor $\textbf{\textsf{K}}_{K_L}\circ \mathbf{F}_{\gamma,\bm\kappa'_{max}}$ is naturally isomorphic to $\textbf{\textsf{K}}_{\bm\kappa'_{max}}$ which 
is equal to $\overline{\mathfrak{X}}\circ\textbf{\textsf{K}}_{\bm\kappa_{max}}$ which is naturally isomorphic to 
$\overline{\mathfrak{X}}\circ\textbf{\textsf{K}}_{K_L}\circ \mathbf{F}_{\gamma,\bm\kappa_{max}}$ which is equal to
$\textbf{\textsf{K}}_{K_L}\circ \widetilde{\mathfrak{X}}\circ\mathbf{F}_{\gamma,\bm\kappa_{max}}$.

\begin{prop}
There exists a natural isomorphism between 
$\mathbf{F}_{\gamma,\bm\kappa'_{max}}$ and 
$\widetilde{\mathfrak{X}}\circ\mathbf{F}_{\gamma,\bm\kappa_{max}}$.
\end{prop}

\begin{proof}
For every object $(\pi,V)$ in $\mathscr{R}(G,\bm\eta_{max})$, the space of $\mathbf{F}_{\gamma,\bm\kappa'_{max}}(V)$ and of $\widetilde{\mathfrak{X}}(\mathbf{F}_{\gamma,\bm\kappa_{max}}(V))$ is $\mathbf{M}_{\bm\eta_{max}}(V)\otimes_{\mathscr{H}_R(B^{\times}_L,K^1_L)}\ind_{K^1_L}^{B_L^{\times}}(1_{K^1_L})$. 
If $m\in\mathbf{M}_{\bm\eta_{max}}(V)$ and $f\in\ind_{K^1_L}^{B_L^{\times}}(1_{K^1_L})$, in the first case the right action of $\Phi\in\mathscr{H}_R(B^{\times}_L,K^1_L)$ on $m$ and the left action of $x\in B_L^{\times}$ on $m\otimes f$ are given by $m\star'\Phi=m.\Theta_{\gamma,\bm\kappa'_{max}}(\Phi)$ and $x\diamond'(m\otimes f)=m\otimes x.f$ while in the second case they are given by $m\star\Phi=m.\Theta_{\gamma,\bm\kappa_{max}}(\Phi)$ and $x\diamond(m\otimes f)=\widetilde{\chi}(x^{-1})m\otimes x.f$. 
Let $\mathfrak{Z}_V$ be the $R$-automorphism of $\mathbf{M}_{\bm\eta_{max}}(V)\otimes_{\mathscr{H}_R(B^{\times}_L,K^1_L)}\ind_{K^1_L}^{B_L^{\times}}(1_{K^1_L})$ that maps $m\otimes f$ to $m\otimes \widetilde{\chi}f$ for every $m\in\mathbf{M}_{\bm\eta_{max}}(V)$ and $f\in\ind_{K^1_L}^{B_L^{\times}}(1_{K^1_L})$. 
By remark \ref{rmk:kappa'} we have $m\star'\Phi=m\star\widetilde{\chi}\Phi$ and then
$\mathfrak{Z}_V(m\star'\Phi\otimes f)=
(m\star'\Phi)\otimes(\widetilde{\chi} f)=
(m\star\widetilde{\chi}\Phi)\otimes(\widetilde{\chi} f)=
m\otimes((\widetilde{\chi}\Phi)*(\widetilde{\chi}f))=
m\otimes\widetilde{\chi}(\Phi*f)=\mathfrak{Z}_V(m\otimes (\Phi*f))$ and so $\mathfrak{Z}_V$ is well-defined.
Moreover, for every $x\in B_L^{\times}$ we have
$\mathfrak{Z}_V(x\diamond'(m\otimes f))=m\otimes\widetilde{\chi}(x.f)=\widetilde{\chi}(x^{-1})m\otimes x.(\widetilde{\chi}f)=x\diamond \mathfrak{Z}_V(m\otimes f)$ and so $\mathfrak{Z}_V$ is an isomorphism of representations of $B_L^{\times}$.
Now, let $V_1$ and $V_2$ be two objects of $\mathscr{R}(G,\bm\eta_{max})$ and let $\phi\in\Hom_G(V_1,V_2)$. Then for every $m\in\mathbf{M}_{\bm\eta_{max}}(V_1)$ and $f\in\ind_{K^1_L}^{B_L^{\times}}(1_{K^1_L})$ we have
$\mathfrak{Z}_{V_2}(\mathbf{F}_{\gamma,\bm\kappa'_{max}}(\phi)(m\otimes f))=
\mathfrak{Z}_{V_2}((\phi\circ m)\otimes f)
=(\phi\circ m)\otimes \widetilde\chi f
=\widetilde{\mathfrak{X}}(\mathbf{F}_{\gamma,\bm\kappa_{max}}(\phi))(m\otimes\widetilde\chi f)=
\widetilde{\mathfrak{X}}(\mathbf{F}_{\gamma,\bm\kappa_{max}}(\phi))(\mathfrak{Z}_{V_1}(m\otimes f))
$.
\end{proof}

By remark \ref{rmk:choix2}, the representations $\bm\kappa_{max}$ and $\bm\kappa'_{max}$ determine two decompositions $\bm\lambda=\bm\kappa\otimes \bm\sigma$ and $\bm\lambda=\bm\kappa'\otimes \bm\sigma'$ where $\bm\sigma$ 
and $\bm\sigma'$
are supercuspidal representations of $\mathscr{M}$
viewed as irreducible representations of $\mathbf{J}_L$ trivial on $\mathbf{J}^1_L$. 
Hence, the bijection $\phi_{\bm\kappa'_{max}}\circ\phi^{-1}_{\bm\kappa_{max}}$ permutes the elements of $\mathbf{Y}$ and it maps $[\mathscr{M},\bm\sigma]$ to $[\mathscr{M},\bm\sigma']$.
Let $\bm\kappa_L$ and $\bm\kappa'_L$ be the restrictions to $\mathbf{J}_L$ of $\bm\kappa$ and $\bm\kappa'$ respectively.
By (\ref{eq:kappamax}) and by (2.20) of \cite{MS} for every $j\in\{1,\dots,l\}$ we have 
$\bm\kappa'_L=\bm\kappa_L\otimes \chi$ and so 
$\bm\sigma'=\bm\sigma\otimes \overline{\chi}^{-1}$.

\bibliographystyle{alpha}
\bibliography{Chinello2}

\begin{thebibliography}{MS14b}

\bibitem[Ber84]{Bern}
Joseph Bernstein.
\newblock Le "centre" de {B}ernstein.
\newblock {\em Représentations des groupes réductifs sur un corps local,
  Travaux en cours, p. 1-32. Hermann, Paris}, 1984.
\newblock Rédigé par P. Deligne.

\bibitem[BH96]{BH}
Colin~J. Bushnell and Guy Henniart.
\newblock Local tame lifting for {$GL(N)$}. {I}: {S}imple characters.
\newblock {\em Publ. {M}ath. {I}nst. {H}autes {\'E}tud. {Sci}.}, 83:105--233,
  1996.

\bibitem[BK93]{BK1}
Colin~J. Bushnell and Philip~C. Kutzko.
\newblock {\em The admissible dual of {$GL(N)$} via compact open subgroups}.
\newblock Princeton University Press, 1993.

\bibitem[BK98]{BK2}
Colin~J. Bushnell and Philip~C. Kutzko.
\newblock Smooth representations of reductive $p$-adic groups: structure theory
  via types.
\newblock {\em Proc. London Math. Soc. (3)}, 77:582--634, 1998.

\bibitem[BK99]{BK3}
Colin~J. Bushnell and Philip~C. Kutzko.
\newblock Semisimple types in {$GL_n$}.
\newblock {\em Compositio Math.}, 119(1):53--97, 1999.
\newblock 

\bibitem[Blo05]{Blo}
Corinne Blondel.
\newblock Quelques propri\'et\'es des paires couvrantes.
\newblock {\em Math. Ann.}, 331(2):243--257, 2005.

\bibitem[BSS12]{BSS}
Paul Broussous, Vincent Sécherre, and Shaun Stevens.
\newblock Smooth representations of {$GL_m(D)$} {V}: {E}ndo-classes.
\newblock {\em Doc. Math.}, 17:23--77, 2012.

\bibitem[Chi15]{Chin}
Gianmarco Chinello.
\newblock {\em Représentations $\ell$-modulaires des groupes $p$-adiques.
  {D}écomposition en blocs de la catégorie des représentations lisses de
  {$GL_m(D)$}, groupe métaplectique et représentation de {W}eil}.
\newblock PhD thesis, Université de Versailles St-Quentin-en-Yvelines, 2015.

\bibitem[Chi17]{Chin1}
G.~Chinello.
\newblock Hecke algebra with respect to the pro-$p$-radical of a maximal
  compact open subgroup for {$GL(n,F)$} and its inner forms.
\newblock {\em {J}ournal of {A}lgebra}, 476:296--317, 2017.

\bibitem[Dat12]{Dat3}
Jean-Francois Dat.
\newblock Théorie de {L}ubin-{T}ate non-abélienne $\ell$-entière.
\newblock {\em Duke Math. J.}, 161(6):951--1010, 2012.
\newblock 

\bibitem[Dat16]{Dat4}
Jean-Francois Dat.
\newblock {E}quivalences of tame blocks for p-adic linear groups.
\newblock preprint, 2016.

\bibitem[Gui13]{Gui}
David-Alexandre Guiraud.
\newblock On semisimple $l$-modular {B}ernstein-blocks of a $p$-adic general
  linear group.
\newblock {\em J. Number Theory}, 133:3524--3548, 2013.

\bibitem[Hel16]{Helm}
David Helm.
\newblock The {B}ernstein center of the category of smooth
  {$W(k)[GL_n(F)]$}-modules.
\newblock {\em Forum Math. Sigma}, 4, 2016.

\bibitem[Kri90]{Krieg}
Aloys Krieg.
\newblock Hecke algebras.
\newblock {\em Mem. Amer. Math. Soc.}, 87(435):x+158, 1990.
\newblock 

\bibitem[MS14a]{MS2}
Alberto Minguez and Vincent Sécherre.
\newblock Repr\'esentations lisses modulo $\ell$ de {$GL_m(D)$}.
\newblock {\em Duke Math. Journal}, 163(4):795--887, 2014.

\bibitem[MS14b]{MS}
Alberto Minguez and Vincent Sécherre.
\newblock Types modulo $\ell$ pour les formes intérieures de {$GL_n$} sur un
  corps local non archimédien.
\newblock {\em Proc. London Math. Soc. (3)}, 109:823--891, 2014.

\bibitem[S{\'e}c04]{SecI}
Vincent S{\'e}cherre.
\newblock Repr\'esentations lisses de {$GL(m,D)$}, {I}. {C}aract\`eres simples.
\newblock {\em Bull. Soc. Math. France}, 132:327--396, 2004.

\bibitem[S{\'e}c05a]{SecII}
Vincent S{\'e}cherre.
\newblock Repr{\'e}sentations lisses de {$GL(m,D)$}, {II}. $\beta$-extensions.
\newblock {\em Composition Math.}, 141:1531--1550, 2005.

\bibitem[S{\'e}c05b]{SecIII}
Vincent S{\'e}cherre.
\newblock Repr{\'e}sentations lisses de {$GL(m,D)$}, {III}. types simples.
\newblock {\em Ann. Scient. Ec. Norm. Sup.}, 38:951--977, 2005.

\bibitem[SS08]{SeSt2}
Vincent S\'echerre and Shaun Stevens.
\newblock Repr\'esentations lisses de {$GL(m,D)$}, {IV}. {R}epr\'esentations
  supercuspidales.
\newblock {\em J. Inst. Math. Jussieu}, 7:527--574, 2008.

\bibitem[SS11]{SeSt3}
Vincent S\'echerre and Shaun Stevens.
\newblock Smooth representations of {$GL_m(D)$}, {VI}. {S}emisimple types.
\newblock {\em Int. Math. Res. Not. IMRN}, 13:2994--3039, 2011.
\newblock 

\bibitem[SS16]{SeSt1}
Vincent S\'echerre and Shaun Stevens.
\newblock Block decomposition of the category of $\ell$-modular smooth
  representations of {$GL_n(F)$} and its inner forms.
\newblock {\em Ann. Scient. Ec. Norm. Sup.}, 49:669--709, 2016.

\bibitem[Vig96]{Vig2}
Marie-France Vign{\'e}ras.
\newblock {\em Repr{\'e}sentations $l$-modulaires d'un groupe r{\'e}ductif
  $p$-adique avec $l\neq p$}, volume 137 of {\em Progress in Mathematics}.
\newblock Birkh{\"a}user Boston, 1996.

\bibitem[Vig98]{Vig1}
Marie-France Vign{\'e}ras.
\newblock Induced {$R$}-representations of $p$-adic reductive groups.
\newblock {\em Selecta Math.}, 4:549--623, 1998.

\end{thebibliography}

\end{document}